\documentclass[a4paper]{amsart}

\usepackage{amsthm,amssymb}
\usepackage{mathtools,thmtools}
\usepackage{bm,esint}
\usepackage{color}
\usepackage{float,graphicx,subcaption}
\usepackage{cleveref}
\usepackage{tikz}
\usetikzlibrary{arrows.meta}
\usetikzlibrary{calc}

\newif\ifvorticity 
\vorticityfalse

\title[Statistical Solutions of the incompressible Euler equations]{
  Statistical Solutions of the \\ incompressible Euler equations
}
\author{S.~Lanthaler, \; S.~Mishra, \and C.~Par\'es-Pulido}
\address[S. Lanthaler, S. Mishra and C. Par\'es-Pulido]{Seminar for Applied Mathematics, ETH Z\"urich, R\"amistrasse 101, 8092 Z\"urich, Switzerland}

\newcommand{\explain}[2]{\overset{\mathclap{\underset{\downarrow}{#2}}}{#1}}
\newcommand{\supp}{\mathrm{supp}}
\newcommand{\defeq}{{:=}}
\renewcommand{\d}{{\mathrm{d}}}
\newcommand{\intF}{{\int_{L^p}}}

\newcommand{\weaklyto}{{\rightharpoonup}}

\newcommand{\T}{\mathbb{T}}
\newcommand{\R}{\mathbb{R}}

\newcommand{\N}{\mathbb{N}}
\renewcommand{\div}{{\mathrm{div}}}
\newcommand{\curl}{{\mathrm{curl}}}

\renewcommand{\P}{{\mathcal{P}}}

\newcommand{\define}{\textbf}
\newcommand{\Prob}{\mathcal{P}}
\newcommand{\U}{U}
\newcommand{\ip}[2]{{\langle #1, #2 \rangle}}
\newcommand{\intavg}{\fint}
\newcommand{\Lp}{{\mathfrak{L}}}

\declaretheoremstyle[
  headfont=\normalfont\bfseries\itshape,
  numbered=unless unique,
  bodyfont=\normalfont,
  spaceabove=1em plus 0.75em minus 0.25em,
  spacebelow=1em plus 0.75em minus 0.25em,
  qed={},
]{deflt}

\theoremstyle{deflt}
\newtheorem{theorem}{Theorem}[section]
\newtheorem{assumption}[theorem]{Assumption}

\newtheorem{remark}[theorem]{Remark}
\newtheorem{algorithm}[theorem]{Algorithm}
\newtheorem{definition}[theorem]{Definition}
\newtheorem{lemma}[theorem]{Lemma}
\newtheorem{proposition}[theorem]{Proposition}
\newtheorem{corollary}[theorem]{Corollary}

\newcommand{\embeds}{{\hookrightarrow}}
\newcommand{\embedsc}{\overset{comp}{\hookrightarrow}}

\makeatletter
\def\grd@save@target#1{%
  \def\grd@target{#1}}
\def\grd@save@start#1{%
  \def\grd@start{#1}}
\tikzset{
  grid with coordinates/.style={
    to path={%
      \pgfextra{%
        \edef\grd@@target{(\tikztotarget)}%
        \tikz@scan@one@point\grd@save@target\grd@@target\relax
        \edef\grd@@start{(\tikztostart)}%
        \tikz@scan@one@point\grd@save@start\grd@@start\relax
        \draw[minor help lines] (\tikztostart) grid (\tikztotarget);
        \draw[major help lines] (\tikztostart) grid (\tikztotarget);
        \grd@start
        \pgfmathsetmacro{\grd@xa}{\the\pgf@x/1cm}
        \pgfmathsetmacro{\grd@ya}{\the\pgf@y/1cm}
        \grd@target
        \pgfmathsetmacro{\grd@xb}{\the\pgf@x/1cm}
        \pgfmathsetmacro{\grd@yb}{\the\pgf@y/1cm}
        \pgfmathsetmacro{\grd@xc}{\grd@xa + \pgfkeysvalueof{/tikz/grid with coordinates/major step}}
        \pgfmathsetmacro{\grd@yc}{\grd@ya + \pgfkeysvalueof{/tikz/grid with coordinates/major step}}
        \foreach \x in {\grd@xa,\grd@xc,...,\grd@xb}
        \node[anchor=north] at (\x,\grd@ya) {\pgfmathprintnumber{\x}};
        \foreach \y in {\grd@ya,\grd@yc,...,\grd@yb}
        \node[anchor=east] at (\grd@xa,\y) {\pgfmathprintnumber{\y}};
      }
    }
  },
  minor help lines/.style={
    help lines,
    step=\pgfkeysvalueof{/tikz/grid with coordinates/minor step}
  },
  major help lines/.style={
    help lines,
    line width=\pgfkeysvalueof{/tikz/grid with coordinates/major line width},
    step=\pgfkeysvalueof{/tikz/grid with coordinates/major step}
  },
  grid with coordinates/.cd,
  minor step/.initial=.2,
  major step/.initial=1,
  major line width/.initial=2pt,
}
\makeatother

\newcommand\xqed[1]{%
  \leavevmode\unskip\penalty9999 \hbox{}\nobreak\hfill
  \quad\hbox{#1}}

\begin{document}
\maketitle
\begin{abstract}
We propose and study the framework of dissipative statistical solutions for the incompressible Euler equations. Statistical solutions are time-parameterized probability measures on the space of square-integrable functions, whose time-evolution is determined from the underlying Euler equations. We prove partial well-posedness results for dissipative statistical solutions and propose a Monte Carlo type algorithm, based on spectral viscosity spatial discretizations, to approximate them. Under verifiable hypotheses on the computations, we prove that the approximations converge to a statistical solution in a suitable topology. In particular, multi-point statistical quantities of interest converge on increasing resolution. We present several numerical experiments to illustrate the theory. 
\end{abstract}
\section{Introduction} \label{sec:intro}
Many interesting incompressible fluid flows are characterized by high to very high values of the \emph{Reynolds number}. Taking the infinite Reynolds number limit of the underlying Navier-Stokes equations (formally), results in the \emph{Incompressible Euler} equations governing the motion of an ideal i.e. inviscid and incompressible fluid, given by
\begin{gather} \label{eq:Eulerfull}
\left\{
\begin{aligned}
\partial_t {u} 
+{u}\cdot \nabla {u} 
+ \nabla p
&=
0, 
\\
\div({u}) 
&= 
0, 
\\
{u}|_{t=0} 
&=
\overline{u}. 
\end{aligned}
\right.
\end{gather}
Here, the velocity field is denoted by ${u} \in \R^d$ (for $d=2,3$), and the pressure is denoted by $p \in \R_+$. The pressure acts as a Lagrange multiplier to enforce the divergence-free constraint \cite{majda2001}. The equations need to be supplemented with
suitable boundary conditions. Throughout this work, we shall assume periodic boundary conditions, and take as our domain $D$, the $d$-dimensional torus $D=\mathbb{T}^d$, $d\in \{2,3\}$.
\subsection{Well-posedness results}
The question of global (in time) well-posedness of classical i.e. $C^1$ solutions of the incompressible Euler equations \eqref{eq:Eulerfull} in three space dimensions, even with sufficiently smooth initial data $\bar{u}$, is not yet resolved. Moreover in two space dimensions, where one can prove well-posedness of classical solutions as long as the initial data is sufficiently regular, many interesting initial data of interest, such as vortex sheets, do not possess this regularity. Hence, it is imperative to consider the so-called \emph{weak solutions} of \eqref{eq:Eulerfull}, defined as,
\begin{definition} \label{def:Eulerweak}
A vector field ${u}\in L^\infty([0,T);L^2(D;\R^d))$ is a weak solution of the incompressible Euler equations with initial data $\overline{u} \in L^2(D;\R^d)$, if
\begin{gather}\label{eq:Eulerweak}
\int_0^T\int_{D}
{u}\cdot \partial_t {\phi} 
+ ({u}\otimes {u}):\nabla {\phi} \, dx\, dt
= -\int_{D} \overline{u}\cdot {\phi}(x,0) \, dx,
\end{gather}
for all test vector fields, ${\phi}\in C_c^\infty([0,T)\times D;\R^d)$, $\div({\phi})=0$, and 
\begin{gather}\label{eq:incomprweak}
\int_{D} {u}\cdot \nabla \psi \, dx = 0,
\end{gather}
for all test functions $\psi\in C^\infty(D)$.
\end{definition}
It is customary to require additional admissibility criteria in order to recover uniqueness of weak solutions. A natural criterion in this context is given by the so-called dissipative or admissible weak solutions, i.e. weak solutions $u$ such that $\Vert {u}(t) \Vert_{L^2} \le \Vert \overline{u} \Vert_{L^2}$. Although the global existence of admissible weak solutions in three space dimensions is open, one can prove global existence of admissible weak solutions in two dimensions with very general initial data. The most general result of this kind is the celebrated work of Delort \cite{Delort1991}, where he proved global existence of weak solutions of \eqref{eq:Eulerfull} as long as the initial data was in the so-called \emph{Delort class}, i.e. with initial vorticity being the sum of a signed Radon measure (in space) and a $p$-integrable function. On the other hand, it is well-known that admissible weak solutions, in both two and three space dimensions are not unique \cite{Sch,scn,DS1} and references therein. 
\subsection{Numerical schemes.} A variety of numerical methods have been proposed in order to approximate the incompressible Euler equations \eqref{eq:Eulerfull}. These include the spectral and spectral viscosity methods \cite{Tadmor1989}, that are particularly suitable for periodic boundary conditions. On bounded domains, efficient methods such as the finite difference projection method \cite{Chorin,BCG} and discontinuous Galerkin (DG) finite element method \cite{Guzshu} have been proposed. Alternatives include numerically approximating the Euler equations in the vorticity-stream function formulation. Methods such as the Lagrangian vortex blob \cite{LiuXin2001}, point vortex \cite{Kras1} and central finite difference schemes \cite{Levy1997} have been proposed in this context.  

Rigorous convergence results for numerical approximations of the incompressible Euler equations \eqref{eq:Eulerfull} are mostly restricted to two space dimensions and to sufficiently regular initial data, see \cite{BardosTadmor} for spectral viscosity methods, \cite{majda2001} for vortex methods and \cite{Guzshu} for DG methods. Notable exceptions include \cite{TadmorNussenzveig} where the central schemes of \cite{Levy1997} were shown to converge to weak solutions of the two-dimensional Euler equations as long as the initial vorticity was in $L^p$, for $p>1$. Similarly in \cite{LiuXin2001,LiuXin1995}, the authors showed convergence of vortex point and vortex blob methods in two space dimensions as long as the initial vorticity was a signed Radon measure. Finally in a recent paper \cite{LM2019}, the authors showed convergence of a spectral viscosity method for initial data in Delort class. 

Furthermore, careful numerical experiments, for instance those presented in \cite{LM2015,FMT16} have shown that numerical schemes may not necessarily converge, even in two space dimensions, for rough initial data. Even when the numerical approximations converge (such as in \cite{LM2019}), the inherent instabilities of the Euler equations lead to a very slow convergence rate (see also Figure \ref{fig:svs_conv} (Left)) and render the computation of weak solutions of \eqref{eq:Eulerfull} prohibitively expensive. 
\subsection{Measure-valued and Statistical solutions.} 
Given the lack of well-posedness results for weak solutions and the lack of convergent numerical approximations, there is considerable scope for the design of alternative solution frameworks for \eqref{eq:Eulerfull}. One such framework is that of measure-valued solutions \cite{DipernaMajda}, where the sought for solutions are no longer functions but space-time parametrized probability measures on state space. The global existence of measure-valued solutions, even in three space dimensions, was shown in \cite{DipernaMajda}. A convergent numerical method (of the spectral viscosity) type and an efficient algorithm to compute measure-valued solutions was proposed in \cite{LM2015}. 

However, measure-valued solutions are generically non-unique. This holds true even for the much simpler case of the one-dimensional Burgers' equation \cite{Schochet1990}. In \cite{FLM17}, the authors implicated the lack of information about multi-point (spatial) correlations in the non-uniqueness of measure-valued solutions. Moreover, they also proposed a framework of \emph{statistical solutions} as an attempt to recover uniqueness.

In the formulation of \cite{FLM17}, statistical solutions are time-parameterized probability measures on $L^p$, for $1 \leq p < \infty$, that are consistent with the underlying PDE in a weak sense. They were shown to be equivalent to a family of \emph{correlation measures}, with the $k$-th member of this family is a Young measure representing correlations (or joint probabilities) of the solution at $k$ distinct spatial points. Thus, one can interpret statistical solutions as measure-valued solutions, augmented with information about all possible multi-point correlations. A priori, statistical solutions contain much more information than a measure-valued solution. Moreover, statistical solutions encode statistical (ensemble averaged) properties of the solutions of the underlying PDE. Thus, statistical solutions provide a suitable framework for uncertainty quantification (UQ) \cite{FLM17,AM2018}. This is particularly relevant for the incompressible Euler equations as it is well-known that the flow of fluids, at very high-Reynolds numbers, can be turbulent and only averaged (or statistical) properties can be inferred from measurements \cite{Frisch}. 

Statistical solutions for scalar conservation laws were considered in \cite{FLM17}, wherein well-posedness was shown under an entropy condition. In \cite{FLYM,FLMW1}, a Monte Carlo algorithm, based on the ensemble averaging algorithm of \cite{FKMT17}, was proposed and analyzed for scalar conservation laws and multi-dimensional hyperbolic systems of conservation laws, respectively. 

Independent notions of statistical solutions of the incompressible Navier-Stokes equations have been proposed in \cite{FoiasRosaTemam2010} and in \cite{VF1}. The zero-viscosity limit of these statistical solutions of Navier-Stokes equations are considered in a forthcoming paper \cite{FMW1}. \subsection{Aims and scope of the current article.} The main goal of this article is to propose a suitable notion of statistical solutions for the incompressible Euler equations \eqref{eq:Eulerfull} and to analyze and approximate these statistical solutions. To this end, we will
\begin{itemize}
    \item propose a notion of dissipative statistical solutions, i.e. a time-parameterized probability measure on $L^p(D)$ that is consistent with \eqref{eq:Eulerfull} in a suitable sense and prove well-posedness in some special cases, namely local in time well-posedness and global well-posedness in two space dimensions for sufficiently regular initial data,
    \item propose a numerical algorithm, based on ensemble averaging and a spectral viscosity spatial discretization, to approximate statistical solutions of \eqref{eq:Eulerfull} and prove that the approximations converge in an appropriate topology to a statistical solution, under reasonable and verifiable hypotheses on the numerical method,
    \item perform and present extensive numerical experiments to verify the theory and to illustrate interesting properties of statistical solutions.

\end{itemize}

The rest of the paper is organized as follows: in section \ref{sec:timecompact}, we present time-parameterized probability measures on $L^2(D;\mathbb{R}^d)$ and characterize convergence in a suitable topology on this space of measures. In section \ref{sec:weakstrong}, we define statistical solutions of \eqref{eq:Eulerfull} and present well-posedness results. The numerical approximation of statistical solutions and its convergence is presented in section \ref{sec:spectralvisc} and numerical experiments are presented in section \ref{sec:numerical}.

\section{Time-parameterized probability measures on $L^2(D;\mathbb{R}^d)$} \label{sec:timecompact}
As mentioned in the introduction, statistical solutions are time-parameterized probability measures on $L^p$. For incompressible Euler equations, the energy bound enforces that $p=2$. In this section, we will describe time-parameterized probability measures, characterize them and describe a suitable topology on them.

\subsection{Notation}
For spatial dimensions $d=2,3$, we denote the spatial domain $D = \mathbb{T}^d$, i.e. $d$-dimensional torus $[-\pi,\pi]^d$, with endpoints identified, the co-domain (or phase space) $U = \mathbb{R}^d$, so that a solution of \eqref{eq:Eulerfull} is given as a vector field $u: D \times [0,T) \to U$. On occasion $x\mapsto u(x,t)$ will also be identified with a $2\pi$-periodic function on $\mathbb{R}^d$ in the obvious way.

As we are interested in time-dependent vector fields $u(x,t)$, we will fix a final time $T>0$ and consider $u(x,t)$ as a mapping $u: [0,T) \to L^2_x=L^2(D;U)$. In general, for a Banach space $Y$, we will denote by $L^q_t(Y)$ the space of measurable functions $u: [0,T) \to Y$, such that 
\[
\Vert u \Vert_{L^q_t(Y)} = \left(\int_0^T \Vert u \Vert_{Y}^q \, dt \right)^{1/q} < \infty.
\]
Of particular interest in the context of the incompressible Euler equations are spaces $L^q_tL^2_x = L^q_t(L^2_x)$, for $1\le q \le \infty$.

For a function $u\in L^2_x$, we will denote by $\widehat{u}(k)$ the $k$-th Fourier coefficient ($k\in \mathbb{Z}^d$), i.e.
\[
\widehat{u}(k) = \frac{1}{(2\pi)^d}\int_D u(x) e^{-ikx} \, dx
\quad
\Rightarrow
\quad
u(x) = \sum_{k \in \mathbb{Z}^d} \widehat{u}(k) e^{ikx}.
\]

We will say that $\omega(r)$ is a \define{modulus of continuity} if $\omega: [0,\infty) \to [0,\infty)$ is a map such that $\lim_{r\to 0} \omega(r) = 0$.

On multiple occasions in this paper, we will need to mollify a given function $x \mapsto u(x)$. Let us therefore fix a standard mollifier $\rho(x) \in C^\infty(\mathbb{R}^d)$, supported in a ball of radius $1$ around the origin, i.e. $\rho \ge 0$ is a function, such that
\[
\int_{\mathbb{R}^d} \rho(x) \, dx = 1, \quad \rho(x) = 0, \text{ if } |x|>1.
\]
Given $\epsilon > 0$, we denote by $\rho_\epsilon \in C_c^\infty(\mathbb{R}^d)$ the function $\rho_\epsilon(x) = \epsilon^{-d} \rho(x/\epsilon)$. The \define{mollification} of $u(x)$ is now defined by convolution with $\rho_\epsilon(x)$, as 
\[
u_\epsilon(x) := (u\ast\rho_\epsilon)(x) = \int_{\mathbb{R}^d} u(x-y) \rho_\epsilon(y) \, dy.
\]
\subsection{Weak convergence of probability measures}

Given a topological space $X$, we denote by $\mathcal{P}(X)$ the space of all probability measures on the Borel $\sigma$-algebra of $X$. Given a sequence $\mu_n \in \P(X)$ ($n\in \mathbb{N}$) and $\mu \in \P(X)$, we say that $\mu_n$ \define{converges weakly} to $\mu$, written $\mu_n \weaklyto \mu$, if 
\begin{align}
\int_X F(u) \, d\mu_n(u) \to \int_X F(u) \, d\mu(u),
\end{align}
for all $F\in C_b(X)$, where $C_b(X)$ is the space of bounded, continuous functions on $X$. 

We shall call a family of probability measures $\{\mu^\Delta\}_{\Delta>0} \subset \P(X)$ defined on a separable Banach space $X$ \define{tight}, provided that for any $\epsilon > 0$ there exists a compact subset $K\subset X$ such that 
\[
\mu^\Delta(K) \ge 1 - \epsilon, \quad \forall \, \Delta > 0.
\]
It is a classical result due to Prokhorov (see e.g. theorems 8.6.7, 8.6.8 of the monograph \cite{Bogachev}) that a family $\mu^\Delta \in \P(X)$, with $X$ a separable Banach space, is tight if and only if $\mu^\Delta$ is relatively compact under the weak topology.

We recall that weak convergence on a Banach space $X$ is metrized by the \textbf{$p$-Wasserstein metric} $W_p$:
\begin{align} \label{eq:wasserstein}
W_p(\mu,\nu) 
&:=
\left(
\inf_{\pi \in \Pi(\mu,\nu)} 
\int_{X\times X} \Vert u - v \Vert_X^p \, d \pi(u,v)
\right)^{1/p},
\end{align}
where the infimum is taken over all \textbf{transfer plans} $\pi \in \Pi$ defined as,
$$
\Pi(\mu,\nu):= \left\{\pi \in \P(X\times X): \int_{X \times X} (F(u) + G(v)) d\pi(u,v) = \int_X F(u) d\mu(u) + \int_X G(v) d\nu(v)\right\},
$$
for all measurable $F,G$. More precisely, let $\mu^\Delta$ be a sequence of probability measures on $X$. If there exists a $M>0$, such that $\mu^\Delta \in \P(X)$ are uniformly concentrated on the ball $B_M(0) = \{u \in X \, | \, \Vert u \Vert_{X} < M\}$ of radius $M$ and centered at the origin in $X$, then 
\[
\mu^\Delta \weaklyto \mu 
\;
\Leftrightarrow
\;
W_p(\mu^\Delta, \mu) \to 0,
\]
where $1\le p < \infty$. In the present work, we shall only consider $p=1$ and $p=2$.

We also note that for $p=1$, the following \textbf{duality formula} holds:
\begin{align} \label{eq:W1duality}
W_1(\mu,\nu) 
=
\sup_{\Phi} \int_{X} \Phi(u) \, [d\mu(u) - d\nu(u)],
\end{align}
where the supremum is taken over all $1$-Lipschitz continuous functions $\Phi: X \to \mathbb{R}$.

Next, we characterize the compactness properties of families of probability measures $\mu^\Delta \in \P(L^2_x)$. The following quantity, the so-called \emph{structure function}, is of particular interest in the present context:
\[
S^2_r(\mu) := 
\left(
\int_{L^2_x} \int_D \fint_{B_r(0)} 
|u(x+h) - u(x)|^2 
\, dh \, dx \, d\mu(u)
\right)^{1/2},
\]
where $B_r(0) = \{ \xi \in \mathbb{R}^d \, |\, |\xi|< r\}$.

We first state the following uniform approximation principle, whose proof is an exercise in general topology. 
\begin{lemma} \label{lem:cpctapprox}
Let $(X,d)$ be a complete metric space with metric $d$. Let $K\subset X$. If for any $\epsilon > 0$, there exists a mapping $i_\epsilon: K \to X$, such that the image $i_\epsilon(K) \subset X$ is precompact (i.e. has compact closure), and $d(x,i_\epsilon(x)) < \epsilon$ for all $x\in K$, then $K$ is precompact.
\end{lemma}

We now state our main result characterizing certain compact subsets of $\P(L^2_x)$.

\begin{theorem} \label{thm:compactnessL2}
Let $\mathcal{F} \subset \P(L^2_x)$ be a family of probability measures on $L^2_x$. Assume that there exists $M>0$, such that $\mu(B_M(0))=1$ for all $\mu \in \mathcal{F}$, where $B_M(0) = \{u\in L^2_x\, | \, \Vert u \Vert_{L^2_x} < M\}$. Then the following statements are equivalent:
\begin{enumerate}
\item $\mathcal{F}\subset L^2_x$ has compact closure (with respect to the weak topology),
\item There exists a modulus of continuity $\omega$, such that we have a uniform bound on the structure function:
\[
S^2_r(\mu)
\le 
\omega(r), \quad \forall \, \mu \in \mathcal{F}.
\]
\end{enumerate}
\end{theorem}
\begin{proof}
We shall only use the implication (2) $\implies$ (1) in the present work. We prove this result here. For a proof of the converse, (1) $\implies$ (2), we refer instead to \cite{LanthalerThesis}.

\textbf{Step 1:}
Fix a mollifier $\rho_\epsilon$ for $\epsilon > 0$. 
Define $i_\epsilon: \P(L^2_x) \to \P(L^2_x)$ by mollification against $\rho_\epsilon$, 
\[
u_\epsilon(x) \defeq (i_\epsilon u)(x) \defeq \int_{D} \rho_\epsilon(y) u(x-y) \, \d y.
\]
Note that for any fixed $\epsilon>0$, the image $i_\epsilon(B_M(0))$ is precompact in $L^2_x$. Let us denote by $i_\epsilon\#: \P(L^2_x) \to \P(L^2_x)$ the push-forward mapping associated with the mollification map $i_\epsilon$. Note that, since all of the probability measures $\mu \in \mathcal{F}$ are supported on $\overline{B_M(0)} \subset L^2_x$, the corresponding push-forward measures $\mu^\epsilon \defeq i_\epsilon \# \mu$ have support in the (compact) set $\overline{i_\epsilon(B_M(0))} \subset L^2_x$. In particular, this implies that, for fixed $\epsilon > 0$, the family $\mathcal{F}_\epsilon := \{\mu^\epsilon; \mu \in \mathcal{F}\}$ is tight. By Prokhorov's theorem, this implies that $\mathcal{F}_\epsilon$ is precompact in the weak topology.

\textbf{Step 2:}
We claim that there exists a constant $C>0$ (independent of $\mathcal{F}$ and $\epsilon > 0$), such that 
\begin{align} \label{eq:W1est}
W_1(\mu^\epsilon,\mu)\le C\omega(\epsilon),
\end{align}
for all $\mu \in \mathcal{F}$. Here we recall that $\omega(r)$ is the uniform modulus of continuity which exists by assumption on $\mathcal{F}$.

 To see \eqref{eq:W1est}, we note that given any Lipschitz continuous function $F: L^2_x \to \R$ with $\Vert F\Vert_{\mathrm{Lip}}\le 1$, and $\mu \in \mathcal{F}$, we have
\begin{align*}
\int_{L^2_x} F(u) \, \d\mu^\epsilon - \int_{L^2_x} F(u)\, \d\mu
&= \int_{L^2_x} \left[F(u_\epsilon) - F(u)\right] \, \d\mu 
\\ &\explain{\le }{(\Vert F \Vert_{\mathrm{Lip}}\le 1)}
\int_{L^2_x}  \Vert u_\epsilon - u \Vert_{L^2_x} \, \d\mu
\\ &\explain{\le }{(\text{Jensen})}
\left(
\int_{L^2_x}  \Vert u_\epsilon - u \Vert^2_{L^2_x} \, \d\mu
\right)^{1/2}
\\ &=
\left(
\int_{L^2_x} \int_D  \vert u_\epsilon(x) - u(x) \vert^2 \, \d x\, \d\mu
\right)^{1/2}.
\end{align*}
The last integral can be bounded using
\begin{align*}
\vert u_\epsilon(x) - u(x) \vert^2
&\le
\int_{B_{\epsilon}(0)} \rho_\epsilon(h) |u(x+h)-u(x)|^2 \, \d h
\\ &\le 
C^2\fint_{B_\epsilon(0)} \vert u(x+h) - u(x) \vert^2 \, \d h,
\end{align*}
where $C^2 = \Vert \rho \Vert_{L^\infty}$ depends only on the (fixed) choice of mollifier $\rho$. Thus,
\begin{align} \label{eq:mollest}
\int_D |u_\epsilon - u|^2 \, dx
\le 
C^2 \int_D \fint_{B_\epsilon(0)} |u(x+h) - u(x)|^2 \, dh \, dx.
\end{align}
It follows from the uniform modulus of continuity assumption on $\mathcal{F}$ that
\begin{align*}
\intF F(u) \d \mu^\epsilon - \intF F(u) \d \mu
\le C\omega(\epsilon),
\end{align*}
for all $F: L^2_x \to \R$, with $\Vert F\Vert_{\mathrm{Lip}}\le 1$. Taking the supremum over all such $F$ on the left and recalling the duality formula for the $1$-Wasserstein distance $W_1$, eq. \eqref{eq:W1duality}, we recover the claimed estimate \eqref{eq:W1est}.

\textbf{Step 3:}
To conclude the proof of this theorem, we note that by Steps 1 and 2, $\mathcal{F}$ is uniformly approximated by precompact sets in the sense of lemma \ref{lem:cpctapprox} (with $X = \P(L^2_x)$ under the $1$-Wasserstein distance, $K=\mathcal{F}$ and $i_\epsilon\#: K \to X$ the pushforward induced by mollification), and hence is itself precompact.
\end{proof}
\subsection{Time parameterized probability measures}
As mentioned before, statistical solutions are time-parameterized probability measures. We define them below,
\begin{definition}
We denote by $L^1_t(\P) = L^1([0,T);\P)$ the space of weak-$\ast$ measurable mappings $[0,T) \to \P(L^2_x)$, namely mappings $t\mapsto \mu_t$ such $t \mapsto \int_{L^2_x} F(u) d\mu_t(u)$ is measurable for a.e $t \in [0,T)$, for all $F \in C_b(L^2_x)$ and with the property that
\[
\int_0^T \int_{L^2_x} \Vert u \Vert_{L^2_x} \, d\mu_t(u) \, dt < \infty.
\]

Denoting by $\delta_0$ the Dirac measure concentrated on $0\in L^2_x$, the above condition can equivalently be written as
\[
\int_0^T W_1(\delta_0, \mu_t) \, dt < \infty.
\]
This leads us to define a natural metric on $L^1([0,T);\P)$ by
\begin{equation}
    \label{eq:dtdef}
d_T(\mu_t,\nu_t) 
:= 
\int_0^T W_1(\mu_t, \nu_t) \, dt.
\end{equation}
\end{definition}
We then have the following proposition, whose proof is presented in appendix \ref{app:completeness}.

\begin{proposition} \label{prop:completeness}
The metric space $(L^1_t(\P),d_T)$ is a complete metric space.
\end{proposition}

Our objective is to find a suitable topology on $L^1_t(\P)$ and characterize compactness in this topology. It would be natural to try and extend the compactness theorem \ref{thm:compactnessL2} to time-parameterized probability measures and find a suitable version of the weak topology. This necessitates formalizing some notion of time-continuity or time-regularity of underlying functions. 

To this end,  fix a (time-independent) divergence-free test function $\phi \in C^\infty_c(D;\mathbb{R}^d)$. Formally, solutions of the incompressible Euler equations \eqref{eq:Eulerfull} satisfy for $s,t\in [0,T)$,
\[
\int_D \left[u(x,t)-u(x,s)\right] \phi(x) \, dx
=
\int_s^t \int_{D} u(x,\tau) \otimes u(x,\tau) : \nabla \phi(x) \, dx \, d\tau,
\]
so that 
\[
\left|
\int_D \left[u(x,t)-u(x,s)\right] \phi(x) \, dx
\right|
\le 
C\Vert u \Vert_{L_t^\infty L^2_x} \Vert \nabla \phi \Vert_{L^\infty_{x}} |t-s|.
\]
Furthermore, we have a natural energy bound $\Vert u \Vert_{L_t^\infty L^2_x} \le \Vert u_0 \Vert_{L^2_x}$, in terms of the initial data $u_0$. If $L>0$ is large enough such that by Sobolev embedding $H^{L}_x = H^L( D;U) \embeds C^1(D;U)$, then it follows that
\[
\left|
\int_D \left[u(x,t)-u(x,s)\right] \phi(x) \, dx
\right|
\le 
C \Vert \phi \Vert_{H^L_x} |t-s|,
\]
where the constant $C>0$ depends only on the initial data. Taking the supremum over all $\phi \in H^L_x$ with $\Vert \phi \Vert_{H^L_x} \le 1$, it follows, at least formally, that
\begin{equation} \label{eq:formaltimereg}
\Vert u(t) - u(s) \Vert_{H^{-L}_x} \le C |t-s|, \quad \forall \, s,t \in [0,T).
\end{equation}

Given these considerations, it is natural to assume that statistical solutions of the Euler equations satisfy some version of this time continuity. A formalization is provided in the following definition,
\begin{definition} \label{def:timereg}
A weak-$\ast$ measurable, time-parametrized probability measure $t \mapsto \mu_t \in \P(L^2_x)$ is called \define{time-regular}, if there exists a constant $L>0$, and a mapping $s,t \mapsto \pi_{s,t} \in \P(L^2_x \times L^2_x)$, such that for almost all $s,t\in [0,T)$:
\begin{itemize}
\item The measure $\pi_{s,t}$ is a transport plan from $\mu_s$ to $\mu_t$,
\item There exists a constant $C>0$, such that $\pi_{s,t}$ satisfies the following regularity condition
\[
\int_{L^2_x \times L^2_x} \Vert u - v \Vert_{H^{-L}_x} \, d\pi_{s,t}(u,v) \le C |t - s|.
\]
\end{itemize}

A family $\{\mu_t^\Delta\}_{\Delta>0}$ of time-parametrized probability measures is \define{uniformly time-regular}, provided that each $\mu^\Delta_t$ is time-regular, and the constants $L,C>0$ above can be chosen independently of $\Delta>0$.

\end{definition}

\begin{remark}
Note that if $\mu_{t}$ is of the form 
\[
\mu_{t} = \frac{1}{M} \sum_{i=1}^M \delta_{u_i(t)},
\]
with $t\mapsto u(t)$ weak solutions of the incompressible Euler equations satisfying \eqref{eq:formaltimereg}, then we can define suitable transfer plans
\[
\pi_{s,t} = \frac{1}{M} \sum_{i=1}^M \delta_{u_i(s)} \otimes \delta_{u_i(t)}.
\]
The time-regularity property follows from the estimate \eqref{eq:formaltimereg} for the $u_i$.
\end{remark}

We now show that a family $\mu^\Delta_t$, $\Delta > 0$, of uniformly time-regular probability measures is relatively compact, provided that they satisfy a time-averaged version of the second property of theorem \ref{thm:compactnessL2}. To this end, we define the time-averaged structure function of  $(t\mapsto \mu_t)\in L^1_t(\P)$ (weak-$\ast$ measurable) as the following quantity:
\begin{equation}
    \label{eq:tasf}
S^2_r(\mu_t,T) := 
\left(
\int_0^T 
\int_{L^2_x} 
\int_D 
\fint_{B_r(0)} |u(x+h)-u(x)|^2 
\, dh \, dx \, d\mu_t(u) \, dt
\right)^{1/2}.
\end{equation}

We first state the following lemma:
\begin{lemma} \label{lem:tapprox}
There exists $C>0$, such that if $\mu_t \in L^1_t(\P)$ has structure function $S^2_r(\mu_t,T)$, bounded by a modulus of continuity $\omega(r)$, i.e.
\[
S^2_r(\mu_t,T)^2 \le \omega(r),
\]
then for any $\epsilon > 0$, we have 
\[
\int_0^T \int_{L^2_x} \int_{D}  |u_\epsilon - u|^2 \, dx \, d\mu_t(u) \, dt
\le C \omega(\epsilon).
\]
\end{lemma}

\begin{proof}
Follows directly from estimate \eqref{eq:mollest}.
\end{proof}

The main result of the present section is the following

\begin{theorem} \label{thm:timecompact}
Let $\mu_t^\Delta \in L^1_t(\P)$ be a family of uniformly time-regular probability measures, for $\Delta > 0$, for which there exists $M>0$, such that $\mu^\Delta_t(B_M(0)) = 1$ for all $\Delta>0$, a.e. $t\in[0,T)$. Here $B_M(0) := \{ \Vert u \Vert_{L^2_x} < M \}$.
If there exists a modulus of continuity $\omega(r)$ such that
\[
S^2_r(\mu_t^\Delta,T) \le \omega(r), \quad \forall \Delta > 0,
\]
then $\mu_t^\Delta$ is relatively compact in $L^1_t(\P)$.
\end{theorem}

The idea behind theorem \ref{thm:timecompact} is to use the spatial regularity of the sequence to show that the weak time-regularity assumption of definition \ref{def:timereg} implies a similar time-regularity with respect to a stronger spatial norm, where $H^{-L}$ is replaced by $L^2$. The details of this argument are provided in appendix \ref{app:timecompact}.

Let us also remark that a limit $\mu_t^\Delta \to \mu_t$ of a uniformly time-regular sequence $\mu_t^\Delta$ is itself time-regular:

\begin{proposition} \label{prop:timereg}
Let $\mu_t^\Delta \in L^1_t(\P)$ be a family of uniformly time-regular probability measures,  for $\Delta > 0$. And such that there exists $M>0$ with $\mu^\Delta_t(B_M(0)) = 1$ for all $\Delta>0$, a.e. $t\in[0,T)$, where $B_M(0) := \{\Vert u \Vert_{L^2_x} < M \}$. If $\mu_t^\Delta \to \mu_t$ in $L^1_t(\P)$, then $\mu_t$ is time-regular (with the same time-regularity constants $C,L>0$ as for the family $\mu^\Delta_t$).
\end{proposition}

\begin{proof}
To prove proposition \ref{prop:timereg} and according to the definition of time-regularity definition \ref{def:timereg}, we need to show that there exists a map $(s,t) \mapsto \pi_{s,t} \in \P(L^2_x\times L^2_x)$, which defines a transport plan from $\mu_s$ to $\mu_t$, and satisfies the $H^{-L}$-estimate of definition \ref{def:timereg}. A proof of existence of such $\pi_{s,t}$ will be achieved by viewing each $\pi^\Delta_{s,t}$ as an element of $\P(H^{-L}_x\times H^{-L}_x)$, and observing that the $\pi^\Delta_{s,t}$ are concentrated on the compact subset $B_M(0) \times B_M(0) \subset H^{-L}_x\times H^{-L}_x$ for almost all $s,t$. Let us now choose a sequence $\Delta_k \to 0$. By assumption, for each $k\in \mathbb{N}$, there exists a subset $I_k \subset [0,T)$ of full Lebesgue measure, such that $\pi^{\Delta_k}_{s,t}$ is concentrated on $B_M(0) \times B_M(0)  \subset H^{-L}_x\times H^{-L}_x$ for all $s,t \in I_k$, and satisfies
\begin{align} \label{eq:H-Lest}
\int_{L^2_x \times L^2_x} \Vert u - v \Vert_{H^{-L}_x} \, d\pi^{\Delta_k}_{s,t}(u,v) 
\le C|t-s|.
\end{align}
Let $I = \bigcap_{k=1}^\infty I_k$. Then $I\subset [0,T)$ is of full Lebesgue measure, and $\pi^{\Delta_k}_{s,t}$ is concentrated on $B_M(0)\times B_M(0) \subset H^{-L}_x\times H^{-L}_x$ for all $s,t \in I$ and \eqref{eq:H-Lest} is satisfied for all $s,t \in I$, \emph{uniformly} for all $k\in \mathbb{N}$.

Fix now $s,t \in I$. Since the $\pi^{\Delta_k}_{s,t}$ are concentrated on $B_M(0)\times B_M(0) \subset H^{-L}_x\times H^{-L}_x$, the sequence $\pi^{\Delta_k}_{s,t}$ is tight in $\P(H^{-L}_x\times H^{-L}_x)$ (and hence relatively compact under weak convergence by Prokhorov's theorem). We can pass to a weakly convergent subsequence $\pi^{\Delta_{k_j}}_{s,t}$ in $\P(H^{-L}_x\times H^{-L}_x)$ such that $\pi^{\Delta_{k_j}}_{s,t} \weaklyto \pi_{s,t}$ for some $\pi_{s,t} \in \P(H^{-L}_x\times H^{-L}_x)$. Then, by definition of weak convergence, we have for any $F\in C_b(H^{-L}_x)$ that
\begin{align*}
\int_{H^{-L}_x\times H^{-L}_x} F(u) \, d\pi_{s,t}(u,v)
&=
\lim_{j\to \infty} \int_{H^{-L}_x \times H^{-L}_x} F(u) \, d\pi^{\Delta_{k_j}}_{s,t}(u,v)
\\
&=
\lim_{j\to \infty} \int_{L^2_x \times L^2_x} F(u) \, d\pi^{\Delta_{k_j}}_{s,t}(u,v)
\\
&=
\lim_{j\to \infty} \int_{L^2_x} F(u) \, d\mu^{\Delta_{k_j}}_{s}(u,v)
\\
&=
\int_{L^2_x} F(u) \, d\mu_s(u).
\end{align*}
When passing to the limit in the last step, we have used the convergence $\mu^{\Delta_j}_s \weaklyto \mu_s$ as elements of $\P(L^2_x)$, as well as the fact that if $F\in C_b(H^{-L}_x)$, then the restriction $F|_{L^2_x}$ to $L^2_x$ is in $C_b(L^2_x)$. It follows from $\int F(u) \, d\pi_{s,t}(u,v) = \int F(u) \, d\mu_s(u)$ for all $F\in C_b(H^{-L}_x)$ that $\mathrm{proj}_1\# \pi_{s,t} = \mu_s$. A similar argument shows that also $\mathrm{proj}_2\# \pi_{s,t} = \mu_t$. In particular, this also implies that $\pi_{s,t}$ is in fact concentrated on $L^2_x\times L^2_x$. Indeed, denote $X = L^2_x \subset H^{-L}_x$. Then the complement of the Cartesian product $X\times X$ in $H^{-L}_x\times H^{-L}_x$ satisfies $(X\times X)^c = X^c \times H^{-L}_x \cup H^{-L}_x \times X^c$. This implies that 
\begin{align*}
\pi_{s,t}((X\times X)^c) 
&\le \pi_{s,t}(X^c \times H^{-L}_x) + \pi_{s,t}(H^{-L}_x\times X^c)
\\
&= \mu_s(X^c) + \mu_t(X^c) \\
&=0.
\end{align*}
In the above estimate, the equality on the second line follows from the fact that $\mathrm{proj}_1\# \pi_{s,t} = \mu_s$ and $\mathrm{proj}_2\# \pi_{s,t} = \mu_t$, and the last equality follows from the fact that $\mu_s,\mu_t$ are supported on $X = L^2_x \subset H^{-L}_x$.

Furthermore, for any fixed $\gamma>0$, $(u,v) \mapsto F_\gamma(u,v) = \min(\gamma,\Vert u - v \Vert_{H^{-L}_x})$ is a continuous bounded function on $H^{-L}_x\times H^{-L}_x$. We can choose $\gamma>0$ sufficiently large so that $F_\gamma(u,v) = \Vert u - v \Vert_{H^{-L}_x}$ for $(u,v) \in \overline{B_M(0)} \times \overline{B_M(0)}$. It follows that 
\begin{align*}
\int_{L^2_x\times L^2_x} \Vert u - v \Vert_{H^{-L}_x} \, d\pi_{s,t}(u,v)
&= \int_{H^{-L}_x \times H^{-L}_x} F_\gamma(u,v) \, d\pi_{s,t}(u,v) \\
&= \lim_{j \to \infty} \int_{H^{-L}_x \times H^{-L}_x} F_\gamma(u,v) \, d\pi^{\Delta_{k_j}}_{s,t}(u,v) \\
&\le \sup_{k} \int_{L^2_x \times L^2_x}  \Vert u - v \Vert_{H^{-L}_x} \, d\pi^{\Delta_k}_{s,t}(u,v) \\
&\le C|s-t|.
\end{align*}
Since $s,t\in I$ were arbitrary and as $I\subset[0,T)$ is of full Lebesgue measure, we have thus shown that for almost all $s,t \in [0,T)$, there exists a transfer plan $\pi_{s,t}$ from $\mu_s$ to $\mu_t$ satisfying the assumptions of definition \ref{def:timereg}. This shows that the limit $\mu_t$ is itself time-regular.
\end{proof}

\subsection{Time-dependent Correlation measures and their compactness.}
It has been shown in \cite{FLM17} that there is a one-to-one correspondence between probability measures on $L^2_x$ and so-called correlation measures. Correlation measures are defined as infinite hierarchies of Young measures, taking into account spatial correlations, or more precisely,
\begin{definition}\label{def:corrmeas}
A \define{correlation measure} is a collection ${\bm \nu} = (\nu^1, \nu^2, \dots)$ of maps $\nu^k : D^k \to \Prob(\U^k)$ satisfying the following properties:
\begin{enumerate}
\item\textit{Weak-$\ast$ measurability:} Each map $\nu^k : D^k \to \Prob(\U^k)$ is weak-$\ast$-measurable, in the sense that the map $x \mapsto \ip{\nu^k_x}{f}$ from $x\in D^k$ into $\R$ is Borel measurable for all $f\in C_0(\U^k)$ and $k\in\N$. In other words, $\nu^k$ is a Young measure from $D^k$ to $\U^k$.
\item\textit{$L^2$-boundedness:} $\bm\nu$ is $L^2$-bounded, in the sense that
\begin{equation}\label{eq:corrlpbound}
\int_{D} \ip{\nu^1_x}{|\xi|^2}\,dx < +\infty.
\end{equation}
\item\textit{Symmetry:} If $\sigma$ is a permutation of $\{1,\dots,k\}$ and $f\in C_0(\R^k)$ then $\ip{\nu^k_{\sigma(x)}}{f(\sigma(\xi))} = \ip{\nu^k_x}{f(\xi)}$ for a.e.\ $x\in D^k$. Here, we denote $\sigma(x) = \sigma(x_1,x_2,\ldots, x_k) = (x_{\sigma_1},x_{\sigma_2},\ldots,x_{\sigma_k})$. $\sigma(\xi)$ is denoted analogously.
\item\textit{Consistency:} If $f\in C_0(\U^k)$ is of the form $f(\xi_1,\dots,\xi_k) = g(\xi_1,\dots,\xi_{k-1})$ for some $g\in C_0(\U^{k-1})$, then $\ip{\nu^k_{x_1,\dots,x_k}}{f} = \ip{\nu^{k-1}_{x_1,\dots,x_{k-1}}}{g}$ for almost every $(x_1,\dots,x_k)\in D^k$.
\item\textit{Diagonal continuity (DC):} If $B_r(x) := \bigl\{y\in D\ :\ |x-y|<r\bigr\}$ then
\begin{equation}\label{eq:dcproperty}
\lim_{r\to0}\int_D\intavg_{B_r(x)}\ip{\nu^2_{x,y}}{|\xi_1-\xi_2|^2}\,dy\,dx = 0.
\end{equation}
\end{enumerate}
Each element $\nu^k$ is called a \define{correlation marginal}. We let $\Lp^{2} = \Lp^{2}(D,\U)$ denote the set of all correlation measures from $D$ to $\U$.
\end{definition}

It has been shown in \cite{FLM17}, that if $\mu \in \P(L^2_x)$, then we can associate  to it a unique correlation measure $\bm{\nu}$, with the interpretation that for $A_1,\dots,A_k\subset U$:
\[
\mu[u(x_i)\in A_i, \; i=1,\dots,k] = \nu^k_{x_1,\dots,x_k}(A_1\times \dots \times A_k).
\]
More precisely, we have the following theorem \cite{FLM17}:
\begin{theorem}\label{thm:duality}
For every correlation measure $\bm\nu\in\Lp^2(D,\U)$ there exists a unique probability measure $\mu\in\Prob(L^2(D;U))$ satisfying
\begin{equation}\label{eq:lpbound}
\int_{L^2_x}  \Vert u\Vert_{L^2_x}^{2}\,d\mu(u) < \infty,
\end{equation}
such that
\begin{equation}\label{eq:nukdef}
\int_{D^k}\int_{\U^k}g(x,\xi)\,d\nu^k_x(\xi)\, dx = \int_{L^2_x} \int_{D^k}g(x,u(x))\,dx \, d\mu(u),
\end{equation}
for all $\ g\in C_0(D^k\times U^k)$ and $k\in\mathbb{N}$ (where $u(x)$ denotes the vector $(u(x_1), \dots, u(x_k))$). Conversely, for every probability measure $\mu\in\Prob(L^2(D;U))$ with finite moment \eqref{eq:lpbound}, there exists a unique correlation measure $\bm\nu\in\Lp^2(D,\U)$ satisfying \eqref{eq:nukdef}. The relation \eqref{eq:nukdef} is also valid for any measurable $g:D\times\U\to\R$ such that $|g(x,\xi)|\leq C|\xi|^2$ for a.e.\ $x\in D$.

Moreover, the \emph{moments}
\begin{equation}
    \label{eq:momdef}
    m^k:D^k \mapsto U^{\otimes k},\quad m^k(x) = \langle \nu^k_x, \xi_1 \otimes \xi_2 \otimes \ldots \otimes \xi_k \rangle,
    \end{equation}
uniquely determine the correlation measure ${\bf \nu}$ and hence the underlying probability measure $\mu$.
\end{theorem} 

The following result now follows from theorem \ref{thm:timecompact}:

\begin{theorem} \label{thm:compactconv}
Let $\{\mu_t^\Delta\}_{\Delta>0}$ be a family of uniformly time-regular probability measures in $L^1_t(\P)$, and assume that there exists $M>0$, such that $\mu_t^\Delta(B_M) = 1$ for all $\Delta>0$ and $t\in [0,T)$. Let $\bm{\nu}_t^\Delta = (\nu_t^{\Delta,1}, \nu_t^{\Delta,2},\dots)$ denote the corresponding time-parametrized correlation measures. If there exists a uniform modulus of continuity $\omega(r)$, such that 
\[
\int_0^T\int_{D} \fint_{B_r(x)} \langle \nu^{\Delta,2}_{t,x,y}, |\xi_1 - \xi_2|^2 \rangle \, dy \, dx \, dt \le \omega(r), \qquad \forall \Delta > 0,
\]
then $\{\mu_t^\Delta \}_{\Delta>0}$ is relatively compact in $L^1_t(\P)$, i.e. there exists a subsequence $\Delta_j \to 0$ ($j\in \N$), and a time-parametrized probability measure $\mu_t \in L^1_t(\P)$, such that 
\[
\int_0^T W_1(\mu_t^{\Delta_j}, \mu_t) \, dt \to 0, \quad \text{as }j\to \infty.
\] 

Furthermore, denoting by $\bm{\nu}_t = (\nu_t^1, \nu_t^2, \dots)$ the correlation measure corresponding to the limit $\mu$, we have
\begin{itemize}
\item $L^2$-bound:
$\int_{D} \langle \nu_{t,x}^1, |\xi|^2 \rangle \, dx \le M^2$, for a.e. $t\in [0,T)$,
\item the two-point correlations satisfy
\[
\int_0^T 
\int_{D} \fint_{B_r(x)} \langle \nu^{2}_{t,x,y}, |\xi_1 - \xi_2|^2 \rangle \, dy \, dx \, dt
 \le \omega(r),
\]
\item We define \emph{admissible observables}, in terms of test functions $g \in C([0,T)\times D^k\times U^k)$, which satisfy the following bounds,
\begin{equation}
    \label{eq:obs}
    \begin{aligned}
    |g(t,x,\xi)| &\le C \prod_{i=1}^k \left(1+|\xi_i|^2\right), \\
    |g(t,x,\xi)-g(t,x,\xi')| &\le C \sum_{i=1}^k \Pi_i(\xi,\xi') \sqrt{1+|\xi_i|^2 + |\xi_i'|^2} |\xi_i-\xi_i'|,
\end{aligned}
\end{equation}
where $C>0$ is a fixed constant, independent of $t\in [0,T)$, $x \in D^k$ and $\xi,\xi'\in U^k$. Here $\Pi_i(\xi,\xi')$ is defined as
\begin{align} \label{eq:observableLipPi}
\Pi_i(\xi,\xi') := \prod_{\substack{j=1 \\ j\ne i}}^k \left(1+|\xi_j|^2+|\xi_j'|^2\right), \quad \xi, \xi' \in U^k.
\end{align}
Then, these admissible observables converge strongly in $L^1_{t,x}$, in the sense that
\[
\lim_{j\to \infty}
\int_0^T
\int_{D^k} 
|\langle \nu_{t,x}^{\Delta_j,k}, g(x,\xi) \rangle - \langle \nu_{t,x}^{k}, g(x,\xi) \rangle| 
\, dx 
\, dt
= 0,
\]
\end{itemize}
\end{theorem}

\begin{remark}
We note that the uniform modulus of continuity estimate in theorem \ref{thm:compactconv} can equivalently be expressed as
\[
\int_0^T \int_{L^2_x} \int_D \fint_{B_r(0)} \left|u(x+h) - u(x)\right|^2 \, dh \, dx \, d\mu^\Delta_t(u) \, dt \le \omega(r),
\]
or $S^2_r(\mu^\Delta_t,T)^2 \le \omega(r)$, for all $\Delta > 0$.
\end{remark}

We now come to the proof of theorem \ref{thm:compactconv}.

\begin{proof}[Proof of theorem \ref{thm:compactconv}]
The compactness property follows easily from theorem \ref{thm:timecompact}, above. We only provide a proof of strong convergence of the observables here. We drop the subscript $j$ in the following, and assume that $\Delta \to 0$ is a sequence such that $\mu_t^\Delta \to \mu_t$ in $L^1_t(\P)$. We recall that by the assumption of this theorem, there exists $M>0$, such that $\mu^\Delta_t(B_M(0)) = 1$, for all $\Delta$. Fix now $g\in C([0,T)\times D^k \times U^k)$ as in the statement of the theorem. Denote 
\[
U^\Delta(t,x) := \langle \nu_{t,x}^{\Delta,k}, g(t,x,\xi) \rangle,
\quad
U(t,x) := \langle \nu_{t,x}^{k}, g(t,x,\xi) \rangle.
\]
By the assumed bound on $g(t,x,\xi)$, we have $U^\Delta, U\in L^1([0,T)\times D^k)$. Given $u\in L^2_x$ and $\epsilon > 0$, let $u_\epsilon$ denote the mollification of $u$. For fixed $\epsilon > 0$, and $x\in D^k$, the mapping $L^2_x \mapsto U^k$, $u \mapsto u_\epsilon(x)$ is continuous and bounded on $B_M(0) \subset L^2_x$, and so is $u \mapsto g(t,x,u_\epsilon(x))$. Here we recall that in our notation, $u_\epsilon(x) = (u_\epsilon(x_1),\dots,u_\epsilon(x_k))\in U^k$ for $x=(x_1,\dots,x_k)\in D^k$. We claim that for fixed $\epsilon > 0$, $x\in D^k$, the function
\begin{align} \label{eq:Ffun}
F_{x,\epsilon}(u) := g(t,x,u_\epsilon(x)), \; \text{ satisfies $F_{x,\epsilon}\in C_b(\overline{B_M(0)})$.}
\end{align}
Indeed, using the assumed Lipschitz bound on $\xi \mapsto g(t,x,\xi)$, we find for $u,v\in \overline{B_M(0)}$ from \eqref{eq:obs}, \eqref{eq:observableLipPi}:
\begin{align*}
\left| 
F_{x,\epsilon}(u) - F_{x,\epsilon}(v)
\right|
&=
\left|
g(t,x,u_\epsilon(x)) - g(t,x,v_\epsilon(x))
\right|
\\
&\le 
C \sum_{i=1}^k \Pi_{i}(u_\epsilon(x),v_\epsilon(x)) \sqrt{1+|u_\epsilon(x_i)|^2 + |v_\epsilon(x_i)|^2} |u_\epsilon(x_i) - v_\epsilon(x_i) |.
\end{align*}
Using the H\"older estimate $|u_\epsilon(x)| \le \Vert \rho_\epsilon \Vert_{L^2_x} \Vert u \Vert_{L^2_x}$ for the point-values of the mollification, it follows that 
\begin{align} \label{eq:FLip}
\left| 
F_{x,\epsilon}(u) - F_{x,\epsilon}(v)
\right|
\le 
C k (|D|+2M^2\Vert \rho_\epsilon \Vert_{L^2}^2)^{k-1/2}  \Vert \rho_\epsilon \Vert_{L^2_x} \Vert u - v \Vert_{L^2_x}, 
\end{align}
for all $u,v \in \overline{B_M(0)}$; this implies that $u\mapsto F_{x,\epsilon}(u)$ is Lipschitz-continuous on $\overline{B_M(0)}$, with 
\[
\Vert F_{x,\epsilon} \Vert_{\mathrm{Lip}} \le C k (|D|+2M^2\Vert \rho_\epsilon \Vert_{L^2_x}^2)^{k-1/2}  \Vert \rho_\epsilon \Vert_{L^2_x}.
\]
The previous observation allows us to define
\[
U^\Delta_\epsilon(t,x) := \int_{L^2_x} g(t,x,u_\epsilon(x)) \, d\mu^\Delta_t(u),
\]
and similarly $U_\epsilon(t,x)$ for $\mu_t$. Then, we have for any $\epsilon >0$ and $\phi \in L^\infty([0,T)\times D^k)$:
\begin{align*}
\int_0^T \int_{D^k} &\left[ U^\Delta_{\epsilon}(t,x) - U^\Delta(t,x) \right] \phi(t,x) \, dx \, dt
\\
&= 
\int_0^T \int_{L^2_x} \int_{D^k} \left[g(t,x,u_\epsilon(x)) - g(t,x,u(x)) \right]\phi(t,x) \, dx \, d\mu^\Delta_t(u) \, dt
\\
&\le 
 \int_0^T  \int_{L^2_x}  \int_{D^k} C\Vert \phi \Vert_{L^\infty_{t,x}}  \sum_{i=1}^k \Pi_i(u_\epsilon(x),u(x)) \\ 
 &\qquad
 \times \sqrt{1+|u(x_i)|^2 +|u_\epsilon(x_i)|^2}  |u_\epsilon(x_i) - u(x_i)| \, dx \, d\mu_t^\Delta(u) \, dt.
\end{align*}
According to the definition \eqref{eq:observableLipPi}, $\Pi_i(u_\epsilon(x),u(x))$ depends only on $x_1,\dots, x_{i-1},x_{i+1},\dots,x_k$. We can thus integrate over these variables to find
\[
\int_{D^{k-1}} \Pi_i(u_\epsilon(x),u(x)) \, dx_1 \dots dx_{i-1}\,dx_{i+1} \dots dx_k
\le
(|D|+2M^2)^{k-1},
\]
for all $u$ in the support of $\mu^\Delta_t$. Where we recall that $\mu_t^\Delta$ is supported on $\overline{B_M(0)} = \{\Vert u\Vert_{L^2_x}\le M\}$, by assumption. On the other hand, from H\"older's inequality, we have
\begin{align*}
\int_{D} \sqrt{1+|u(x_i)|^2+|u_\epsilon(x_i)|^2} |u_\epsilon(x_i)-u(x_i)| \, dx_i
&\le (|D|+2M^2)^{1/2} \Vert u - u_\epsilon \Vert_{L^2_x},
\end{align*}
on the support of $\mu^\Delta_t$. Combining these estimates, we conclude that 
\begin{align*}
\int_0^T \int_{D^k} &\left[ U^\Delta_{\epsilon}(t,x) - U^\Delta(t,x) \right] \phi(t,x) \, dx \, dt
\\
&\le 
C\Vert \phi \Vert_{L^\infty_{t,x}} k (|D|+2M^2)^{k-1/2} \int_0^T \int_{L^2_x} \Vert u - u_\epsilon \Vert_{L^2_x} \, d\mu^\Delta_t(u) \, dt
\\
&\le
C\Vert \phi \Vert_{L^\infty_{t,x}} k (|D|+2M^2)^{k-1/2} \sqrt{T} \left(\int_0^T \int_{L^2_x} \Vert u - u_\epsilon \Vert_{L^2_x}^2 \, d\mu^\Delta_t(u) \, dt\right)^{1/2}.
\end{align*}
In the last estimate, we have used Jensen's inequality. Finally, we recall that by \eqref{eq:mollest}, there exists a constant $C>0$, such that
\begin{align*}
 \Bigg(\int_0^T &\int_{L^2_x} \Vert u - u_\epsilon \Vert_{L^2_x}^2 \, d\mu^\Delta_t(u) \, dt\Bigg)^{1/2}
 \\
 &\le 
 C \left(\int_0^T \int_{L^2_x} \int_D \fint_{B_\epsilon(0)} |u(x+h)-u(x)|^2 \, dh \, dx \, d\mu^\Delta_t(u) \, dt \right)^{1/2}
 \\
 &=
 C S^2_\epsilon(\mu^\Delta_t;T).
\end{align*}

We have thus shown that
\begin{align*}
\int_0^T \int_{D^k} \left[ U^\Delta_{\epsilon}(t,x) - U^\Delta(t,x) \right] \phi(t,x) \, dx \, dt
&\le
C\Vert \phi \Vert_{L^\infty_{x,t}} S_{\epsilon}^2(\mu_t^\Delta,T),
\end{align*}
where $C = C(g,k,M,T)$ depends on the observable $g(t,x,\xi)$, the number of point-correlations $k$, the uniform a priori $L^2_x$-bound $M>0$, and the maximal time $T>0$, but is independent of $\Delta$, $\epsilon$ and $\phi$.

Taking the supremum over all $\phi$ s.t. $\Vert \phi \Vert_{L^\infty_{t,x}} \le 1$ on the left-hand side, we obtain
\begin{align} \label{eq:epsest}
\Vert U^\Delta_\epsilon(t,x) - U^\Delta(t,x) \Vert_{L^1_{t,x}}
\le C S_\epsilon^2(\mu_t^\Delta,T).
\end{align}
The same inequality holds also true for $U(x,t)$, following the same argument but dropping the subscript $\Delta$ in the above estimates. By lemma \ref{lem:tapprox}, and the uniform (in $\Delta$) upper bound on the structure function, it now follows that there exists an absolute constant $C>0$, such that
\begin{align} \label{eq:modest}
\Vert U^\Delta_\epsilon(t,x) - U^\Delta(t,x) \Vert_{L^1_{t,x}}
\le C\omega(\epsilon),
\quad
\Vert U_\epsilon(t,x) - U(t,x) \Vert_{L^1_{t,x}}
\le C\omega(\epsilon).
\end{align}
From the convergence $\mu_t^\Delta \to \mu_t$, it furthermore follows that for any $\epsilon > 0$ fixed:
\begin{align} \label{eq:deltest}
\Vert U^\Delta_\epsilon - U_\epsilon \Vert_{L^1_{t,x}}
\to 0, \text{ as } \Delta \to 0.
\end{align}
Indeed, this follows from the fact that 
\[
U^\Delta(t,x) = \int_{L^2_x} F_{x,\epsilon}(u) \, d\mu^\Delta_t(u),
\]
where $F_{x,\epsilon}(u)$ has been defined in \eqref{eq:Ffun} above, 
 and the following estimate:
\begin{align*}
\Vert U^\Delta_\epsilon - U_\epsilon \Vert_{L^1_{t,x}}
&=
\int_{D^k} \int_0^T
\left|
U^\Delta_\epsilon(t,x) - U_\epsilon(t,x)
\right| 
\, dt \, dx
\\
&=
\int_{D^k} \int_0^T 
\left|
\int_{L^2_x} F_{x,\epsilon}(u) \left[d\mu^\Delta_t(u) - d\mu_t(u)\right]
\right| 
\, dt \, dx
\\
&\le \int_{D^k} \left( \int_0^T \Vert F_{x,\epsilon}\Vert_{\mathrm{Lip}} W_1(\mu^\Delta_t,\mu_t) \, dt \right) \, dx
\\
&\le |D|^k \Vert F_{x,\epsilon}\Vert_{\mathrm{Lip}}  \left( \int_0^T W_1(\mu^\Delta_t,\mu_t) \, dt \right).
\end{align*}
By \eqref{eq:FLip}, we have $\Vert F_{x,\epsilon}\Vert_{\mathrm{Lip}} < \infty$, so that last term converges to $0$ as $\Delta \to 0$, as follows from the fact that $\mu^\Delta_t \to \mu_t$ in $L^1(\P)$.

Thus, combining \eqref{eq:epsest}, \eqref{eq:modest} and \eqref{eq:deltest}, we find
\begin{align*}
\Vert U^\Delta(t,x) - U(t,x) \Vert_{L^1_{t,x}}
&\le 
\Vert U^\Delta(t,x) - U^\Delta_\epsilon(t,x) \Vert_{L^1_{t,x}}
\\
&\quad
+ \Vert U^\Delta_\epsilon(t,x) - U_\epsilon(t,x) \Vert_{L^1_{t,x}}
\\
&\quad
+ \Vert U_\epsilon(t,x) - U(t,x) \Vert_{L^1_{t,x}}
\\
&\le 2C \omega(\epsilon) + \Vert U^\Delta_\epsilon(t,x) - U_\epsilon(t,x) \Vert_{L^1_{t,x}}.
\end{align*}
Which implies that 
\begin{align*}
\limsup_{\Delta \to 0} \Vert U^\Delta(t,x) - U(t,x) \Vert_{L^1([0,T)\times D^k)}
&\le 2C \omega(\epsilon),
\end{align*}
for any fixed $\epsilon > 0$. Finally, noting that the left-hand side is independent of $\epsilon$, we may let $\epsilon \to 0$ to show that 
\[
\limsup_{\Delta \to 0}\Vert U^\Delta(t,x) - U(t,x) \Vert_{L^1([0,T)\times D^k)} \le 0,
\]
or, equivalently, that $U^\Delta(t,x) \to U(t,x)$ in $L^1([0,T)\times D^k)$, as $\Delta \to 0$.
\end{proof}
\section{Dissipative Statistical solutions and their well-posedness}
\label{sec:weakstrong}
Given the discussion on time-parameterized probability measures in the last section, we can now define statistical solutions of \eqref{eq:Eulerfull} as,
\begin{definition}\label{def:statsol}
A time-parametrized probability measure $\mu_t \in L^1_t(\P)$ is a \define{statistical solution} of the incompressible Euler equations with initial data $\bar{\mu}$, if $t \mapsto \mu_t$ is time-regular, and the associated correlation measure $\bm{\nu}_t$ satisfies:
\begin{enumerate}
\item Given ${\phi}_1, \dots, {\phi}_k \in C^\infty([0,T)\times D;\R^d)$ with $\div({\phi}_i) = 0$ for all $i=1,\dots,k$, set
\[
{\phi}(t,x)
= 
{\phi}_1(t,{x}_1) \otimes \dots \otimes {\phi}_k(t,{x}_k),
\quad
\text{where }x=({x}_1,\dots,{x}_k).
\]
Then $\nu^k = \nu^{k}_{t,{x}_1,\dots,{x}_k}$ satisfies
\begin{gather*} 
\begin{aligned}
\int_0^T \int_{D^k} &\langle \nu^{k}, {\xi}_1 \otimes \dots \otimes {\xi}_k \rangle : \partial_t {\phi}
 \\
 &\quad 
 + \sum_i \langle \nu^{k}, {\xi}_1 \otimes \dots \otimes {F}({\xi}_i) \otimes \dots \otimes  {\xi}_k \rangle : \nabla_{{x}_i} {\phi}
\, dx \, dt
\\
&\quad + \int_{D^k} \langle \bar{\nu}^{k}, {\xi}_1 \otimes \dots \otimes {\xi}_k \rangle : {\phi}(0,x) \, dx
 =
0.
\end{aligned}
\end{gather*}
Here $\bar{\nu}$ is the correlation measure corresponding to the initial data $\bar{\mu}$. We denote ${F}({\xi}) := {\xi}\otimes {\xi}$ and the contraction in the second term is more explicitly given by
\[
 \left({\xi}_1 \otimes \dots \otimes {F}({\xi}_i) \otimes \dots \otimes  {\xi}_k \right): \nabla_{{x}_i} {\phi}
 =
\left[
 \textstyle\prod_{j\ne i} \left({\xi}_j \cdot {\phi}_j\right)
 \right] 
 \left(
 {\xi}_i\cdot \nabla_{{x}_i} {\phi}_i 
 \right)\cdot {\xi}_i. 
\]
\item For all $\psi \in C_c^\infty(D)$, we have
\[
\int_{D^2} \langle \nu^2_{t,{x}_1,{x}_2}, {\xi}_1 \otimes {\xi}_2 \rangle : \left(\nabla \psi({x}_1) \otimes \nabla \psi({x}_2)\right)
 \, dx_1 \, dx_2 = 0, 
\]
for a.e. $t\in [0,T)$.
\end{enumerate}
\end{definition}
The above PDEs specify the time-evolution of the moments \eqref{eq:momdef} for all $k$ and by 
theorem \ref{thm:duality}, determine the evolution of the probability measure $\mu_t$.
\begin{remark}
As $\nu^1$ above is a standard Young measure, it is straightforward to observe that the corresponding identity for the evolution of $\nu^1$ corresponds to the definition of measure-valued solution of \eqref{eq:Eulerfull} in the sense of \cite{DipernaMajda} under the further assumption that there is no concentration. Hence, one can think of statistical solutions as measure-valued solutions coupled with information about all possible multi-point correlations.

\end{remark}

We first show that the second property of definition \ref{def:statsol} is equivalent to the requirement that $\mu_t$ be supported on divergence-free vector fields for almost all $t$.
\begin{lemma} \label{lem:charincomp}
Let $\mu \in \mathcal{P}(L^2_x)$, with associated correlation measure $\bm{\nu}$. Then $\mu$ is concentrated on divergence-free vector fields if, and only if, 
\[
\int_{D^2} \langle \nu^2_{{x}_1,{x}_2}, {\xi}_1 \otimes {\xi}_2 \rangle : \left(\nabla \psi({x}_1) \otimes \nabla \psi({x}_2)\right)
 \, dx_1 \, dx_2 = 0, 
\]
for all $\psi \in C_c^\infty(D)$.
\end{lemma}

\begin{proof}
Let $\psi \in C_c^\infty(D)$. Then we have the following identity
\begin{gather}  \label{eq:relation}
\begin{aligned}
\int_{L^2_x} &\left[\int_{D} {u} \cdot \nabla \psi \, dx\right]^2 \, d\mu({u})
\\
&=
\int_{L^2_x} \int_{D} 
\left(
{u}({x}_1)\otimes {u}({x}_2) 
\right) 
: 
\left(
\nabla \psi({x}_1)\otimes \nabla \psi({x}_2)
\right)
 \, dx_1 \, dx_2 \, d\mu(u)
 \\
 &=
  \int_{D} 
\left\langle
\nu^2_{{x}_1,{x}_2}, 
{\xi}_1\otimes {\xi}_2 
\right\rangle 
: 
\left(
\nabla \psi({x}_1)\otimes \nabla \psi({x}_2)
\right)
 \, dx_1 \, dx_2.
\end{aligned}
\end{gather}
Therefore, if $\mu$ is concentrated on the divergence-free vector fields, then 
\[
\int_{D} {u} \cdot \nabla \psi \, dx = 0, \quad \text{$\mu$-almost surely},
\]
and hence, from equation \eqref{eq:relation}, we obtain
\begin{align} \label{eq:nuvanish}
  \int_{D} 
\left\langle
\nu^2_{{x}_1,{x}_2}, 
{\xi}_1\otimes {\xi}_2 
\right\rangle 
: 
\left(
\nabla \psi({x}_1)\otimes \nabla \psi({x}_2)
\right)
 \, dx_1 \, dx_2
=0.
\end{align}

To prove the converse, let us assume that relation \eqref{eq:nuvanish} holds for all $\psi\in C^\infty_c$. Let $\psi_n\in C^\infty_c$, $n\in \mathbb{N}$, be a countable, dense subset of $H^1(D)$.
Then, we have 
\[
\left\{ {u}\in L^2_x \, | \, \div({u}) \ne 0 \text{ distributionally}\right\}
=
\bigcup_{n\in \mathbb{N}} 
\left\{ {u}\in L^2_x \, \Big| \,\textstyle \int_{D} {u}\cdot \nabla \psi_n \, dx \ne 0\right\}.
\]
Set now $F_n({u}) := \left[\int_{D} {u}\cdot \nabla \psi_n \, dx\right]^2$. We note that for any $\epsilon>0$, we have
\[
\mu\left[F_n({u}) > \epsilon\right] 
\le 
\frac{1}{\epsilon} \int_{L^2_x} F_n({u}) \, d\mu({u})
= 0,
\]
where the last equality follows from \eqref{eq:relation} and the assumption \eqref{eq:nuvanish}. Letting $\epsilon \to 0$, we conclude that $\mu\left[F_n({u}) > 0\right] = 0$ for all $n\in \mathbb{N}$. Equivalently, we have 
\[
\mu\left[\int_{D} {u}\cdot \nabla \psi_n \ne 0 \right] = 0, \quad \text{for all }n\in \mathbb{N}.
\]
Finally, we conclude that
\begin{align*}
\mu\left[\div({u}) \ne 0\right]
&= 
\mu \left[\textstyle\bigcup_{n\in \mathbb{N}} \left\{\textstyle \int_{D} {u}\cdot \nabla \psi_n \ne 0\right\}\right]
\\
&\le 
\sum_{n\in \mathbb{N}} \mu\left[\textstyle \int_{D} {u}\cdot \nabla \psi_n \ne 0\right]
\\
&= 0.
\end{align*}
Hence $\mu\left[\div({u}) = 0\right] = 1$, i.e. $\mu$ is concentrated on divergence-free vector fields.
\end{proof}

Note that if $\rho, \mu\in \P(L^2_x)$ are probability measures, and if $\rho$ is of the form 
\[
\rho = \sum_{i=1}^M \alpha_i \delta_{u_i},
\]
where $\alpha_i > 0$, $\sum_{i=1}^M \alpha_i = 1$, and $u_i \in L^2_x$, then a transport plan from $\mu$ to $\rho$ is necessarily of the form \cite{FLM17}:
\[
\pi = \sum_{i=1}^M \alpha_i \mu_i \otimes \delta_{u_i},
\]
where $\mu_i \in P(L^2_x)$, and $\sum_{i=1}^M \alpha_i \mu_i = \mu$. Therefore, given $\alpha = (\alpha_1, \dots, \alpha_M)$ as above, and $\mu \in \P(L^2_x)$, we denote 
\[
\Lambda(\alpha,\mu) 
:=
\left\{
(\mu_1, \dots, \mu_M) 
\, \Big| \, 
\mu_i \in \P(L^2_x), \; \textstyle\sum_{i=1}^M \alpha_i \mu_i = \mu
\right\}.
\]
Note that the set $\Lambda(\alpha,\mu)$ is non-empty, since it contains $(\mu, \dots, \mu)$.

In analogy with work of \cite{FLM17,FLMW1} on entropy statistical solutions for hyperbolic systems of conservation laws, we define
\begin{definition}[Dissipative statistical solution] \label{def:dissipative}
A statistical solution $\mu_t \in L^1_t(\P)$ is called \emph{dissipative}, if for every choice of coefficients $\alpha_i > 0$ with $\sum_{i=1}^M \alpha_i = 1$ and for every $(\overline{\mu}_1, \dots, \overline{\mu}_M)\in \Lambda(\alpha,\overline{\mu})$, there exists a function $t \mapsto (\mu_{1,t},\dots,\mu_{M,t}) \in \Lambda(\alpha,\mu_t)$, such that $t \mapsto \mu_{i,t}$ is weak-$\ast$ measurable, $\mu_{i,t}|_{t=0} = \overline{\mu}_i$, such that each $\mu_{i,t}$ satisfies
\[
 \int_0^T \int_{L^2_x} \int_{D} 
\left[
{u}\cdot \partial_t {\phi}
+ 
({u}\otimes {u}) : \nabla {\phi}
\right]
\, dx \, d \mu_{i,t}(u) \, dt 
=
-\int_{L^2_x} \int_{D} 
{u}\cdot {\phi}(0,x) \, dx \, d\overline{\mu}_i({u}),
\]
for all ${\phi}\in C^\infty_c([0,T)\times D)$, $\div({\phi}) = 0$, and all $i=1,\dots, M$. And, in addition, we have for almost every $t\in [0,T)$:
\[
\int_{L^2_x} \Vert {u} \Vert^2_{L^2_x} \, d\mu_{i,t} ({u})
\le 
\int_{L^2_x} \Vert {u} \Vert^2_{L^2_x} \, d\overline{\mu}_i({u}), 
\quad i=1,\dots, M.
\]
\end{definition}

\subsection{Existence and uniqueness of dissipative solutions}
\label{sec:existuniq}
In this section, we show based on purely topological arguments, that if the set of $C^1$-regular initial data admitting classical solutions of \eqref{eq:Eulerfull}, over a given time-interval $[0,T)$ is dense in $L^2_x$, then there exists a (topologically) generic set $\mathcal{G} \subset L^2_x$, containing these regular initial data, with the following property: For any initial data $\bar{\mu}\in \P(L^2_x)$ which is concentrated on this generic set $\mathcal{G}\subset L^2_x$, we have \emph{existence and uniqueness in the class of dissipative statistical solutions}. By a ``generic'' set $\mathcal{G}$, we denote a set whose complement $\mathcal{E} = L^2_x \setminus \mathcal{G}$ is a countable union of nowhere dense sets (implying that $\mathcal{E}$ is a meagre set in the topological sense). We say that $\bar{\mu}$ is concentrated on $\mathcal{G}$, if $\bar{\mu}(\mathcal{G}) = 1$.

The construction of such a generic $\mathcal{G}$ under the above mentioned assumption has first been carried out in \cite{Lions}. Let us first review the construction of $\mathcal{G}$. We let $\mathcal{C} \subset C^1(D;U)$ denote the set of initial data $\overline{v}$ admitting a classical solution $v(t)$ on $[0,T)$, with $C(\overline{v}) := \sup_{t\in[0,T)}\Vert \nabla v(t) \Vert_{L^\infty}$ finite, i.e. $C(\overline{v})< \infty$. For $n\in \mathbb{N}$, define the open set $\mathcal{G}_n$, by
\begin{align} \label{eq:genericn}
\mathcal{G}_n
:=
\left\{
\overline{u} \in L^2_x
\, \Big|\, 
\exists\, \overline{v}\in \mathcal{C} \text{ s.t. } \Vert \overline{u}-\overline{v} \Vert_{L^2_x} < \frac 1n e^{-C(\overline{v}) T} 
\right\}
\end{align}
Finally, we let $\mathcal{G} = \bigcap_{n\in \mathbb{N}} \mathcal{G}_n$. 

\begin{remark}
If there exists a dense set of initial data $\overline{v} \in \mathcal{C}$, then $\mathcal{G}$ is generic in the topological sense (more precisely a $G_\delta$ set), being the countable intersection of the \emph{dense} open sets $\mathcal{G}_n$. By the Baire category theorem, the set $\mathcal{G}$ is non-empty and dense in this case. In particular, this would hold true if there is no finite-time blow-up for sufficiently smooth classical solutions of the incompressible Euler equations (e.g. for $C^{1,\alpha}$ initial data $\overline{v}$ possessing a H\"older continuous derivative), which is an established fact in two space dimensions but an open question in three space dimensions.
\end{remark}

We can now state the main theorem of the present section:

\begin{theorem} \label{thm:existuniq}
Define the generic set $\mathcal{G}$ as above (cp. equation \eqref{eq:genericn}). If $\bar{\mu} \in \P(L^2_x)$ is an initial datum such that $\bar{\mu}(\mathcal{G}) = 1$ and there exists $M>0$ such that $\bar{\mu}(B_M(0))=1$, then there exists a unique dissipative statistical solution $\mu_t$ of the incompressible Euler equations with initial data $\bar{\mu}$. 
\end{theorem}

The proof of theorem \ref{thm:existuniq} is somewhat technical and is left to appendix \ref{app:existuniqgeneric}.

\begin{remark}
Even though a topologically generic set $\mathcal{G}$ is large in some sense (e.g. in the sense that a countable intersection of generic sets is again generic, and in particular non-empty), it might be very small from a different point of view. For example, there exist generic subsets of $\mathbb{R}^n$, which have Lebesgue measure zero. It is therefore a priori not clear whether there are any ``interesting'' $\bar{\mu}$, such that $\bar{\mu}(\mathcal{G}) = 1$, beyond those $\bar{\mu}$ which are concentrated on smooth initial data. Nevertheless, theorem \ref{thm:existuniq} implies suitable weak-strong uniqueness results as we show below.
\end{remark}
We can derive the following corollaries from theorem \ref{thm:existuniq}.

\begin{corollary}[Short-time existence and uniqueness] \label{cor:shorttime}
If $m \ge \lfloor d/2 \rfloor + 2$, and if there exists a $C>0$, such that $\overline{\mu} \in \P(L^2_x)$ is concentrated on 
\[
\{ 
\overline{u} \in H^m_x \, | \, \Vert \overline{u} \Vert_{H^m_x} \le C
\},
\]
then there exists $T^\ast > 0$ (depending only on $C$) and a statistical solution $\mu_t: [0,T^\ast] \to \P(L^2_x)$ with initial data $\overline{\mu}$. Furthermore, $\mu_t$ is unique in the class of dissipative statistical solutions for $t\in [0,T^\ast]$.
\end{corollary}

\begin{proof}
Classical short-time existence results for the Euler equations show \cite{majda2001} that there exists $T^\ast > 0$, such that for initial data $\overline{u}$ with $\Vert \overline{u} \Vert_{H^m_x} \le C$, there exists a unique solution $u(t)$ such that
\[
\sup_{t\in [0,T^\ast]} \Vert u(t) \Vert_{H^m_x} 
\le C' \Vert \overline{u} \Vert_{H^m_x}.
\]
Since $H^m_x \embeds C^1$, this implies that $\overline{\mu}$ is concentrated on $\overline{u} \in \mathcal{C}$. In particular, we conclude that $\overline{\mu}(\mathcal{G}) =1$, and the result now follows from theorem \ref{thm:existuniq}.
\end{proof}

\begin{corollary}[Weak-Strong uniqueness in 2d] \label{cor:weakstrong}
Let $d=2$, and let $\alpha \in (0,1)$. If $\overline{\mu}$ is concentrated on $C^{1,\alpha}(D;U)$ and if there exists $M>0$, such that $\overline{\mu}(B_M(0))=1$, then there exists a dissipative statistical solution $\mu_t$ with initial data $\overline{\mu}$. Furthermore, $\mu_t$ is unique in the class of dissipative statistical solutions with initial data $\overline{\mu}$.
\end{corollary}

\begin{proof}
Again, we observe that for any $\overline{u}\in C^{1,\alpha}$, there exists a unique solution $u(t) \in C^{1,\alpha}$. Hence, we have $\overline{u} \in \mathcal{C}$ for all such $\overline{u}$. In particular, it follows that $\overline{\mu}$ is concentrated on $\mathcal{G}$. The claim follows from theorem \ref{thm:existuniq}.
\end{proof}
\section{Numerical approximation of statistical solutions}
\label{sec:spectralvisc}
In this section, we will propose an algorithm for computing statistical solutions of the incompressible Euler equations \eqref{eq:Eulerfull}. As mentioned before, this algorithm is very similar to the one proposed in \cite{FLMW1} for computing statistical solutions of hyperbolic systems of conservation laws, which in turn was inspired by the ensemble averaging algorithms of \cite{FKMT17,LM2015} for computing measure-valued solutions. This algorithm requires a spatio-temporal discretization and a Monte Carlo sampling of the underlying probability space. We propose to use a spectral viscosity spatial discretization which is described below.

\subsection{Spectral hyper-viscosity scheme}
We write ${u}^\Delta(x,t) = \sum_{|{k}|_\infty\le N} \widehat{{u}}^\Delta_{{k}}(t) e^{i{k}\cdot{x}}$, where now and in the following we shall consistently denote $\Delta = 1/N$, and we denote $|{k}|_\infty\defeq \max_{i=1,\dots,d} |k_i|$. We consider the following spectral viscosity approximation \cite{Tadmor1989,Tadmor2004} of the incompressible Euler equations,
\begin{gather} \label{eq:spectralvisc}
\left\{
\begin{aligned}
\partial_t {u}^\Delta
+\P_N({u}^\Delta\cdot \nabla {u}^\Delta) 
+ \nabla p^\Delta 
&=
-\epsilon_N |\nabla|^{2s} (Q_N \ast {u}^\Delta), 
\\
\div({u}^\Delta) 
&= 
0, 
\\
{u}^\Delta|_{t=0} 
&=
\P_N \overline{u}. 
\end{aligned}
\right.
\end{gather}
Here $\P_N$ is the spatial Fourier projection operator, mapping an arbitrary function $f(x,t)$ onto the first $N$ Fourier modes: $\P_N f(x,t) = \sum_{|{k}|_\infty\le N} \widehat{f}_{{k}}(t) e^{i{k}\cdot{x}}$.
$Q_N$ is a Fourier multiplier of the form 
\begin{gather}
Q_N(x) = \sum_{m_N < |{k}| \le N} \widehat{Q}_{{k}} e^{i{k}\cdot{x}}, 
\end{gather}
and we assume $0 \le \widehat{Q}_{{k}} \le 1$. 

The idea behind the SV method is that dissipation is only applied on the upper part of the spectrum, i.e. for $|{k}| > m_N$, thus preserving the formal spectral accuracy of the method, while at the same time enabling us to enforce a sufficient amount of energy dissipation on the small scale Fourier modes which is needed to stabilize the method. The additional hyperviscosity parameter $s\ge 1$ in \eqref{eq:spectralvisc} can be chosen larger to enforce more numerical dissipation on the high Fourier modes, thus allowing a larger part of the Fourier spectrum to remain free of numerical diffusion, while still ensuring stability of the resulting numerical scheme.

Note that $Q_N$ is defined via its Fourier coefficients $\widehat{Q}_k$, which, given an additional parameter $\theta>0$, are assumed to satisfy the following constraints:
\begin{gather} \label{eq:prereq}
\widehat{Q}_k = 0, \quad \text{for }|k|\le m_N, \quad
1-\left(\frac{m_N}{|k|}\right)^{(2s-1)/\theta} 
\le \widehat{Q}_k \le 1.
\end{gather}
In \cite{Tadmor2004}, the parameters $m_N$, $\epsilon_N$, $\theta$ are chosen such that
\begin{align} \label{eq:prereq2}
m_N \sim N^\theta, \quad \epsilon_N \sim \frac{1}{N^{2s-1}}, \quad 0\le \theta < \frac{2s-1}{2s}.
\end{align}

Multiplying the evolution equation \eqref{eq:spectralvisc} by $u^\Delta$ and integrating by parts, we obtain the following energy balance,
\begin{align} \label{eq:Eest}
\Vert u^\Delta(t) \Vert^2_{L^2_x} 
+ 2\epsilon_N \int_0^t \sum_{|k|_\infty \le N} \widehat{Q}_k |k|^{2s} |\widehat{u}^\Delta_k(\tau)|^2 \,\d \tau 
= \Vert u^\Delta(t=0) \Vert^2_{L^2_x} 
\le \Vert \overline{u} \Vert_{L^2_x}.
\end{align}
\subsection{Monte Carlo algorithm}
Following \cite{FLMW1}, the computation of statistical solutions of \eqref{eq:Eulerfull} requires combining the spectral viscosity scheme \eqref{eq:spectralvisc} with the following Monte Carlo sampling, 

\begin{algorithm}[Monte Carlo] \label{alg:MC}
Given $\bar{\mu}\in \mathcal{P}(L^2_x)$, and a grid scale $\Delta = 1/N$, we determine an approximate statistical solution $\mu^\Delta_t$, as follows: For $m=m(N)$, 
\begin{itemize}
\item Generate i.i.d. samples $\bar{u}_1,\dots, \bar{u}_m \sim \bar{\mu}$,
\item Evolve the samples, using the numerical scheme $u_i^\Delta(t) := S^{\Delta}\bar{u}_i$, where $S^\Delta_t$ denotes the solution operator, defined by the scheme \eqref{eq:spectralvisc}.
\item The approximate statistical solution $\mu^\Delta_t$ is given by the so-called \emph{empirical measure},
\begin{equation}
    \label{eq:emeas}
\mu^\Delta_t := \frac{1}{m} \sum_{i=1}^m \delta_{u_i^\Delta(t)}.
\end{equation}
\end{itemize}
\end{algorithm}
We remark that in practice, the samples $\bar{u}_i$ for $1 \leq i \leq m$ are random realizations with respect to a certain underlying probability space.
\begin{remark}
The Monte Carlo algorithm \ref{alg:MC}, when restricted only to the computation of the first correlation marginal $\nu^1$, reduces to the ensemble averaging algorithm proposed in \cite{LM2015} for computing measure-valued solutions of the incompressible Euler equations.
\end{remark}
\subsection{Convergence to Statistical Solutions.} In this section, we will investigate the convergence of the empirical measure $\mu^{\Delta}_t$ \eqref{eq:emeas}, generated by the Monte Carlo algorithm \ref{alg:MC}, to a statistical solution of \eqref{eq:Eulerfull}. To this end, we seek to apply the convergence theorem \ref{thm:timecompact} to these approximations. We start by verifying the temporal regularity of the empirical measures in the following lemma,
\begin{lemma} \label{lem:timeregularity}
There exists $L\in \mathbb{N}$, such that if ${u}^\Delta$ is obtained from the spectral hyper-viscosity method \eqref{eq:spectralvisc}, with $\Delta = 1/N$ and initial data $\overline{u}\in L^2_x$, then 
\[
\partial_t {u}^\Delta + \div( {u}^\Delta\otimes {u}^\Delta) + \nabla p^\Delta = E^\Delta,
\] 
where $\Vert E^\Delta \Vert_{H^{-L}_x} \le C \Delta (1+\Vert {u}^\Delta \Vert_{L^2_x}^2)$. Furthermore, there exists a constant $C'$, such that 
\[
\Vert u^\Delta(t) - u^\Delta(s) \Vert_{H^{-L}_x} \le C' (1+\Vert \overline{u} \Vert_{L^2_x}^2) |t-s|.
\]
\end{lemma}

\begin{proof}
From the definition of the spectral hyperviscosity scheme \eqref{eq:spectralvisc}, we obtain,
\[
\partial_t u^\Delta + \div(u^\Delta \otimes u^\Delta) + \nabla p^\Delta = E^\Delta,
\]
with $E^\Delta = E^\Delta_1 + E^\Delta_2$, where
\[
E^\Delta_1 = \epsilon_N |\nabla|^{2s} u^\Delta, \qquad
E^\Delta_2 = (I-\P_N) \div(u^\Delta \otimes u^\Delta).
\]
Clearly, the first error term $E_1^\Delta$ can be estimated by
\begin{align*}
\Vert E_1^\Delta \Vert_{H^{-2s}_x}
&= \epsilon_N \Vert |\nabla|^{2s} u^\Delta \Vert_{H^{-2s}_x}
\\
&\le \epsilon_N \Vert u^\Delta \Vert_{L^2_x}
\\
&\le \frac{1}{N} (1 + \Vert u^\Delta \Vert_{L^2_x}^2),
\end{align*}
where we have used that $\epsilon_N \sim N^{-2s+1} \le N^{-1}$ in the last step (assuming $s\ge 1$).
On the other hand, to estimate the second term, let $\phi \in C^\infty(D;U)$ be a given vector field. Then
\begin{align*}
\int_D \phi \cdot E^\Delta_2 \, dx
&=
\int_D \phi \cdot (I-\P_N) \div(u^\Delta \otimes u^\Delta) \, dx
\\
&=
-\int_D (I-\P_N)\nabla\phi :(u^\Delta \otimes u^\Delta) \, dx
\\
&\le 
\Vert (I-\P_N) \nabla \phi \Vert_{L^\infty_x} \Vert u^\Delta \Vert_{L^2_x}^2.
\end{align*}
Choosing $\ell>0$ sufficiently large, we have by the Sobolev embedding theorem $H^\ell_x \embeds L^\infty_x$. Hence, we can further estimate for some absolute constant $C>0$:
\begin{align*}
\Vert (I-\P_N) \nabla \phi \Vert_{L^\infty_x}
&\le 
C \Vert (I-\P_N) \nabla \phi \Vert_{H^\ell_x}
\\
&\le 
C \Vert (I-\P_N) \phi \Vert_{H^{\ell+1}_x}
\\
&\le 
C N^{-1} \Vert \phi \Vert_{H^{\ell+2}_x}.
\end{align*}
Thus, we have shown that 
\[
\int_D \phi \cdot E^\Delta_2 \, dx
\le 
\Vert \phi \Vert_{H^{\ell+2}_x} CN^{-1} \Vert u^\Delta \Vert_{L^2_x}^2,
\]
for all $\phi \in C^\infty(D;U)$, which implies by duality that $\Vert E^\Delta_2 \Vert_{H^{-(\ell+2)}_x} \le CN^{-1} \Vert u^\Delta \Vert_{L^2_x}^2$.

Let now $L := \max(2s,\ell+2)$. Then, combining the above two estimates, and noting that $H^{-(\ell+s)}_x, H^{-2s}_x \embeds H^{-L}_x$, we now conclude that there exists an absolute constant $C>0$, such that 
\[
\Vert E^\Delta \Vert_{H^{-L}_x} \le CN^{-1} (1+\Vert u^\Delta \Vert_{L^2_x}^2).
\]
Since $\Delta = 1/N$, this is the claimed estimate for $E^\Delta$.

For the time-regularity estimate, we write 
\[
\partial_t u^\Delta = -\P_N \div(u^\Delta \otimes u^\Delta) - \nabla p^\Delta + E_1^\Delta.
\]
By an argument analogous to the above, we then find that 
\[
\Vert \partial_t u^\Delta \Vert_{H^{-L}_x} 
= 
\Vert \P_N \div(u^\Delta \otimes u^\Delta) + \nabla p^\Delta + E_1^\Delta\Vert_{H^{-L}_x} 
\le C' (1+\Vert u^\Delta \Vert_{L^2_x}^2).
\]
Recalling that also $\Vert u^\Delta \Vert_{L^2_x} \le \Vert \overline{u}\Vert_{L^2_x}$, it follows that
\begin{align*}
\Vert u^\Delta(t) - u^\Delta(s) \Vert_{H^{-L}_x}
&= \left\Vert \int_s^t \partial_t u^\Delta(\tau) \, d\tau \right\Vert_{H^{-L}_x}
\\
&\le
\int_s^t\left\Vert \partial_t u^\Delta(\tau) \right\Vert_{H^{-L}_x} \, d\tau 
\\
&\le C' (1+\Vert \overline{u} \Vert_{L^2_x}^2) |t-s|.
\end{align*}
This concludes our proof.
\end{proof}

From lemma \ref{lem:timeregularity}, it is now easy to see that if $\mu_t^\Delta$ is generated by the Monte-Carlo algorithm \ref{alg:MC}, i.e.
\[
\mu^\Delta_t = \frac 1M \sum_{i=1}^M \delta_{u_i^\Delta(t)},
\]
with $u_i^\Delta(t)$ computed by the spectral hyper-viscosity scheme \eqref{eq:spectralvisc}, then the transport plan defined by
\[
\pi^\Delta_{s,t} := \frac 1M \sum_{i=1}^M \delta_{u_i^\Delta(s)} \otimes \delta_{u_i^\Delta(t)},
\]
satisfies the properties required by the definition of time-regularity, definition \ref{def:timereg}. This provides the required temporal regularity required by theorem \ref{thm:compactconv}.

Next, we turn our attention to the spatial regularity bounds of theorem \ref{thm:timecompact}. In particular, we need to obtain uniform estimates on the structure function \eqref{eq:tasf}. We start with the following simple observation,
\begin{lemma} \label{lem:simpleest}
For any $r\ge 0$, we have
\[
\fint_{B_r(0)} \left| e^{ik\cdot h} - 1 \right|^2 \, dh
\le C \min(|k|^2 r^2, 1) 
\le C |k|^2 r^2,
\]
where $C=4$.
\end{lemma}

\begin{proof}
Fix $k,h \in \mathbb{R}^d$. Let $f(\tau) := |e^{i\tau k\cdot h} - 1|$. Then
\[
f(0) = 0, \; |f'(\tau)| \le  |k||h|.
\]
This implies that for $|h|\le r$, and $\tau\in [0,1]$:
\[
|f(\tau)| \le \int_0^t |k||h| \, ds \le |k| r.
\]
Furthermore, the upper bound $|f(t)|\le 2$ is obvious, so that 
\[
|e^{ik\cdot h} - 1|^2 \le 4\min(|k|^2 r^2,1).
\]
Averaging over $h\in B_r(0)$, it follows that 
\[
\fint_{B_r(0)} 
|e^{ik\cdot h} - 1|^2 \, dh
 \le 4\min(|k|^2 r^2,1).
\]
\end{proof}
The next result is an estimate on the structure function \eqref{eq:tasf} at the grid scale $\Delta$.

\begin{lemma} \label{lem:weakBV}
If $\mu^{\Delta}_t$ is an approximate statistical solution obtained from the spectral hyper-viscosity method with $\Delta = 1/N$, and initial data $\bar{\mu}$ for which there exists $M>0$ such that $\bar{\mu}(B_M(0)) = 1$ where $B_M(0) = \{ \Vert u \Vert_{L^2_x} < M\}$, then 
\[
S^2_{\Delta}(\mu^\Delta_t,T) 
\le C M \Delta^{1/(2s)},
\]
for some absolute constant $C>0$. The same estimate is also true for $r\le \Delta$, i.e. we have
\[
S^2_{r}(\mu^\Delta_t,T) \le C M r^{1/(2s)}, \quad \text{ for all } r\le \Delta.
\]
\end{lemma}

\begin{proof}
  Our goal is to estimate the structure function $S_r^2(\mu_t^\Delta,T)$ based on the energy estimate \eqref{eq:Eest}. By construction, we have
  \[
\mu_t^\Delta = \frac{1}{m} \sum_{i=1}^m \delta_{u^\Delta_i(t)},
  \]
  and $u^\Delta_i$ is obtained by the spectral hyper-viscosity method \eqref{eq:spectralvisc} with initial data $\overline{u}_i$. By Plancherel's theorem, we obtain
\begin{align*}
  \int \fint_{B_r(0)} |u^\Delta_i(x+h)-u^\Delta_i(x)|^2 \, dh \, dx
  &= \fint_{B_r(0)} \sum_{|k|_\infty \le N} \left|e^{ik\cdot h} - 1 \right|^2 |\widehat{u}_k^{\Delta}|^2 \d h
  \\
  &\le C \sum_{|k|_\infty \le N} |k|^2r^2  |\widehat{u}_{i,k}^{\Delta}|^2,
\end{align*}
where the last estimate is shown in lemma \ref{lem:simpleest} below. Let us now fix some $\alpha$ with $\theta \le \alpha < 1$. We split the summation over modes $|k|\le 2N^\alpha$ and $|k|>2N^\alpha$:
\begin{align}
\int \fint_{B_r(0)} |u^\Delta_i(x+h)-u^\Delta_i(x)|^2 \,\d h\, \d x
&\le r^2 \left( \sum_{|k|\le 2N^\alpha}  + \sum_{|k|> 2N^\alpha} \right) |k|^2 |\widehat{u}_{i,k}^{\Delta}|^2 .
\end{align}
From the $L^2$-bound $\Vert u_i^\Delta(t) \Vert_{L^2_x}\le \Vert \overline{u}_i \Vert_{L^2_x}$, we trivially estimate
\begin{align*}
\sum_{|k|\le 2N^\alpha} r^2|k|^2 |\widehat{u}_{i,k}^\Delta|^2
\le 4r^2 N^{2\alpha} \Vert \overline{u}_i \Vert_{L^2_x}.
\end{align*}
Recall that for all $|k|\ge 2N^\alpha \ge 2m_N$ we have
\[
\widehat{Q}_k 
\ge 1-\left(\frac{m_N}{|k|}\right)^{(2s-1)/\theta}
\ge 1-2^{-(2s-1)/\theta}
\ge 1-2^{-2s} = C,
\]
as a consequence of \eqref{eq:prereq}, \eqref{eq:prereq2}. We thus find
\begin{align*}
\int_0^T \sum_{|k|> 2N^\alpha} r^2|k|^2 |\widehat{u}_{i,k}^\Delta|^2 \, \d t
&\le C r^2 \int_0^T \sum_{|k|>2N^\alpha} \widehat{Q}_k |k|^2 |\widehat{u}_{i,k}^\Delta|^2 \, \d t \\
&\le C r^2 N^{-2\alpha(s-1)} \int_0^T \sum_{|k|>2N^\alpha} \widehat{Q}_k |k|^{2s} |\widehat{u}_{i,k}^\Delta|^2 \, \d t \\
&\le Cr^2 \frac{N^{-2\alpha (s-1)}\Vert \overline{u}_i \Vert_{L^2_x}}{\epsilon_N} \\
&= Cr^2 N^{2\alpha}N^{-2\alpha s+2s-1}\Vert \overline{u}_i  \Vert_{L^2_x}.
\end{align*}
where we have used the prescribed scaling $\epsilon_N \sim N^{-2s+1}$ in the last step.

Combining both estimates, and noting that for all $i$, $\Vert u_i \Vert_{L^2} \le M$ by assumption, we conclude that
\begin{equation} \label{eq:weakBV}
\int_0^T \int \fint_{B_r(0)} 
|u_i^{\Delta}(x+h)-u_i^{\Delta}(x)|^2 
\, dh \, dx \, d t
\le C M r^2 N^{2\alpha} \left(1 + N^{-2\alpha s + 2s -1 }\right),
\end{equation}
with an implied constant that depends on $s,T$, but is independent of the initial data, and $N$. If we now choose $r = \Delta = 1/N$, so that $\Delta $ is a length at the grid-scale, then we find
\[
\int_0^T \int \fint_{B_{\Delta}(0)} 
|u_i^{\Delta}(x+h)-u_i^{\Delta}(x)|^2 
\, dh \, dx \, d t \le C M N^{2\alpha-2}(1+N^{-2\alpha s+2s-1}).
\]
We finally would like to determine a constant $\alpha$, which minimizes this last expression. This is achieved when the term in brackets is of order $1$, i.e. with  $\alpha = \frac{2s-1}{2s}$. With this choice of $\alpha$, we have $2\alpha -2 = \frac{-1}{s}$, and 
\[
\int_0^T \int \fint_{B_{\Delta}(0)} 
|u_i^{\Delta}(x+h)-u_i^{\Delta}(x)|^2 
\, dh \, dx \, d t
\le C M N^{2\alpha -2} = C M N^{\frac{-1}{s}} = C M \Delta^{1/s}.
\]
Let us finally sum over $i=1,\dots, m$. Thus, if $\mu^\Delta_t$ is an approximate statistical solution obtained from the spectral vanishing hyperviscosity method of order $s$, and initial data $\overline{\mu}$, with $\Delta = N^{-1}$, then it satisfies the following bound on the structure function \emph{at the grid scale}:
\begin{equation} \label{eq:weakBVdx}
S_{\Delta}^2(\mu^{\Delta}_t,T)
\le C M \Delta^{1/(2s)}.
\end{equation}
Inspection of \eqref{eq:weakBV} shows that also
\begin{equation*}
S_{r}^2(\mu^{\Delta}_t,T)
\le C M r^{1/(2s)},
\end{equation*}
if $r\le \Delta$.
\end{proof}
As in \cite{FLMW1} section 4.2, we have uniform estimates on the structure function at (or below) the grid scale. Large scale features are in any case independent of the resolution $\Delta$. However, we lack any information on the \emph{intermediate scales}, in between the two. To close this information gap, we follow \cite{FLMW1} and make an assumption on scaling of the structure function \eqref{eq:tasf} at intermediate scales. The resulting theorem is,
\begin{theorem} \label{thm:inertialrangeconv}
Consider the incompressible Euler equations with initial data $\bar{\mu}\in \mathcal{P}(L^2_x)$, such that $\supp(\bar{\mu}) \subset B_M$, with $B_M$ the ball of radius $M$ in $L^2_x$, for some $M>0$. Define the approximate statistical solution $\mu^\Delta_t$ by the Monte-Carlo algorithm \ref{alg:MC}. If the approximate statistical solutions $\mu^\Delta_t$ satisfy:
\begin{itemize}
\item Approximate scaling: For every $\ell > 1$, there exists a constant $0<\lambda_\ell \le 1/(2s)$, fixed $C>0$ possibly depending on the initial data, but independent of $\ell$ and the grid size $N$, such that
\[
S^2_{\ell\Delta}(\mu_t^\Delta, T)
\le 
C \ell^{\lambda_\ell} S^2_{\Delta}(\mu_t^\Delta,T), \quad (T>0).
\]
\end{itemize}
Then the approximate statistical solutions $\mu^\Delta_t$ converge (up to a subsequence still denoted by $\Delta$), as $\Delta \to 0$, to some $\mu_t\in L^1_t(\P)$.
\end{theorem}

\begin{proof}
By lemma \ref{lem:weakBV}, there exists a constant $C>0$, such that
\begin{equation}
    \label{eq:scal}
S^2_r(\mu_t^\Delta,T) \le \bar C r^{1/(2s)},
\end{equation}
for all $r\le \Delta$. If $r > \Delta$, then by the assumed approximate scaling property, we write $r = \ell \Delta$, with $\ell>1$, and obtain
\begin{align*}
S^2_r(\mu_t^\Delta,T)
&= S^2_{\ell \Delta} (\mu_t^\Delta,T)
\\
&\le C \ell^{\lambda_{\ell}} S^2_{\Delta} (\mu_t^\Delta,T)
\\
&\le C\bar C  \ell^{1/2s} \Delta^{1/(2s)} 
\\
&= C\bar C  r^{1/(2s)},
\end{align*}
for some constant $C\bar C>0$. The convergence now follows from theorem \ref{thm:compactconv}.
\end{proof}
\begin{remark}
The scaling assumption \eqref{eq:scal} can be interpreted as a weaker version of the scaling assumptions of Kolmogorov (see hypothesis H2, equation 6.3, page 75 of \cite{Frisch}) that was instrumental in the K41 theory for homogeneous, isotropic turbulence. We also note that the inequalities in \eqref{eq:scal} can accommodate intermittency in the form of deviations for the standard Kolmogorov determination of the exponent $1/3$ for the structure function \eqref{eq:tasf}.
\end{remark}
\subsection{Decay of energy spectrum.}
In this section, we will provide an alternative criterion to ensure convergence of probability measures with respect to the metric \eqref{eq:dtdef}. 

This criterion is motivated from well-known experimental and theoretical concepts in the study of turbulent flows and is based on the 
energy spectrum $E({u};K)$ ($K \in \mathbb{N}_0$) associated to a vector field ${u}$, defined as
\[
E({u};K) = \frac 12 \sum_{K-1< |k|\le K} |\widehat{u}(k)|^2.
\]
Note that the kinetic energy is obtained as a sum 
\[
\frac 12 
\int_D |{u}|^2 \, dx
= 
\sum_{K=1}^\infty E({u};K).
\]
Given a probability measure $\mu \in \P(L^2_x)$, let us similarly define:
\[
E(\mu;K) =  \int_{L^2_x} E(u,K) \, d\mu(u),
\]
so that $E(\delta_{u};K) = E(u;K)$, for $u\in L^2_x$. Finally, we denote by $E_T(\mu_t;K)$ the time-integrated energy spectrum
\[
E_T(\mu_t;K) = \int_0^T E(\mu_t;K) \, dt.
\]

It is an experimentally observed fact \cite{Frisch} that the typical energy spectrum of turbulent flows with a sufficiently strong dissipation mechanism at small scales, typically takes a shape similar to the one shown in Figure \ref{fig:typicalspectrum}: Visible are three parts of the energy spectrum. The left-most part (small $K$) corresponds to large-scale features for the flow, the middle part (intermediate $K$) is referred to as the inertial range, while the right-most part (large $K$) may be referred to as the dissipation range. The appearance of these three parts is heuristically explained as follows. Starting from initial data (with a sufficiently fast decay of the energy spectrum) initially fixes the large-scale features of the flow. Due to the non-linear nature of the evolution equation, these large-scale features decay to smaller scales, corresponding to energy cascading from small values of $K$ to larger values of $K$. While a satisfactory mathematical treatment of the precise nature of this energy cascade remains an outstanding challenge, there is evidence by physical reasoning and as well as from numerical and real-world experiments that typically the energy spectrum resulting from this cascade process satisfies at least an upper bound of the form $E(K) \lesssim K^{-\gamma}$, for some fixed $\gamma$ that is associated with the non-linearity. In the presence of a dissipative mechanism acting on small scale features of the flow, this ``free'' energy cascade to larger values of $K$ due to the non-linearity is finally interrupted by the dissipation. Thus, energy is dissipated at dissipative scales. 

From this heuristic point of view, we would expect the large-scale features to depend mostly on the initial data, while the decay of the energy spectrum at the largest values of $K$ can be controlled in a numerical approximation scheme by a suitable choice of the numerical dissipation. On the other hand, there is no a priori information on the decay of the spectrum in the intermediate, \emph{inertial} range. Hence, we make the following, rather natural, assumption,
\begin{assumption} \label{ass:scaling}
There exist $\beta>0$ and constant $C > 0$ such that the computed energy spectra with algorithm \ref{alg:MC} scale as,
\begin{equation}
\label{eq:esir}
E_T(\mu_t^\Delta,K)
\le C K^{-2\beta}, \qquad \forall \, \Delta > 0.
\end{equation}
\xqed{$\diamond$}
\end{assumption}
Under this assumption on the energy spectrum, we have the following convergence theorem,
\begin{theorem} \label{prop:inertialrange}
If $\mu_t^\Delta$ is obtained by the spectral viscosity method through algorithm \ref{alg:MC}, and if the energy spectra $E_T(\mu^\Delta_t;K)$ satisfy the inertial range assumption \ref{ass:scaling} with $\beta > 1/2$, then there exists a subsequence (not relabeled) $\Delta \to 0$ and a time-parametrized probability measure $\mu_t$, such that $\mu_t^\Delta \to \mu_t$ in $L^1_t(\P)$.
\end{theorem}
\begin{proof}
From Plancherel's identity and lemma \ref{lem:simpleest}, we have
\begin{align*}
  S^2_r(\mu_t;T)^2
  &=
  \int_0^T \int_{L^2_x} \fint_{B_r(0)} \int |u(x+h)-u(x)|^2 \, dx \, dh \, d\mu_t \, dt
  \\  
  &\lesssim \int_0^T \int_{L^2_x} \sum_k \min(|k|^2 r^2, 1) |\widehat{u}(k)|^2 \, d\mu_t \, dt
  \\
  &\sim r^2 \sum_{K \le 1/r} K^2 E_T(\mu_t;K) + \sum_{K> 1/r} E_T(\mu_t;K).
\end{align*}

Hence, based on the assumption \ref{ass:scaling}, we now obtain the estimate,
\begin{align*}
  S^2_r(\mu_t^\Delta;T)^2
  &\lesssim
  r^2 \sum_{K\le 1/r} K^2 K^{-2\beta}
  + \sum_{K>1/r} K^{-2\beta} \\
  &\sim
  r^{2}(1+r^{2\beta-3}) + r^{2\beta-1} \\
  &\sim r^{\min(2,2\beta-1)}, \quad \text{as }r\to 0.
\end{align*}
Therefore, the scaling assumption on the average energy spectrum leads to the uniform diagonal continuity:
\begin{equation}
    \label{eq:cessf}
E_T(\mu^\Delta_t, K) \lesssim K^{-2\beta} \quad \Rightarrow \quad S^2_r(\mu^\Delta_t;T) \lesssim r^{\beta-1/2}, 
\quad \text{if } 1 < 2\beta < 3.
\end{equation}

From theorem \ref{thm:compactconv}, we obtain compactness of the sequence $\mu^{\Delta}_t$.
\end{proof}

\begin{figure}[H]
\centering
\begin{subfigure}{.48\textwidth}
\begin{tikzpicture}
\pgftext{
\includegraphics[width=\textwidth]{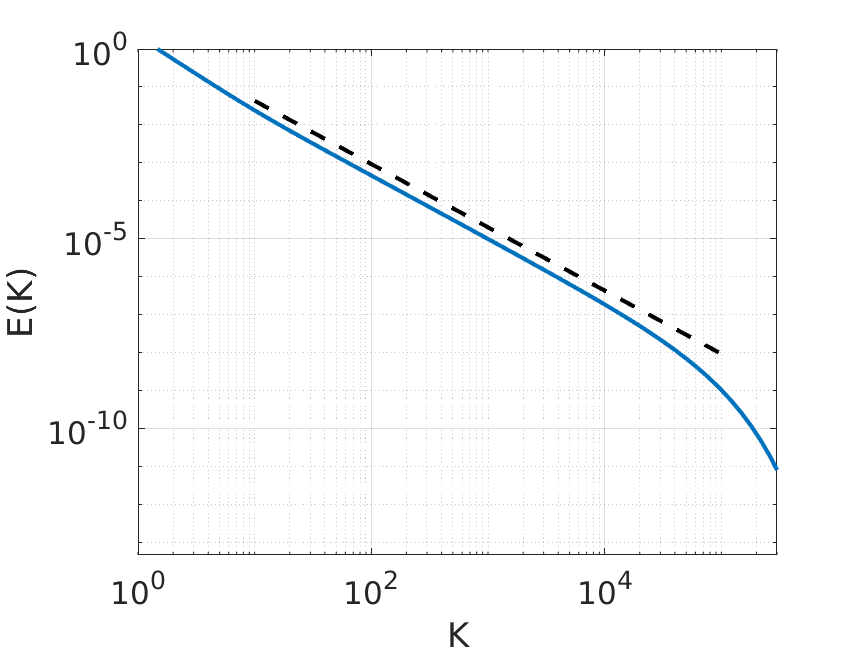}
};
\fill [red,opacity=.2] (1.1,1.97) rectangle (2.45,-1.59);
\fill [blue,opacity=.2] (-2.05,1.97) rectangle (-.8,-1.59);


\node[text opacity=1,fill=white,opacity=0.3,rectangle,draw,rounded corners] at (-1.45,-1.15) {\small small $K$};
\node[text opacity=1,text width=1cm,fill=white,opacity=0.3,rectangle,draw,rounded corners] at (.15,-1.15) {\small inertial range};
\node[text opacity=1,text width=1cm,fill=white,opacity=0.3,rectangle,draw,rounded corners] at (1.75,-1.15) {\small dissip. range};

\node[text opacity=1] at (.4,1.2) {$\sim K^{-\gamma}$};
\end{tikzpicture}
\caption{Energy spectrum $E(K)$}
\end{subfigure}
\begin{subfigure}{.48\textwidth}
\begin{tikzpicture}
\pgftext{
\includegraphics[width=\textwidth]{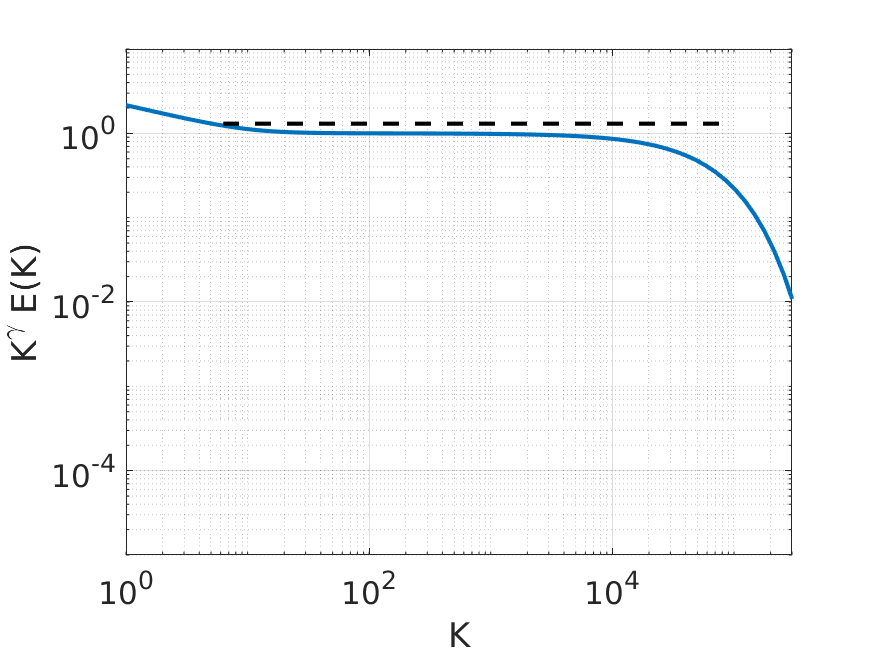}
};
\fill [red,opacity=.2] (1.1,1.93) rectangle (2.45,-1.55);
\fill [blue,opacity=.2] (-2.15,1.93) rectangle (-.8,-1.55);
\node[text opacity=1] at (.2,1.) {$\sim \mathrm{const.}$};

\node[text opacity=1,fill=white,opacity=0.3,rectangle,draw,rounded corners] at (-1.5,-1.1) {\small small $K$};
\node[text opacity=1,text width=1cm,fill=white,opacity=0.3,rectangle,draw,rounded corners] at (.15,-1.1) {\small inertial range};
\node[text opacity=1,text width=1cm,fill=white,opacity=0.3,rectangle,draw,rounded corners] at (1.8,-1.1) {\small dissip. range};

%

\end{tikzpicture}
\caption{Compensated E'spectrum $K^\gamma E(K)$}
\end{subfigure}
\caption{Typical energy spectrum for turbulent flows}
\label{fig:typicalspectrum}
 \end{figure}
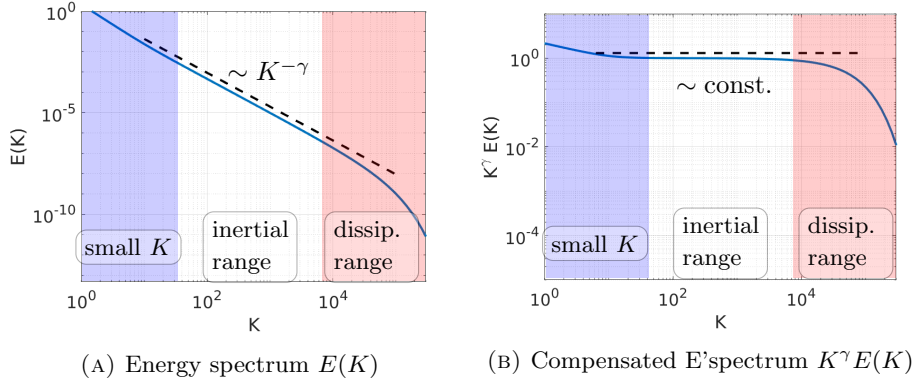
\begin{remark} \label{rem:compespectrum}
As indicated in Figure \ref{fig:typicalspectrum} (B), a convenient way to check the scaling assumption \ref{ass:scaling} in practice is to consider the \define{compensated energy spectrum}, which is defined as $K^\gamma E(K)$, where $\gamma$ is the (proposed) scaling exponent in the inertial range. proposition \ref{prop:inertialrange} says that if there exists $\gamma > 1$, such that the compensated energy spectrum $K^\gamma E_T(\mu_t^\Delta;K)$ is uniformly bounded by a constant, and independently of $\Delta$, then $\{\mu^\Delta_t \, |\, \Delta>0\}$ is compact in $L^1(\P)$.
\end{remark}

\begin{remark}
If $d = 3$ and $p=2$, then Kolmogorov's theory states that for fully developed turbulence $S_r^2 \sim r^{1/3}$. Based on our estimate, this requires $\beta = \frac 56$. So that the (expected) energy spectrum is $E(K) \sim K^{-2\beta} \sim K^{-5/3}$. Such an assumed scaling is consistent with many real, as well as numerical, experiments reported in the literature, and is sufficient for compactness in the space of probability measures $L^1_t(\P)$ (cp. proposition \ref{prop:inertialrange}).
\end{remark}

\subsection{Lax-Wendroff type theorem}

We have used a compactness argument to show that under some reasonable hypotheses on the approximations, numerical solutions computed by the spectral hyper-viscosity converge to a limiting time-parametrized probability measure. In this section, we show that such a limit necessarily is a statistical solution of the incompressible Euler equations in the sense of definition \ref{def:statsol}.

\begin{theorem}[Lax-Wendroff type theorem] \label{thm:LxWendroff}
Let $\mu^{\Delta}_t$ be computed by the spectral hyper-viscosity scheme with initial data $\bar{\mu}$, and assume $\mu^\Delta_t \to \mu_t$ in $L^1_t(\P)$, as $\Delta \to 0$. Then $\mu_t$ is a statistical solution of the incompressible Euler equations with initial data $\bar{\mu}$.
\end{theorem}

\begin{proof}
Fix $k\in \mathbb{N}$. Let ${\phi}_1, \dots, {\phi}_k \in C_c^\infty(D\times [0,\infty))$ be given solenoidal test functions. Set ${\phi} := {\phi}_1 \otimes \dots \otimes {\phi}_k$ and denote $\nu^k = \nu^k_{x_1,\dots,x_k,t}$. Let ${u}^\Delta$ be obtained from the spectral method, with initial data $\bar{{u}}$. Let us denote $({u},{\phi}) := \int_{D} {u}\cdot{\phi} \, dx$. Then, as a consequence of lemma \ref{lem:timeregularity}, we can write
\[
\frac{d}{dt} ({u}^\Delta,{\phi}_i) = ({u}^\Delta, \partial_t {\phi}_i) + ({u}^\Delta\otimes {u}^\Delta, \nabla \phi_i) + (E^\Delta,{\phi}_i),
\]
where there exists $L>0$ independent of $\Delta$ and the initial data $\bar{{u}}$, such that the error term $E^\Delta$ satisfies $\Vert E^\Delta \Vert_{H^{-L}} \le C\Delta(1+\Vert \bar{{u}} \Vert_{L^2_x}^2)$. Taking the product over $i=1,\dots, k$, we find
\[
\frac{d}{dt} \prod_{i=1}^k ({u}^\Delta,{\phi}_i) = \sum_{i=1}^k \left[\prod_{j\ne i} ({u}^\Delta, {\phi}_j)\right] \left\{({u}^\Delta, \partial_t{\phi}_i) + ({F}({u}^\Delta),\nabla \phi_i) + (E^\Delta,\phi_i) \right\},
\]
where ${F}({u}) := {u}\otimes {u}$. Recognizing the special structure of the empirical measure $\mu^{\Delta}_t$ \eqref{eq:emeas} as a convex combination, denoting by $\nu^{k,\Delta} = \nu^{k,\Delta}_{x_1,\dots,x_k,t}$ the k-point correlation measure corresponding to $\mu^\Delta_t$, we obtain from the above identity that, 
\begin{gather} \label{eq:approxweak}
\begin{aligned}
\int_0^T \int_{D^k} &\langle \nu^{k,\Delta}, {\xi}_1 \otimes \dots \otimes {\xi}_k \rangle : \partial_t {\phi}
 \\
 &\quad 
 + \sum_i \langle \nu^{k,\Delta}, {\xi}_1 \otimes \dots \otimes {F}({\xi}_i) \otimes \dots \otimes  {\xi}_k \rangle : \nabla_{{x}_i} {\phi}
\, dx \, dt
\\
&\quad + \int_{D^k} \langle \bar{\nu}^{k,\Delta}, {\xi}_1 \otimes \dots \otimes {\xi}_k \rangle : {\phi}(x,0) \, dx
\\
&\quad =
\int_0^T \int_{L^2_x} \sum_{i=1}^k \left[\prod_{j\ne i} ({u}^\Delta, {\phi}_j)\right] (E^\Delta, \phi_i) \, d\mu^\Delta_t \, dt.
\end{aligned}
\end{gather}
The right-hand side can be bounded by 
\[
\int_0^T \int_{L^2_x}  \Vert {u}^\Delta\Vert_{L^2_x}^{k-1} \sum_{i=1}^k  \prod_{j\ne i}\Vert {\phi}_j\Vert_{L^2_x}  \Vert E^\Delta \Vert_{H^{-L}_x} \Vert \phi_i\Vert_{H^L_x} \, d\mu^\Delta_t \, dt,
\]
which, by lemma \ref{lem:timeregularity} is further bounded by
\[
\le C({\phi},k) \Delta \int_0^T \int_{L^2_x} (1+\Vert {u}^\Delta\Vert_{L^2_x}^2)^{k} \, d\mu^\Delta_t({u})\, dt.
\]
Note that if $\bar{\mu}$ is supported on $\overline{B_M(0)}\subset L^2_x$, then it follows that $\mu^\Delta_t$ is supported on $\overline{B_M(0)}$, as well. This is a consequence of the a priori $L^2$-bound \eqref{eq:Eest}. Hence the error term in equation \eqref{eq:approxweak} is in this case bounded by $C\Delta$, where $C = C(\phi,k,M,T)$ is a constant independent of $\Delta$.

Let us also note that the terms on the left-hand side of \eqref{eq:approxweak} converge strongly in $L^1_{t,x}$ as $\Delta \to 0$. Indeed, it is not difficult to see that all terms on the left -hand side, e.g.
\begin{align*}
g(t,x,\xi) :=
\left({\xi}_1 \otimes \dots \otimes {F}({\xi}_i) \otimes \dots \otimes  {\xi}_k \right) : \nabla_{{x}_i} {\phi}(x,t),
\end{align*}
are admissible observables in the sense of \eqref{eq:obs}. For such observables, the $L^1_{t,x}$-convergence of 
\[
\langle \nu^{k,\Delta}_{t,x}, g(t,x,\xi) \rangle 
\to 
\langle \nu^{k}_{t,x}, g(t,x,\xi) \rangle, \quad \text{as }\Delta \to 0,
\]
has been established in theorem \ref{thm:compactconv}. The same holds true for the other two terms on the left-hand side.

Passing to the limit $\mu^\Delta_t \to \mu_t$, it thus follows that 
\begin{gather*} 
\begin{aligned}
\int_0^T \int_{D^k} &\langle \nu^{k}_{t,x}, {\xi}_1 \otimes \dots \otimes {\xi}_k \rangle : \partial_t {\phi}
 \\
 &\quad 
 + \sum_i \langle \nu^{k}_{t,x}, {\xi}_1 \otimes \dots \otimes {F}({\xi}_i) \otimes \dots \otimes  {\xi}_k \rangle : \nabla_{{x}_i} {\phi}
\, dx \, dt
\\
&\quad + \int_{D^k} \langle \bar{\nu}^{k}_{x}, {\xi}_1 \otimes \dots \otimes {\xi}_k \rangle : {\phi}(x,0) \, dx
 =
0.
\end{aligned}
\end{gather*}
The fact that $\mu_t$ is concentrated on incompressible vector fields follows immediately from the corresponding property of the approximations $\mu^\Delta_t$ (cp. lemma \ref{lem:charincomp}). Furthermore, from proposition \ref{prop:timereg}, it also follows that the limit $\mu_t$ is time-regular. This finishes the proof that $\mu_t$ is a statistical solution of the incompressible Euler equations with initial data $\bar{\mu}$.
\end{proof}

\begin{remark}
It is straightforward to show that if $\mu^\Delta_t$ are generated from the spectral hyper-viscosity scheme \eqref{eq:spectralvisc}, and if they satisfy the assumptions of theorem \ref{thm:LxWendroff}, the limit $\mu_t$ is in fact a dissipative statistical solution in the sense of definition \ref{def:dissipative}.
\end{remark}
\section{Numerical Experiments}
\label{sec:numerical}
In this section, we will present a suite of numerical experiments to demonstrate the effectiveness of the Monte Carlo algorithm \ref{alg:MC} in computing statistical solutions of the incompressible Euler equations. 
\subsection{Implementation} 

For our numerical experiments, we use the implementation of the spectral hyper-viscosity scheme \eqref{eq:spectralvisc} provided by the SPHINX code, which was first presented in \cite{LeonardiPhD}. In SPHINX, the non-linear advection term is implemented using the $O(N^2 \log N)$-costly fast Fourier-transform. Aliasing is avoided via the use of a padded grid, as e.g. described in \cite{LeonardiPhD,LM2019}. The spectral scheme is implemented based on the primitive variable formulation. 

For the numerical experiments reported below, we use a spectral viscosity operator of order $s=1$ (cp. equation \eqref{eq:spectralvisc}), with $\epsilon_N = \epsilon/N$, $\epsilon = 1/20$ unless otherwise stated. The Fourier multiplier $Q_N$ is chosen with Fourier coefficients
\[
\widehat{Q}_k = 
\begin{cases}
1-N/{|k|^2} & |k|\ge \sqrt{N}, \\
0, & \text{otherwise}.
\end{cases}
\]
corresponding to $m_N = \sqrt{N}$.

For each sample, SPHINX solves the following system of ODEs for the Fourier coefficients $\hat{u}_k(t)$, $|k|_\infty \le N$ of $u(x,t) = \sum_{|k|_\infty \le N} \widehat{u}_k(t) e^{ikx}$:

\begin{gather} \label{eq:SPHINX}
\left\{
\begin{aligned}
\frac{d}{dt} \widehat{u}_k 
&= 
-ik \cdot \left(I - \frac{k\otimes k}{|k|^2} \right) \cdot \widehat{\left(u\otimes u\right)}_k
- \frac{\epsilon}{N} \max\left(|k|^2-N,0\right) \, \widehat{u}_k
,
\\
\widehat{u}_k(0) &= \widehat{\overline{u}}_k,
\end{aligned}
\right.
\end{gather}
for $|k|_\infty \le N$, $k\ne 0$. Here $\widehat{\overline{u}}_k$ denote the Fourier coefficients of the initial data. The multiplication by the matrix $(I - {(k\otimes k)}/{|k|^2} )$ implements the Leray projection onto divergence-free vector fields. The dot product with $ik\cdot (\ldots)$ corresponds to taking the divergence in Fourier space. We note that the $0$-th Fourier component of $u$ is constant in time, reflecting conservation of momentum. In SPHINX, this component is set equal to $\widehat{u}_0 \equiv 0$. The above system of ODEs \eqref{eq:SPHINX} is integrated in time using an adaptive explicit third-order Runge-Kutta method. 

Although the theory of section \ref{sec:spectralvisc} is valid for both two and three space dimensions
and the SPHINX code is available for both cases, we restrict our focus to two space dimensions in this section, on account of affordable computational costs. 
\subsection{Flat Vortex Sheet.} Vortex sheets occur in many models in physics and are an important test bed for numerical experiments for the Euler equations \cite{LM2015} and references therein. We first consider a randomly perturbed version of the \emph{flat vortex sheet} that corresponds to the following initial data also considered in \cite{LM2015},
\subsubsection{Initial data}
Given a smoothing parameter $\rho > 0$, and a parameter $\delta\ge 0$ (measuring the size of the random perturbation of the interface), this vortex sheet initial data is of the form 
\begin{equation}
\label{eq:fvs1}
\overline{u}^{\rho,\delta}(x) 
=
\mathbb{P}(U^\rho(x_1,x_2 + \sigma_\delta(x_1))),
\end{equation}
where $\mathbb{P}$ denotes the Leray projection, $U^\rho(x) = (U_1^\rho(x),U_2^\rho(x))$ is the following smoothened flat vortex sheet initial data:
\[
U^\rho(x)
:=
\begin{cases}
\tanh\left(\frac{x_2-1/4}{\rho}\right), & (x_2 \le 1/2), \\
\tanh\left(\frac{3/4-x_2}{\rho}\right), & (x_2 > 1/2),
\end{cases}
\qquad
U_2^\rho(x) = 0.
\]
and $\sigma_\delta(x)$ is a random function, which for a given (random) choice of parameters $\alpha_1, \dots, \alpha_q\in (0,\delta)$, $\beta_1, \dots, \beta_q \in [0,2\pi)$, is defined by
\begin{align} \label{eq:randompert}
\sigma_\delta(x_1) 
= 
\sum_{k=1}^q \alpha_k \sin(2\pi x_1-\beta_k).
\end{align}
We will also consider the discontinuous case of initial data that are obtained in the limit $\rho \to 0$ resulting in 
\[
U^0_1(x)
:=
\begin{cases}
+1, & (1/4 < x_2 \le 3/4), \\
-1, & (\text{otherwise}),
\end{cases}
\qquad
U_2^0(x) = 0.
\]
For our simulations, we fix $q=10$ modes for the perturbations. The coefficients $\alpha_k$ are drawn independently, uniformly in $(0,1)$, and then multiplied by $\delta$. The coefficients $\beta_k$ are i.i.d., with a uniform distribution on $[0,2\pi)$. The initial data for the statistical solution $\bar{\mu}^{\delta}_\rho \in \P(L^2_x)$ is defined as the law of these random perturbations. It depends on the two parameters $\rho \ge 0$, $\delta \ge 0$. While $\rho$ controls the smoothness of the initial data, $\delta$ measures the amplitude of the perturbation. We fix $\delta = 0.025$ in the following and consider different values of $\rho$. Note that the choice $\rho=0$ corresponds to an initial measure supported on discontinuous flows with a very sharp transition (see figure \ref{fig:sphsldiscont_sample} (A,B) for realizations (samples) of this initial data). In figure \ref{fig:sphsldiscont_sample} (C,D), we present the initial mean and variance that correspond to the random variations of the initial interface location. 

\begin{figure}[H]
	\begin{subfigure}{.33\textwidth}
		\centering
		\includegraphics[width=\textwidth]{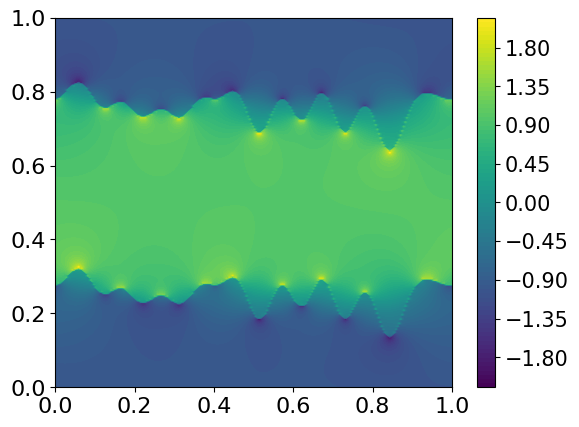}
		\caption{$u_1$-component}
	\end{subfigure}%
	\begin{subfigure}{.33\textwidth}
		\centering
		\includegraphics[width=\textwidth]{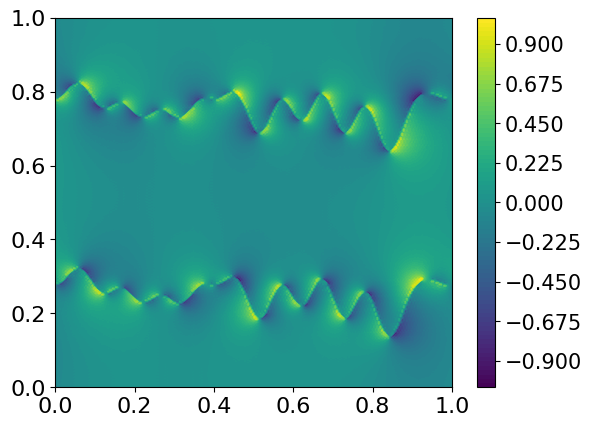}
		\caption{$u_2$-component}
	\end{subfigure}%
	
	\begin{subfigure}{.33\textwidth}
		\centering
		\includegraphics[width=\textwidth]{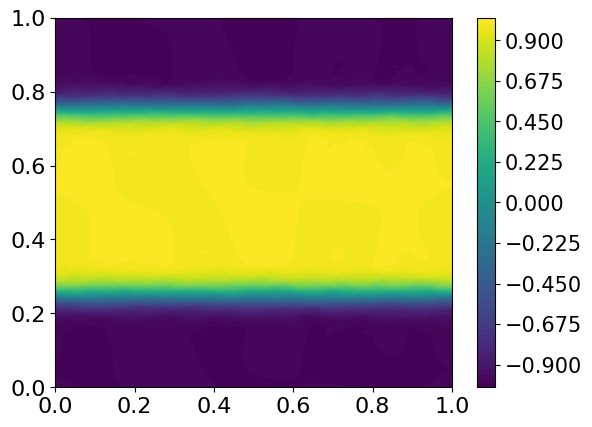}
		\caption{mean: $u_1$-comp.}
	\end{subfigure}%
	\begin{subfigure}{.33\textwidth}
		\centering
		\includegraphics[width=\textwidth]{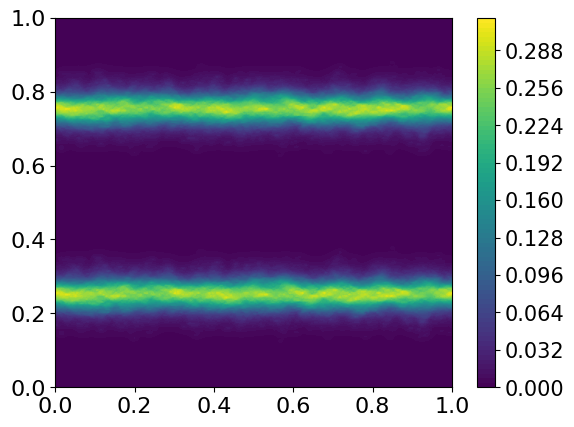}
		\caption{variance: $u_1$-comp.}
	\end{subfigure}%
	
	\caption{Initial data for the perturbed discontinuous flat vortex sheet ($\rho = 0$), samples for $u_{1,2}$, and mean and variance of $u_1$}
	\label{fig:sphsldiscont_sample}
\end{figure}
Clearly when $\rho > 0$, the corresponding initial data for every sample is smooth. Consequently, smooth solutions of \eqref{eq:Eulerfull} are well-posed and the spectral viscosity method converges to this solution as $N \rightarrow \infty$ \cite{BardosTadmor}. However for $\rho = 0$, which corresponds to the case of a discontinuous vortex sheet, there are no well-posedness results even for weak solutions, as the vorticity corresponding to the initial datum (for each sample) is a sign changing measure and does not belong to the \emph{Delort Class}. In \cite{LM2015}, the authors had presented multiple numerical experiments to illustrate the approximate solutions, computed with a spectral viscosity method, may not converge (or converge too slowly to be of practical interest) for individual samples (see figures 5 and 6 of \cite{LM2015}). Hence, it would be interesting to study if approximate statistical solutions, generated by algorithm \ref{alg:MC} converge in this case. 
\subsubsection{Structure functions and Compensated Energy spectra.} 
The convergence theorem \ref{thm:inertialrangeconv}, based on the compactness theorem \ref{thm:timecompact}, provides us with verifiable criteria to check convergence of algorithm \ref{alg:MC}. In particular, we need to check certain decay conditions on the structure function \eqref{eq:tasf} for small correlation lengths. To this end, we consider the following instantaneous version of the structure function \eqref{eq:tasf},
\begin{equation}
    \label{eq:sfin}
S^{2,\Delta}_{r,t}(\mu_t) := 
\left(
\int_{L^2_x} 
\int_D 
\fint_{B_r(0)} |u(x+h)-u(x)|^2 
\, dh \, dx \, d\mu^{\Delta}_t(u)
\right)^{1/2},
\end{equation}
Note that the above is a formal definition and it can be made rigorous in terms of the time-dependent correlation measures. It is much simpler to compute the instantaneous quantity \eqref{eq:sfin} than the time-averaged version \eqref{eq:tasf}. 

Our objective is to check whether the structure function \eqref{eq:tasf}, or rather its instantaneous version \eqref{eq:sfin}, decays (uniformly in resolution $\Delta$) as $r \rightarrow 0$. Such a decay would automatically imply convergence of the approximations to a statistical solutions by theorems \ref{thm:timecompact} and \ref{thm:LxWendroff}.

\begin{figure}[H]
\begin{subfigure}{.45\textwidth}
\includegraphics[width=\textwidth]{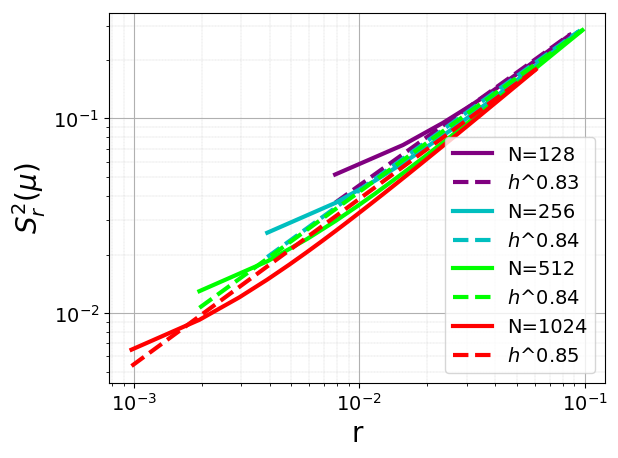}
\caption{$\rho = 0.1$}
\end{subfigure}
\begin{subfigure}{.45\textwidth}
\includegraphics[width=\textwidth]{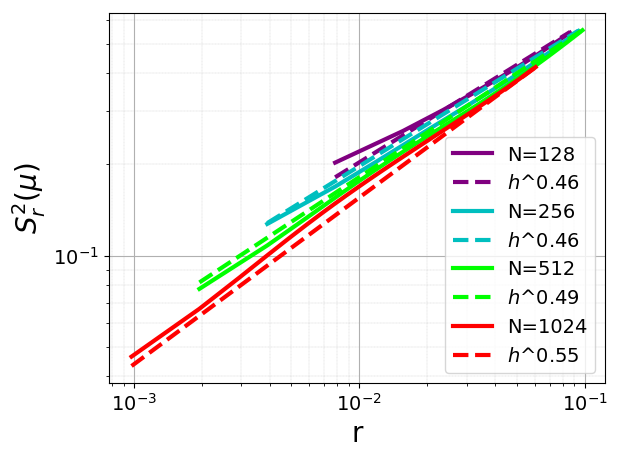}
\caption{$\rho = 0.0$}
\end{subfigure}
\caption{Instantaneous structure function \eqref{eq:sfin} vs correlation length $r$ for different resolutions ($N \sim \Delta^{-1}$) for different values of smoothness parameter $\rho$, at $t=0.4$}
\label{fig:fvs_sf}
\end{figure}

Clearly if $\rho > 0$ in \eqref{eq:fvs1}, the spectral viscosity method converges to the unique classical solution as $\Delta \rightarrow 0$. Moreover, a straightforward calculation shows that the structure function \eqref{eq:sfin} should scale as,
\begin{equation}
    \label{eq:sfins}
    S^{2,\Delta}_{r,t}(\mu_t) \approx r, \quad \forall \Delta, t.
\end{equation}

This is indeed verified from figure \ref{fig:fvs_sf} (A) where we plot the structure function \eqref{eq:sfin} at $t=0.4$ and $\rho = 0.1$ for different values of the mesh parameter. We see from this figure that $S^{2,\Delta}_{r,0.4}(\mu_t) \approx r^{0.9}$, at fine resolutions, which is very close to the expected value of $1$ for the scaling exponent of the structure function. 

On the other hand, for $\rho = 0$, corresponding to the discontinuous flat vortex sheet, the lack of smoothness inhibits us from inferring a particular form of decay of \eqref{eq:tasf} (or \eqref{eq:sfin}) a priori.

At least initially, calculations in \cite{LanthalerThesis} imply that  $ S^{2,\Delta}_{r,0}(\bar{\mu}) \sim r^{\frac{1}{2}}$, for the discontinuous flat vortex sheet. Surprisingly, we find from figure \ref{fig:fvs_sf} (B) that at fine resolutions,  $S^{2,\Delta}_{r,0.4}(\mu_t) \approx r^{0.52}$, which agrees with the decay of the structure function of the initial data. Although we do not present the results there, we observe that the structure function \eqref{eq:sfin} scales as $r^{\theta_t}$, with $\theta_t \geq 0.5$ for all $t$. This implies an uniform decay of the structure function \eqref{eq:tasf} and convergence of the approximations to a statistical solution of the Euler equations \eqref{eq:Eulerfull}, even for this case of discontinuous vortex sheet data. 
Note that the computed structure functions \eqref{eq:sfin} in figure \ref{fig:fvs_sf} clearly satisfy the approximate scaling hypothesis \eqref{eq:scal} and thus imply convergence through theorem \ref{thm:inertialrangeconv}. 
\begin{figure}[htbp]
\begin{subfigure}{.45\textwidth}
\includegraphics[width=\textwidth]{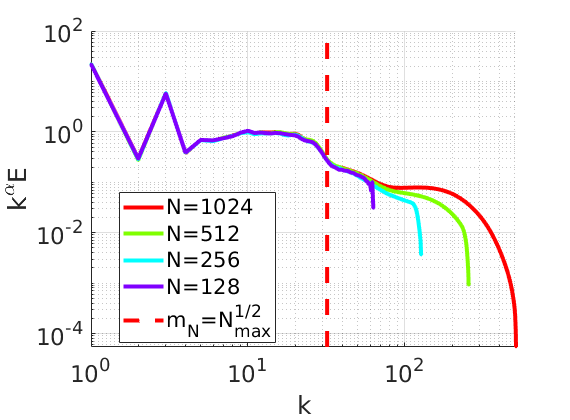}
\caption{$\rho = 0.1,\gamma = 3$}
\end{subfigure}
\begin{subfigure}{.45\textwidth}
\includegraphics[width=\textwidth]{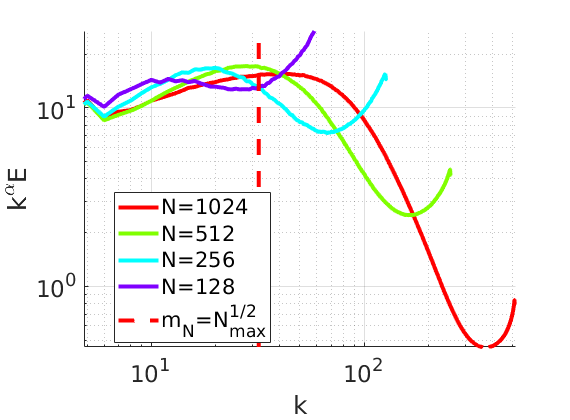}
\caption{$\rho=0.0,\gamma = 2$}
\end{subfigure}
\caption{The instantaneous compensated energy spectrum ${\mathcal C}^{\Delta}_{\gamma,t}(\mu_t;K)$ \eqref{eq:ces} for the flat vortex sheet, at time $t=0.4$. Note different values of $\gamma$ for the smooth and discontinuous vortex sheets}
\label{fig:fvs_ces}
\end{figure}

An alternative criterion for convergence of statistical solutions is provided by the energy spectrum decay in the inertial range \eqref{eq:esir}. To check whether this criterion is satisfied, we follow theorem \ref{prop:inertialrange} and compute the following instantaneous \emph{compensated energy spectrum},

\begin{equation}
    \label{eq:ces}
    {\mathcal C}^{\Delta}_{\gamma,t}(\mu_t;K) := K^{\gamma} E(\mu_t,K).
\end{equation}
Following the arguments in the proof of theorem \ref{prop:inertialrange}, we can relate the decay of the instantaneous energy spectrum to the corresponding decay of the structure function \eqref{eq:sfin} by a direct analogue of \eqref{eq:cessf}.

For $\rho = 0.1$ in \eqref{eq:fvs1}, we plot the compensated energy spectrum ${\mathcal C}^{\Delta}_{3,0.4}(\mu_t;K)$ for all $K$ and at time $t=0.4$, with compensating factor $\gamma = 3$ in figure \ref{fig:fvs_ces} (A). Note that this choice of $\gamma$ is consistent with a decay exponent of $1$ for the structure function in \eqref{eq:cessf}, i.e. $S_r^2(\mu_t^\Delta;T)\lesssim r$. We observe from this figure that as expected for this case, the compensated energy spectrum is clearly bounded and in fact, decays faster than the expected rate for the entire range of wave numbers. 

On the other hand, we plot the compensated energy spectrum ${\mathcal C}^{\Delta}_{2,0.4}(\mu_t;K)$ \eqref{eq:ces} for the discontinuous flat vortex sheet case, i.e. $\rho = 0$ in \eqref{eq:fvs1}, in figure \ref{fig:fvs_ces} (B). In this case, we expect from the structure function computations (see figure \ref{fig:fvs_sf}(B)) that the instantaneous structure function decays with an exponent of $\approx 0.5$. From \eqref{eq:cessf}, we see that this corresponds to the choice of $\gamma = 2$ as the exponent of compensation in \eqref{eq:ces}. Moreover, in figure \ref{fig:fvs_ces} (B), we also plot the line corresponding to wave number $m_N \approx \sqrt{N}$, which for the spectral viscosity method \eqref{eq:spectralvisc} represents the wave number after which the spectral viscosity is activated and hence, demarcates the separation between inertial and dissipation ranges. We observe from figure \ref{fig:fvs_ces} (B) that the compensated energy spectrum is clearly uniformly bounded (in terms of the resolution $\Delta$) for the whole inertial range and for all resolutions $\Delta$ barring the coarsest resolution, and decays fast in the dissipation range, although there is a slight kink upwards at the very end of the dissipation range, almost at the grid scale. This might be attributed to numerical errors, which are dominant at this range.  Translating these results to the energy spectrum, we see that the spectrum decays as $K^{-2}$ in the inertial range uniformly with respect to resolution. Hence, according to theorems \ref{prop:inertialrange} and \ref{thm:LxWendroff}, the sequence of approximations will converge to a statistical solution of \eqref{eq:Eulerfull}. 
\subsubsection{Convergence in Wasserstein Metrics.}
Given the computational results on the structure function and the compensated energy spectra, results in section \ref{sec:spectralvisc} clearly imply convergence of the approximations $\mu^{\Delta}_t$, generated by the Monte Carlo algorithm \ref{alg:MC} to a statistical solution of the incompressible Euler equations. Moreover from the discussion in section \ref{sec:timecompact}, we should observe with respect to the following \emph{Cauchy rates}:
\begin{equation}
    \label{eq:met1}
d_T(\mu^{\Delta}_t,\mu^{\Delta/2}_t) = \int_0^T W_1(\mu^{\Delta}_t,\mu^{\Delta/2}_t) \, dt.
\end{equation}
Unfortunately, the calculation of the Wasserstein distance between probability measured defined on high-dimensional (or indeed $\infty$-dimensional) spaces is a highly non-trivial issue, which we cannot tackle with present computational resources.

On the other hand, one can compute finite-dimensional \emph{marginals} of \eqref{eq:met1} by utilizing the complete characterization of $L^1_t(\P)$ in terms of \emph{correlation measures} as given in theorem \ref{thm:duality}. Following \cite{FLMW1} (theorem 5.7), one can prove that,
\begin{equation}
    \label{eq:met2}
\int_{D^k} W_1(\nu^{\Delta,k}_{t,x}, \nu^{\Delta/2,k}_{t,x})\, dx 
\le 
C_k W_1(\mu_t^\Delta,\mu_t^{\Delta/2}), \quad \text{a.e. } t
\end{equation}
Here, $k \geq 1$ and $\nu^{\Delta,k}_{t,x}$ is the $k$-th correlation marginal corresponding to the approximate statistical solution $\mu^{\Delta}_t$. Note that we consider instantaneous versions of the Wasserstein metric \eqref{eq:met1} for reasons of computational convenience. 

\begin{figure}[H]
\begin{subfigure}{.3\textwidth}
\includegraphics[width=\textwidth]{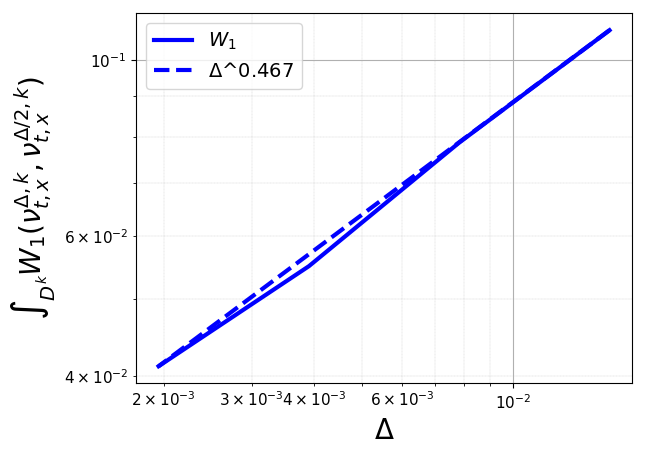}
\caption{$k=1$}
\end{subfigure}
\begin{subfigure}{.3\textwidth}
\includegraphics[width=\textwidth]{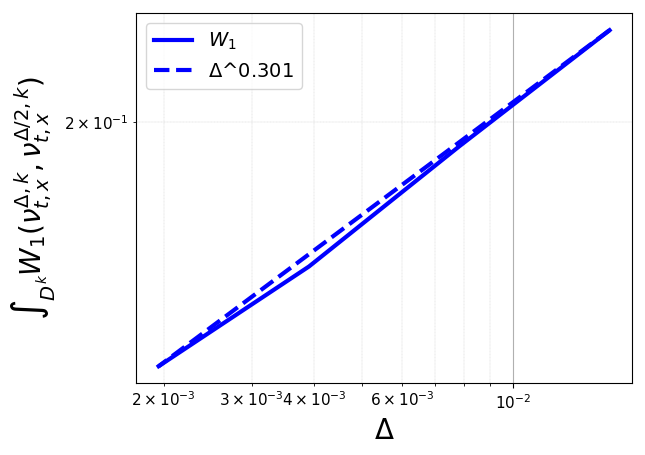}
\caption{$k=2$}
\end{subfigure}
\begin{subfigure}{.3\textwidth}
\includegraphics[width=\textwidth]{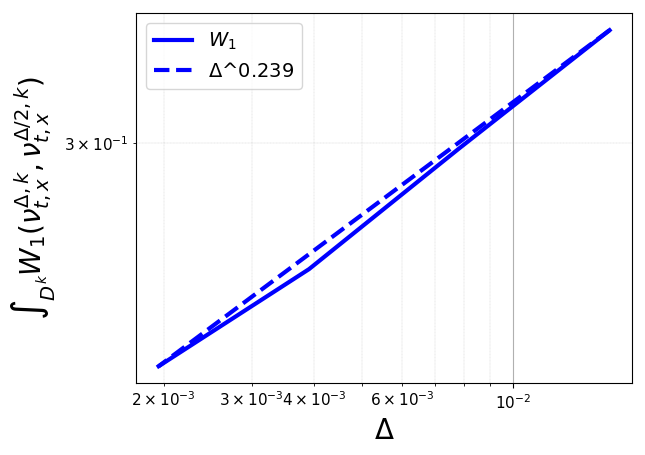}
\caption{$k=3$}
\end{subfigure}
\caption{The Wasserstein distances between correlation marginals $\int_{D^k} W_1(\nu^{\Delta,k}_{t,x}, \nu^{\Delta/2,k}_{t,x})\, dx$ for $k=1,2,3$, at time $t=0.4$ with respect to resolution} 
\label{fig:fvs_wass}
\end{figure}

We remark that computing the Wasserstein distances $W_1(\nu^{\Delta,k}_{t,x}, \nu^{\Delta/2,k}_{t,x})$ for small $k$ is much more tractable. We have computed these Wasserstein distances using the algorithm of \cite{Bonneel2011} (as implemented in \cite{Wass}) and the corresponding results for $k=1,2,3$, at time $t=0.4$ for the discontinuous flat vortex sheet, i.e. $\rho = 0$ in \eqref{eq:fvs1} are presented in figure \ref{fig:fvs_wass}. As seen from this figure, we observe a clear convergence of these Wasserstein distances (in the Cauchy sense as in \eqref{eq:met2}) for the one-point, two-point and three-point correlation measures, albeit at a slow rate for the second and third correlation marginals. This, together with the results on the structure function and compensated energy spectra, provides considerable evidence that the approximate statistical solutions, generated by algorithm \ref{alg:MC}, converge to a statistical solution of \eqref{eq:Eulerfull}. Moreover, given theorem \ref{thm:compactconv}, results shown in figure \ref{fig:fvs_wass} establish convergence with respect to any admissible observable in the sense of \eqref{eq:obs}, corresponding of one-point, two-point and three-point statistical quantities of interest. These include mean, variance, structure functions, energy spectra as well as three-point correlation functions. 

\subsection{Sinusoidal vortex sheet}
\label{sec:sinvortex}
In this section, we will consider a random perturbation of the so-called sinusoidal vortex sheet, i.e. the initial vorticity is concentrated on a sine curve. This test case was extensively studied in a recent paper \cite{LM2019} in the context of numerical approximation of weak solutions (in Delort class) of the two-dimensional incompressible Euler equations. 
\subsubsection{Initial Data.}
We fix a sinusoidally perturbed vortex sheet, where the initial vorticity is a Borel measure of the form
\[
{\omega}_0 = \delta(x-\Gamma) - \int_{\T^2} \, d\Gamma,
\]
such that $\int_{\T^2} {\omega}_0 \, dx = 0$, and up to a constant, ${\omega}_0$ is uniformly distributed along a curve $\Gamma$, which is defined as the graph:
\[
\Gamma = 
\{
(x,y) \, | \, y = d \sin(2\pi x), \; x \in [0,1] 
\}.
\]
We chose $d = 0.2$ for our simulations. 

\begin{figure}[H]
	
	\begin{subfigure}{.3\textwidth}
		\centering
		\includegraphics[width=\textwidth]{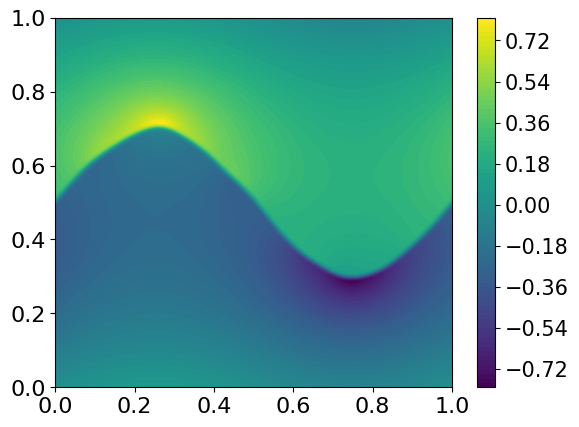}
		\caption{Sample}
	\end{subfigure}%
	\begin{subfigure}{.33\textwidth}
		\centering
		\includegraphics[width=\textwidth]{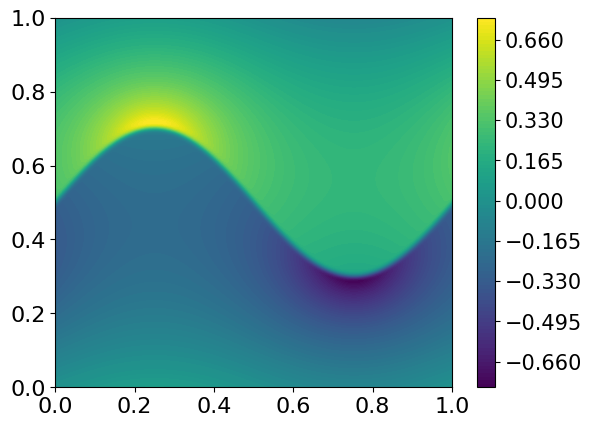}
		\caption{Mean}
	\end{subfigure}%
	\begin{subfigure}{.33\textwidth}
		\centering
		\includegraphics[width=\textwidth]{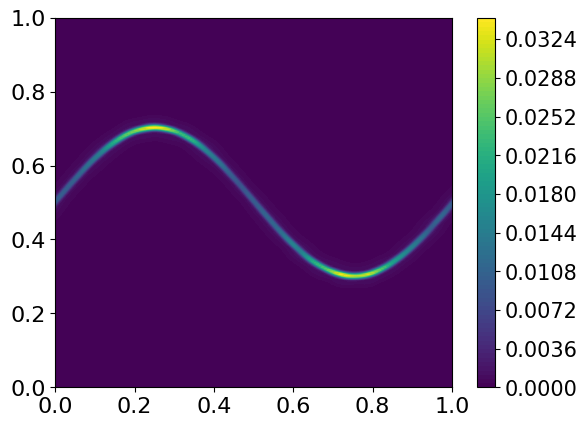}
		\caption{Variance}
	\end{subfigure}
	
	\caption{Initial conditions for the horizontal velocity $u_1$ for the sinusoidal vortex sheet.}
	\label{fig:svs_init}
\end{figure}

The numerical initial data is obtained from the mollification of this initial data with a parameter $\rho>0$. As a mollifier, we consider the third-order B-spline
\[
\psi(r) := \frac{80}{7\pi} \left[(r+1)^3_{+} - 4(r+1/2)^3_{+} + 6r^3_{+} - 4(r-1/2)^3_{+} + (r-1)^3_{+}\right].
\]
Next, we define $\psi_\rho(x) = \rho^{-2} \psi(|x|/\rho)$. The numerical approximation of the perturbed vortex sheet is now defined by setting
\[
\omega_0^\rho(x) := \int_{{\Gamma}} \psi_{\rho}(x-y) {\omega_0}(y) \, dy
\]
where $\rho$ determines the thickness (smoothness) of the approximate vortex sheet. The convolution at $x = (x_1, x_2) \in \T^2$ is evaluated via numerical quadrature:
\[
({\omega_0} \ast \psi_\rho)(x)
\approx 
\frac{\rho}{Q} \sum_{i} \psi_\rho(x-(\xi_i,{g}(\xi_i))) \sqrt{1+|{g}'(\xi_i)|^2},
\]
with $\xi_i = x_1 + i\rho/Q$ equidistant quadrature points in $x_1$, and ${g}(\xi)$ the function whose graph is ${\Gamma}$, i.e. $g(\xi) = d \sin(2\pi\xi)$. We choose $Q=400$ quadrature points. We denote by $U^\rho(x_1,x_2)$ the velocity field such that $\div(U^\rho)=0$ and $\curl(U^\rho) = \omega_0^\rho$.

Similar to the case of the flat vortex sheet, we carry out random perturbations of the sinusoidal vortex sheet as follows:
\[
\overline{u}^{\rho,\delta}(x) := \mathbb{P}(U^\rho(x_1,x_2+\sigma_\delta(x_1)).
\]
Here, $\mathbb{P}: L^2_x \to L^2_x$ denotes the Leray projection onto divergence-free vector fields, and we again fix a random function $\sigma_\delta(x)$,
\[
\sigma_\delta(x_1) = \sum_{k=1}^q \alpha_k \sin(2\pi x_1 - \beta_k),
\]
depending on a parameter $\delta \ge 0$ and a choice of (random) coefficients $\alpha_1, \dots, \alpha_q \in (0,\delta)$, $\beta_1, \dots, \beta_q \in [0,2\pi)$. For our simulations, we fix $q=10$ modes for the perturbations. In practice, the coefficients $\alpha_k$ are first drawn independently, uniformly in $(0,1)$, and then multiplied by $\delta$. The coefficients $\beta_k$ are i.i.d., with a uniform distribution on $[0,2\pi)$. The initial data for the statistical solution $\bar{\mu}^{\delta}_\rho \in \P(L^2_x)$ is defined as the law of these random perturbations. It depends on the two parameters $\rho \ge 0$, $\delta \ge 0$. While $\rho$ controls the smoothness of the initial data, $\delta$ measures the amplitude of the perturbation. We fix $\delta = 0.003125$ in the following and vary $\rho$ as a function of the grid size $N$. To approximate vortex sheet initial data, we must scale $\rho = \rho(N)$ with $N$, such that $\rho \to 0$ as $N\to \infty$. We use $\rho = 5/N$ for our simulations. The additional diffusion parameter $\epsilon$ of the spectral viscosity scheme is set to $\epsilon = 0.01$. With this choice of parameters, we will drop the sub- and superscripts and denote the initial data at a given resolution simply by $\overline{\mu}\in \P(L^2_x)$.

\begin{figure}[htbp]
	\centering
	\begin{subfigure}{.24\textwidth}
		\centering
		\includegraphics[width=\linewidth]{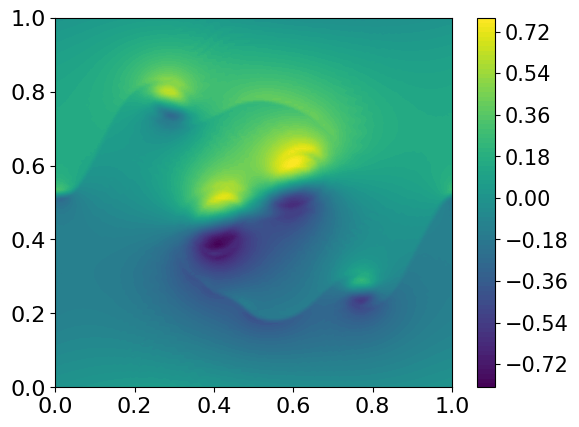}
	\end{subfigure}%
	\begin{subfigure}{.24\textwidth}
		\centering
		\includegraphics[width=\linewidth]{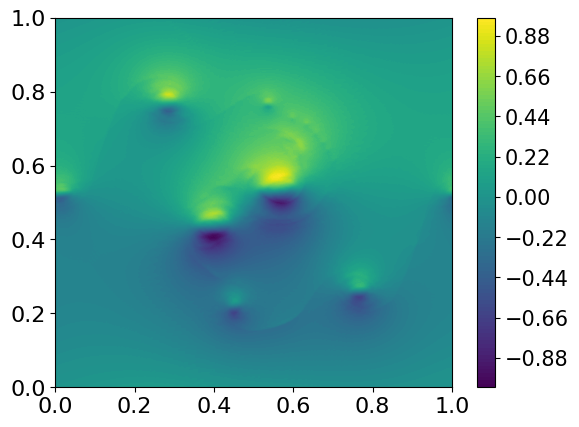}
	\end{subfigure}%
	\begin{subfigure}{.24\textwidth}
		\centering
		\includegraphics[width=\linewidth]{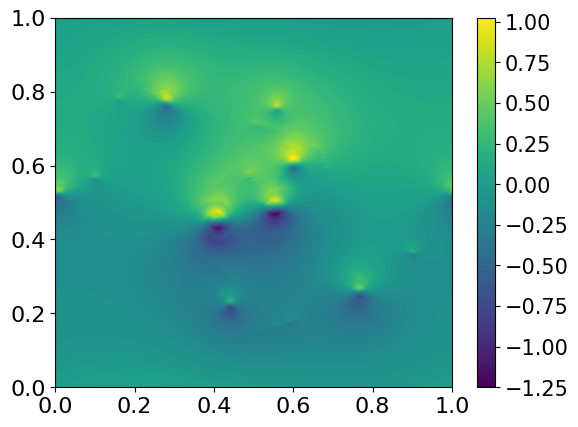}
	\end{subfigure}%
	\begin{subfigure}{.24\textwidth}
		\centering
		\includegraphics[width=\linewidth]{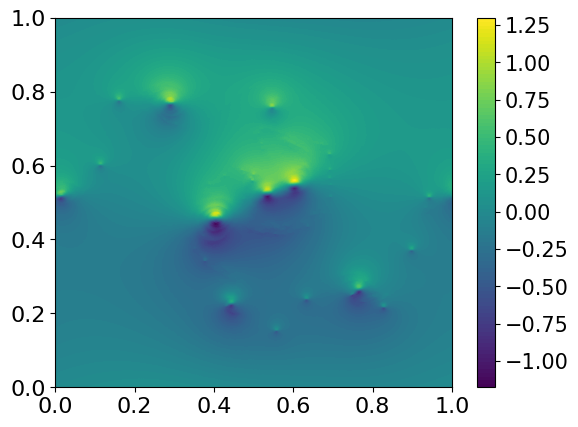}
	\end{subfigure}%
	
	\begin{subfigure}{.24\textwidth}
		\centering
		\includegraphics[width=\linewidth]{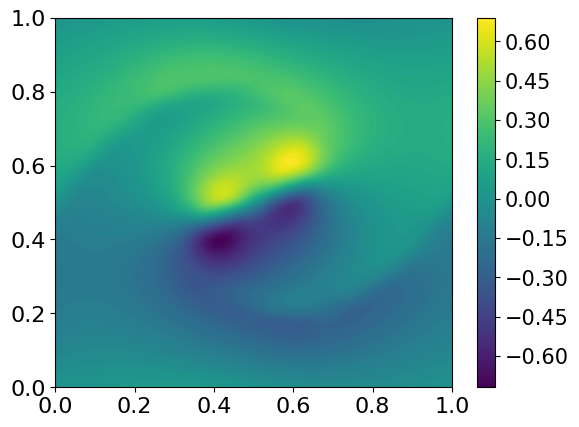}
	\end{subfigure}%
	\begin{subfigure}{.24\textwidth}
		\centering
		\includegraphics[width=\linewidth]{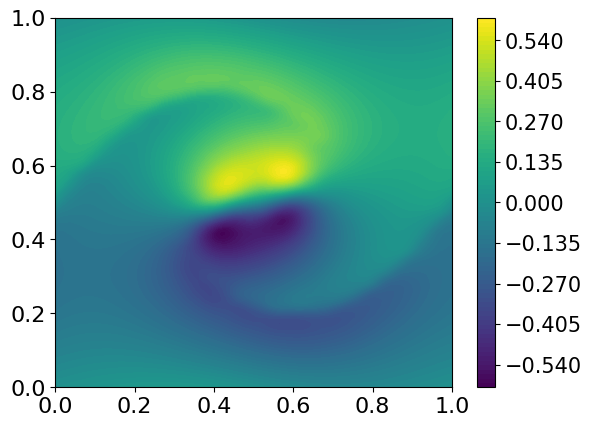}
	\end{subfigure}%
	\begin{subfigure}{.24\textwidth}
		\centering
		\includegraphics[width=\linewidth]{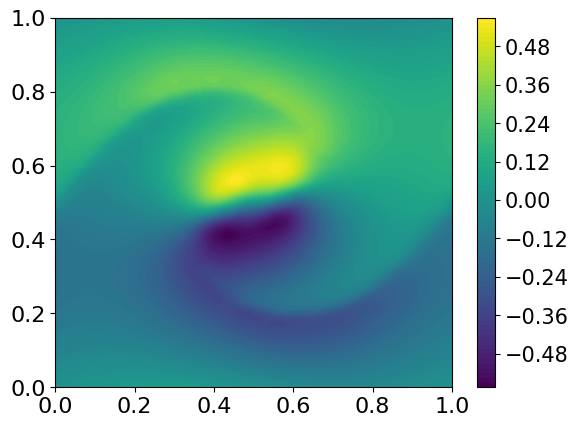}
	\end{subfigure}%
	\begin{subfigure}{.24\textwidth}
		\centering
		\includegraphics[width=\linewidth]{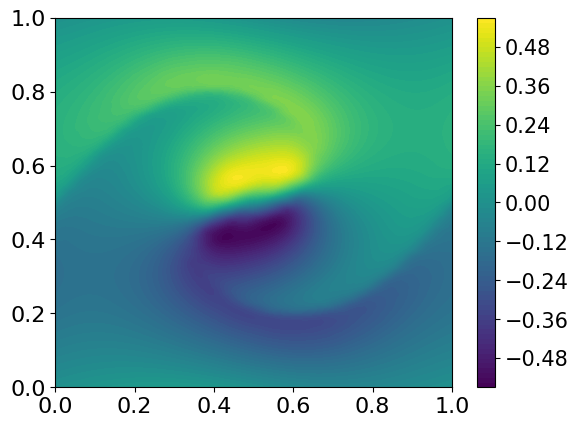}
	\end{subfigure}%

	\begin{subfigure}{.24\textwidth}
		\centering
		\includegraphics[width=\linewidth]{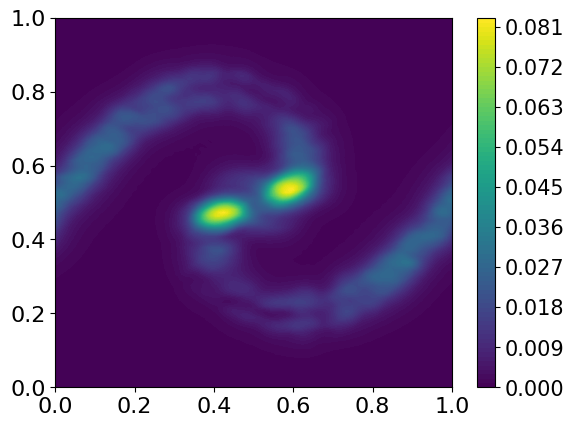}
		\caption{$N=128$}
	\end{subfigure}%
	\begin{subfigure}{.24\textwidth}
		\centering
		\includegraphics[width=\linewidth]{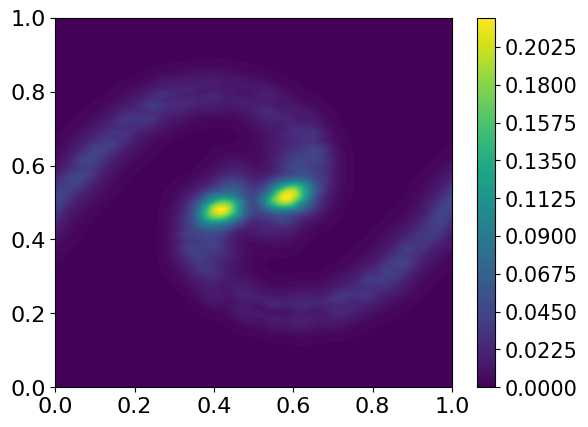}
		\caption{$N=256$}
	\end{subfigure}%
	\begin{subfigure}{.24\textwidth}
		\centering
		\includegraphics[width=\linewidth]{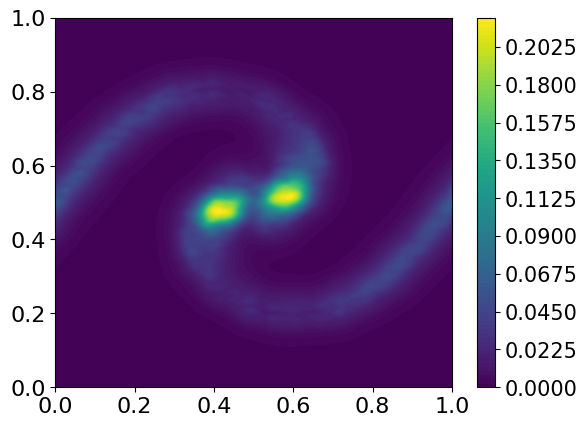}
		\caption{$N=512$}
	\end{subfigure}%
	\begin{subfigure}{.24\textwidth}
		\centering
		\includegraphics[width=\linewidth]{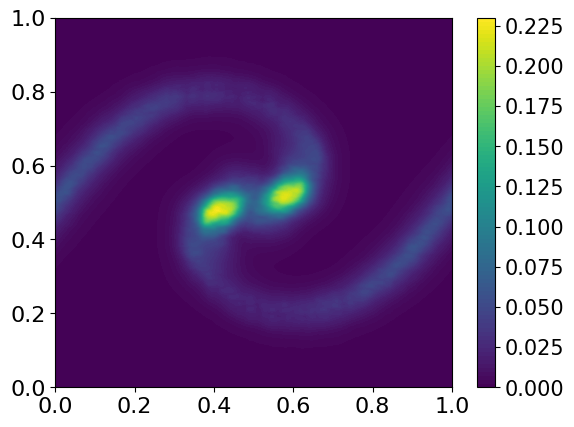}
		\caption{$N=1024$}
	\end{subfigure}%
	\caption{Results at time $T=1.2$ for the horizontal velocity $u_1$ of the sinusoidal vortex sheet, at different resolutions. Top Row: Sample; Middle Row: Mean; Bottom Row: Variance.}
	\label{fig:svs_dyn}
	\end{figure}

\begin{figure}[htbp]
	\centering
	\begin{subfigure}{.48\textwidth}
		\centering
		\includegraphics[width=\linewidth]{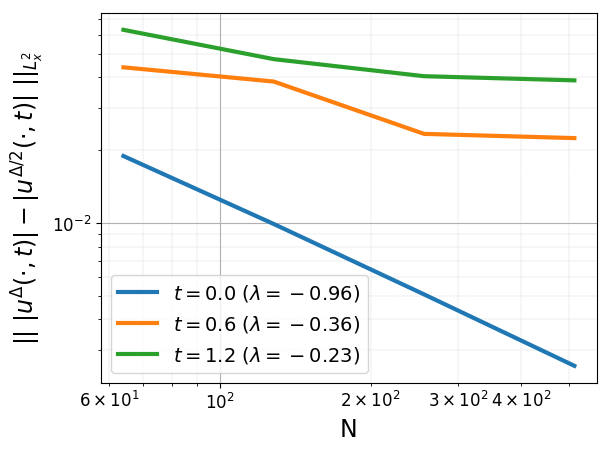}
		\caption{Sample}
	\end{subfigure}%
    \begin{subfigure}{.48\textwidth}
		\centering
		\includegraphics[width=\linewidth]{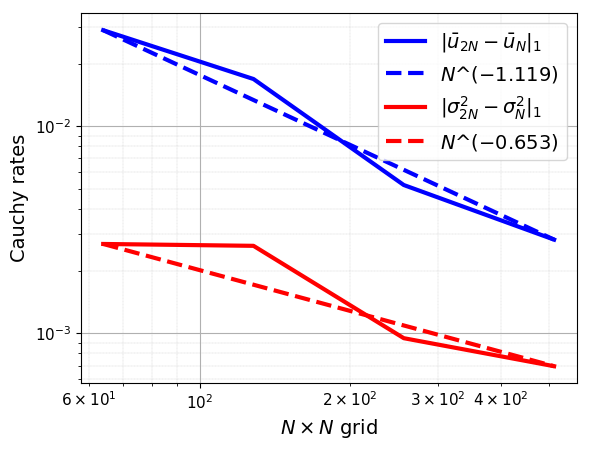}
		\caption{Mean and Variance}
		\end{subfigure}
\caption{Cauchy rates for the norm of the velocity field ($\sqrt{u_1^2+u_2^2}$) for the sinusoidal vortex sheet. Slope $\lambda$ is determined by a best fit. Left: Sample convergence rates at three different times $t=0,0.6,1.2$. Right: Convergence of mean and variance at $T=1.2$}
\label{fig:svs_conv}
\end{figure}
\subsubsection{Computation of individual samples} For any single realization of the random perturbation $\sigma_\delta(x)$, the resulting vorticity of the initial datum (sample) is a positive measure, concentrated on a sine curve (see figure \ref{fig:svs_init} (A) for horizontal component of velocity $u_1$). Hence, any single sample of the initial data in the Delort class. Therefore, by the results of \cite{LM2019}, the approximate solutions generated by the spectral viscosity method \eqref{eq:spectralvisc} will converge, on increasing resolution, to a weak solution of \eqref{eq:Eulerfull}. However, as noted in \cite{LM2019}, this convergence can be very slow as the flow breaks down into smaller and smaller vortices. In fact, this phenomenon is also seen from figure \ref{fig:svs_dyn} (Top row), where we plot the horizontal component of velocity $u_1$ at time $t=1.2$ and different resolutions. At this time, the initial vortex sheet has rolled over and broken down into a succession of small vortices, whose location and amplitude are different for different resolutions. This very slow convergence is also displayed in figure \ref{fig:svs_conv} (A), where we plot the Cauchy rates $\|u^{\Delta}(t) - u^{\Delta/2}(t)\|_{L^2_x}$, with $u^{\Delta}$ denoting the approximate solution computed with the spectral viscosity method \eqref{eq:spectralvisc}, for three different times $t=0,0.6,1.2$. As seen from this figure, the rate of convergence decreases very rapidly and at time $t=1.2$, it appears as if there is no convergence on mesh refinement. 
\subsubsection{Structure functions and Compensated energy spectra.} 
Given this apparent non-convergence of individual samples, it is pertinent to investigate if computing the statistics will be more convergent. To this end, we consider the initial data to be the initial probability measure $\bar{\mu}$. The mean and variance (of the horizontal component $u_1$) are plotted in figure \ref{fig:svs_init} (B,C). From this figure, we observe that the initial probability measure is concentrated on very small perturbations of the underlying sinusoidal vertex sheet, as reflected in the initial variance.

In order to investigate the convergence of approximations to the statistical solution, generated by the algorithm \ref{alg:MC}, we follow the template of the previous numerical experiment and compute the (instantaneous) structure function \eqref{eq:sfin} and the compensated energy spectrum \eqref{eq:ces} in figure \ref{fig:svs_sf}. From this figure, we observe that the structure function at time $t=1.2$ scales with an exponent of $\approx 0.7$ at the finest resolutions. From \eqref{eq:cessf}, this implies roughly a $\gamma = 2.4$ in the scaling of the energy spectrum \eqref{eq:ces}. A better fit to the scaling of the energy spectrum is found with $\gamma = 2.2$. We plot the compensated energy spectrum with the latter value of $\gamma$ in figure \ref{fig:svs_sf} (B). From this figure, we see that for the inertial range, the energy spectrum clearly decays (faster than) a rate of $2.2$. Thus, the assumptions of theorems \ref{thm:timecompact}, \ref{prop:inertialrange} are satisfied and the approximations will converge to a statistical solution of \eqref{eq:Eulerfull}. 

	\begin{figure}[H]
	\centering
	\begin{subfigure}{.48\textwidth}
		\centering
		\includegraphics[width=\linewidth]{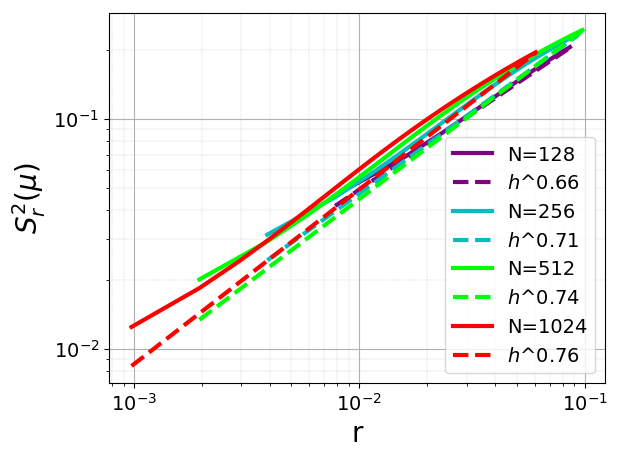}
		\caption{Instant. structure function \eqref{eq:sfin}}
	\end{subfigure}%
    \begin{subfigure}{.48\textwidth}
		\centering
		\includegraphics[width=\linewidth]{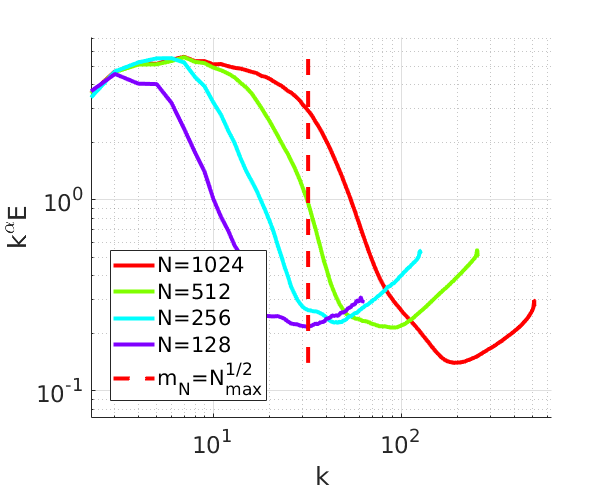}
		\caption{Compensated energy spectrum \eqref{eq:ces} with $\gamma = 2.2$ }
		\end{subfigure}
\caption{The structure function and compensated energy spectrum for the sinusoidal vortex sheet at time $T=1.2$}
\label{fig:svs_sf}
\end{figure}

\subsubsection{Convergence of observables and Wasserstein Distances.}
Given the results on the computed structure functions and energy spectra, the approximations will converge. But is this convergence at a better rate than that of single samples? To investigate this issue, we consider two different sets of computations. First, we compute the mean and the variance of the velocity field at different resolutions and plot them (for the horizontal velocity at time $t=1.2)$ in figure \ref{fig:svs_dyn} (Middle and Bottom rows). Clearly, the one-point statistics appear much more convergent than the single sample results. The mean flow consists of a coherent set of large vortices, which is in stark contrast to the large number of vortices formed in the single sample simulations. Moreover, we also plot Cauchy rates for the mean and the variance, corresponding to the norm $\sqrt{u_1^2 + u_2^2}$ at time $t=1.2$, and different resolutions in figure \ref{fig:svs_conv} (B). Again, we observe that these one-point statistics converge at a significantly faster rate than the single sample. These results indicate that one can expect significantly better convergence of approximations for statistics than for individual realizations of fluid flows, even if the initial probability measure is a small perturbation of the underlying deterministic data and further reinforces the results of \cite{FKMT17,LM2015,FLMW1} in this direction.

\begin{figure}[H]
	\centering
	\begin{subfigure}{.48\textwidth}
		\centering
		\includegraphics[width=\linewidth]{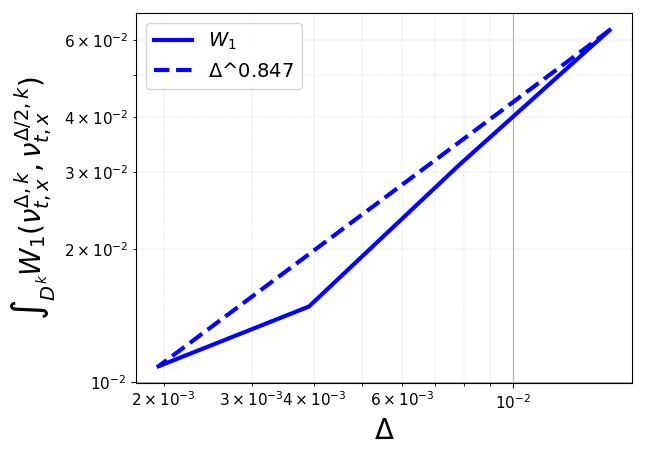}
		\caption{$k=1$}
	\end{subfigure}%
	\begin{subfigure}{.48\textwidth}
		\centering
		\includegraphics[width=\linewidth]{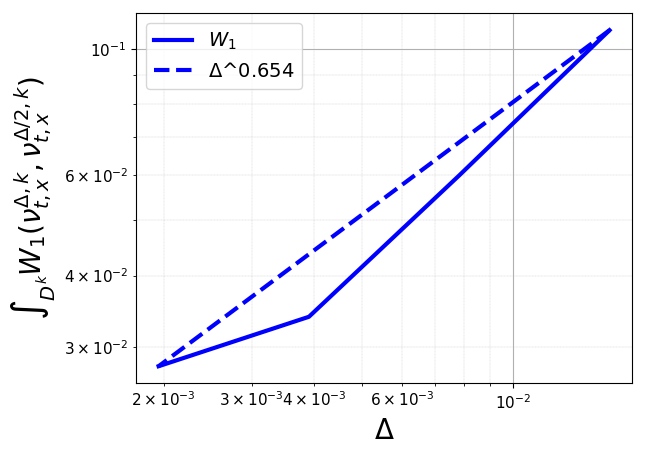}
		\caption{$k=2$}
	\end{subfigure}
\caption{The metrics $\int_{D^k}W_1(\nu^{\Delta,k}_{t,x}, \nu^{\Delta/2,k}_{t,x}) dx$ for the one- and two-point correlation marginals for the sinusoidal vortex sheet at time $t=1.2$}
\label{fig:svs_wass}
\end{figure}

Finally, we plot the Wasserstein distances \eqref{eq:met2} for $k=1,2$, corresponding to the one- and two-point correlation marginals, at time $t=1.2$, in figure \ref{fig:svs_wass}. The results clearly show convergence in these metrics at a significantly faster rate than for individual samples and indicate possible convergence in the metric \eqref{eq:dtdef} on probability measures on $L^2$.

\subsection{Fractional Brownian Motion.} 
The study of the evolution of initial ensembles corresponding to (fractional) Brownian motion stems from \cite{SAF1,Sinai}, where the authors model interesting aspects of Burgers turbulence by evolving Brownian motion initial data for the (scalar) Burgers' equation, see \cite{FLYM} for a more recent numerical study. Similarly in \cite{FLMW1}, the authors consider the compressible Euler equations with (fractional) Brownian motion initial data. Following these articles, we will consider the two-dimensional Euler equations \eqref{eq:Eulerfull} with initial data corresponding to fractional Brownian motion, i.e. the following initial data:
\begin{equation}
    \label{eq:fbmin}
u^{x,H}_0(\omega; x):=B_1^H(\omega; x),~
w^{y,H}_0(\omega;x) := B_2^H(\omega; x).
\qquad\text{for }\omega\in\Omega,\ x\in D
\end{equation}
where $B^H_1$ and $B^H_2$ are two independent two-dimensional fractional Brownian motions with the Hurst index $H\in(0,1)$. Standard Brownian motion corresponds to a Hurst index of $H=1/2$. The initial probability measure $\bar{\mu}$ is the law of the above random field.

\begin{figure}[htbp]
\begin{subfigure}{.3\textwidth}
\includegraphics[width=\textwidth]{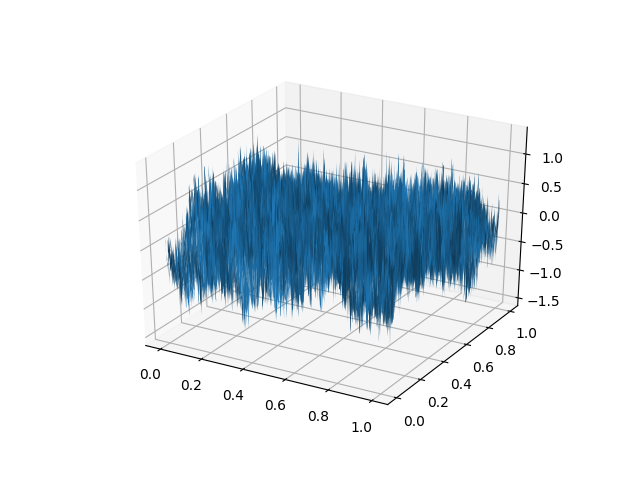}
\caption{$H=0.15$}
\end{subfigure}
\begin{subfigure}{.3\textwidth}
\includegraphics[width=\textwidth]{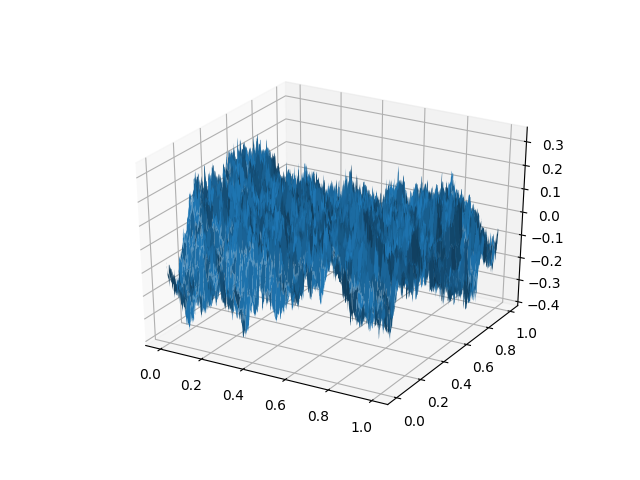}
\caption{$H=0.5$}
\end{subfigure}
\begin{subfigure}{.3\textwidth}
\includegraphics[width=\textwidth]{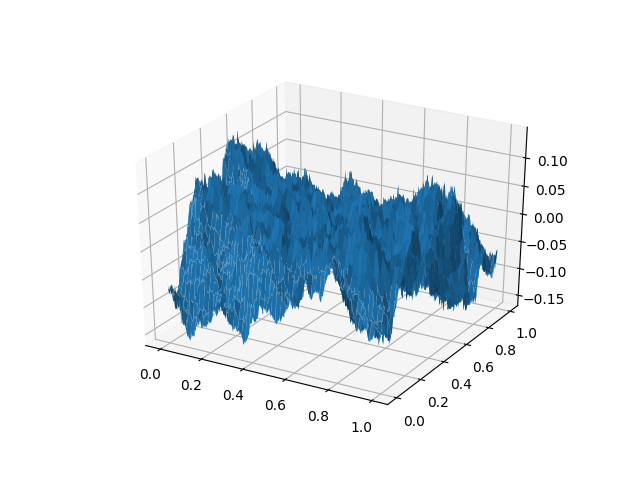}
\caption{$H=0.75$}
\end{subfigure}
\caption{A single sample of initial horizontal velocity $u_1$ for the fractional Brownian motion initial data \eqref{eq:fbmin} for three different Hurst indices}
\label{fig:fbm_init}
\end{figure}

To generate fractional Brownian motion, we use the random midpoint displacement method originally introduced by L\'evy~\cite{Levy1965} for Brownian motion, and later adapted for fractional Brownian motion, see \cite{FLMW1} section 6.6.1.

Considering fractional Brownian motion initial data \eqref{eq:fbmin} is a significant deviation from the vortex sheet initial data in the following respects,
\begin{itemize}
    \item For the vortex sheet initial data, the initial measure $\bar{\mu} \in \Prob(L^2_x)$ was concentrated on a 20-dimensional subset of $L^2_x$ (corresponding to the choice of 20 free parameters $\alpha_k$, $\beta_k$). On the other hand, in the limit of infinite resolution ($\Delta \rightarrow 0$), the fractional Brownian motion initial data corresponds to a measure concentrated on an infinite dimensional subset of $L^2_x$.
    \item For any $0 < H < 1$, and for any sample $\omega \in \Omega$, the initial vorticity for \eqref{eq:fbmin} is not a Radon measure. Consequently, the initial data does not belong to the Delort class and there are no existence results for the corresponding samples. Hence, fractional Brownian motion does not fall within the ambit of any of the available well-posedness theories for two-dimensional Euler equations.
    \item The Hurst index $H$ in \eqref{eq:fbmin} controls the regularity (and also roughness) of the initial data (pathwise). Roughly speaking, each sample is H\"older continuous with exponent $H$. Hence, we can consider a very wide range of scenarios in terms of roughness of the initial data by varying the Hurst-index $H$, see figure \ref{fig:fbm_init} for realizations of the horizontal velocity field for three different Hurst indices. In particular, one can observe from this figure that lowering the value of $H$ leads to oscillations of both higher amplitude and frequency in the initial velocity field.
    \end{itemize}

\subsubsection{Structure functions and Compensated energy spectra}
In order to verify convergence of the approximations, generated by algorithm \ref{alg:MC}, for the fractional Brownian motion initial data \eqref{eq:fbmin}, we will check if the computed structure functions \eqref{eq:sfin} decay uniformly with respect to resolution, on decreasing correlation lengths. In figure \ref{fig:fbm_sf} (Top Row), we plot the structure function at time $T=1$ for three different Hurst indices of $H=0.75,0.5,0.15$ and observe that the structure functions indeed decay to zero at a certain exponent (independent of resolution). These exponents are approximately $0.8$ for initial $H=0.75$, $0.6$ for the standard Brownian motion initial data ($H=0.5$) and $0.55$ for the initially rough $H=0.15$. These results indicate that the conditions of the compactness theorem \ref{thm:timecompact} are fulfilled and the approximations converge to a statistical solution of \eqref{eq:Eulerfull}.

\begin{figure}[H]

\begin{subfigure}{.3\textwidth}
\includegraphics[width=\textwidth]{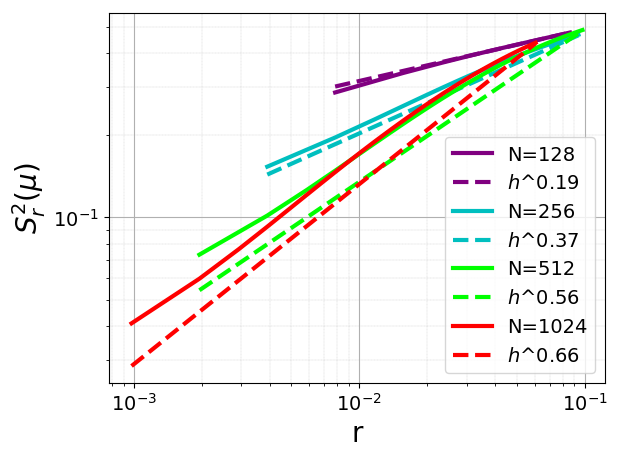}
\end{subfigure}
\begin{subfigure}{.3\textwidth}
\includegraphics[width=\textwidth]{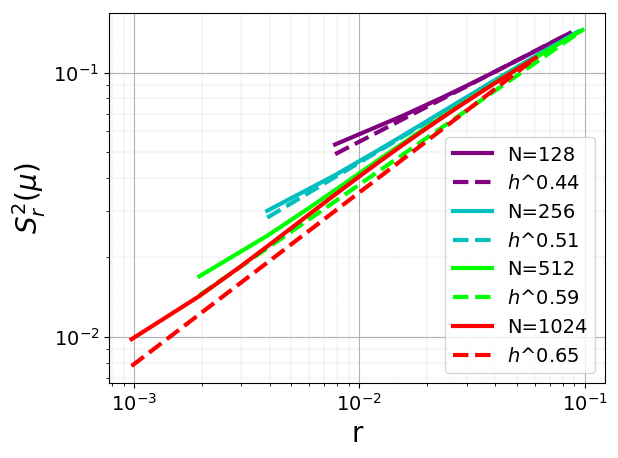}
\end{subfigure}
\begin{subfigure}{.3\textwidth}
\includegraphics[width=\textwidth]{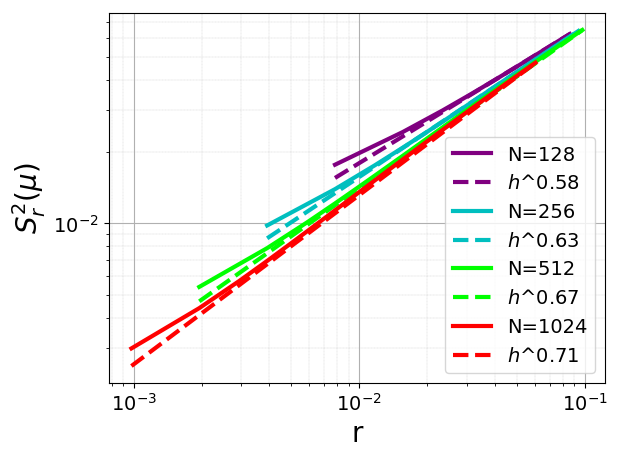}
\end{subfigure}

\begin{subfigure}{.3\textwidth}
\includegraphics[width=\textwidth]{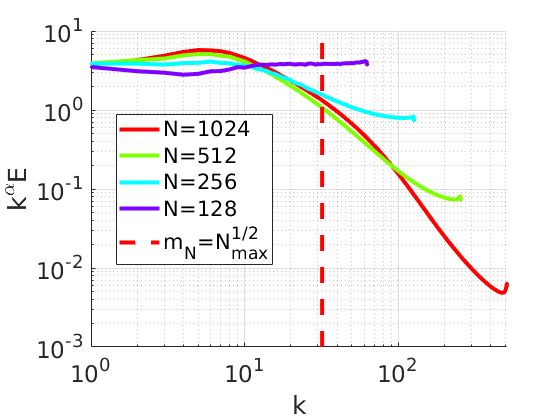} 
\caption{$H=0.15$}
\end{subfigure}
\begin{subfigure}{.3\textwidth}
\includegraphics[width=\textwidth]{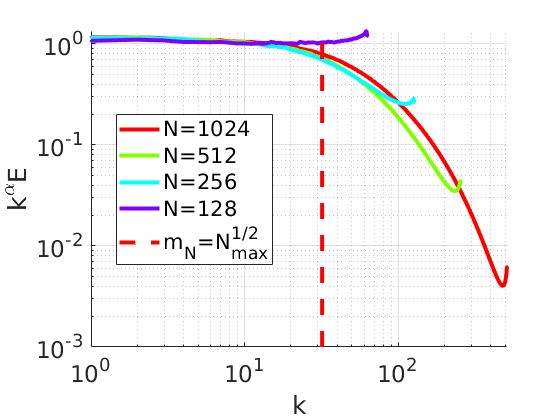}
\caption{$H=0.5$}
\end{subfigure}
\begin{subfigure}{.3\textwidth}
\includegraphics[width=\textwidth]{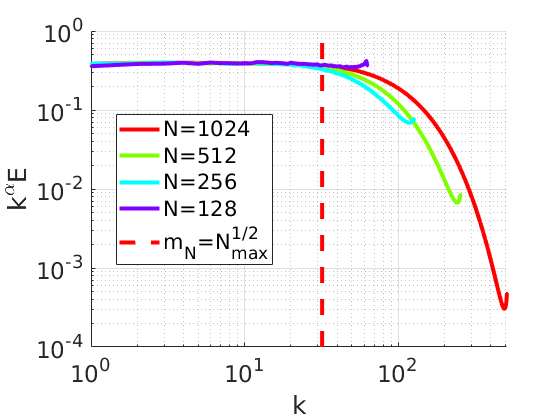}
\caption{$H=0.75$}
\end{subfigure}
\caption{Instantaneous structure function \eqref{eq:sfin} (Top Row) and Compensated energy spectrum \eqref{eq:ces} (Bottom Row) for Fractional Brownian motion initial data with three different Hurst indices at time $T=1$. The compensated energy spectrum \eqref{eq:ces} is computed with $\gamma = 1.3$  ($H=0.15$), $\gamma = 2.0$ ($H=0.5$) and $\gamma = 2.5$ ($H=0.75$)}
\label{fig:fbm_sf}
\end{figure}

This convergence is further reinforced by the computed compensated energy spectra \eqref{eq:ces}, at time $T=1$, for the three different Hurst indices shown in figure \ref{fig:fbm_sf} (Bottom Row). Based on the value of the Hurst index, we choose the compensating index  $\gamma = 2.5,2,1.3$ for the $H=0.75$,$H=0.5$,$H=0.15$, respectively. These values of $\gamma$ are chosen to provide the correct scaling of the energy spectra at the initial time $t=0$. As seen from figure \ref{fig:fbm_sf}, the compensated energy spectra remain bounded up to the final time $t=T$, independent of the spectral resolution. Hence, the energy spectrum decays at least at the rate of $K^{-\gamma}$ for increasing wave number $K$, in the inertial range. Consequently, we can readily apply proposition \ref{prop:inertialrange} and conclude that the approximations, generated by the algorithm \ref{alg:MC}, converge to a statistical solution, for all three values of the Hurst index $H$ in \eqref{eq:fbmin}. 

\subsubsection{Convergence in Wasserstein distance.} 
Next, we seek to verify convergence of observables (statistical quantities of interest). To this ends, we follow the previous section and compute the Wasserstein distances $\int_{D^k}W_1(\nu^{\Delta,k}_{t,x}, \nu^{\Delta/2,k}_{t,x}) dx$, corresponding to the $k$-point correlation marginals for the three different Hurst indices of $H=0.75,0.5,0.15$. In figure \ref{fig:fbm_wass}, these metrics are computed at time $T=1$, for $k=1,2$, corresponding to one-point and two-point statistical quantities of interest. As observed from the figure, the approximations clearly converge in this metric for both one- and two-point statistics, at rates which are independent of the underlying initial Hurst index. The two-point correlation marginals appear to converge at a slower rate than the one-point Young measures. These results validate convergence of all one- and two-point statistical quantities of interest. Taken together with the results for the structure function, compensated energy spectra and theorem \ref{thm:timecompact}, they strongly suggest convergence in metric $d_T$ \eqref{eq:dtdef} on $L^1_T(\Prob)$.
\begin{figure}[H]

\begin{subfigure}{.3\textwidth}
\includegraphics[width=\textwidth]{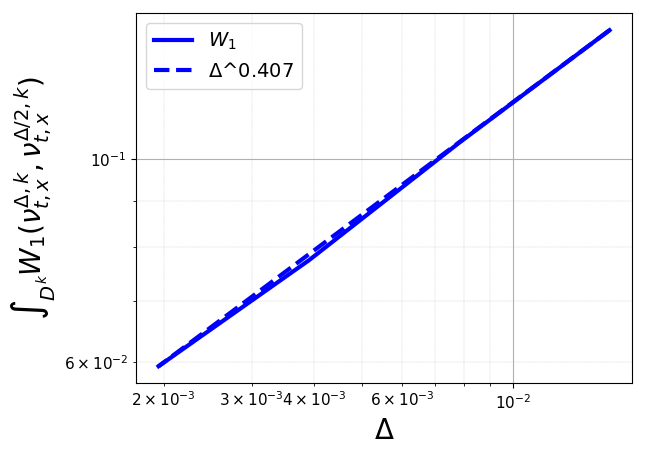}
\end{subfigure}
\begin{subfigure}{.3\textwidth}
\includegraphics[width=\textwidth]{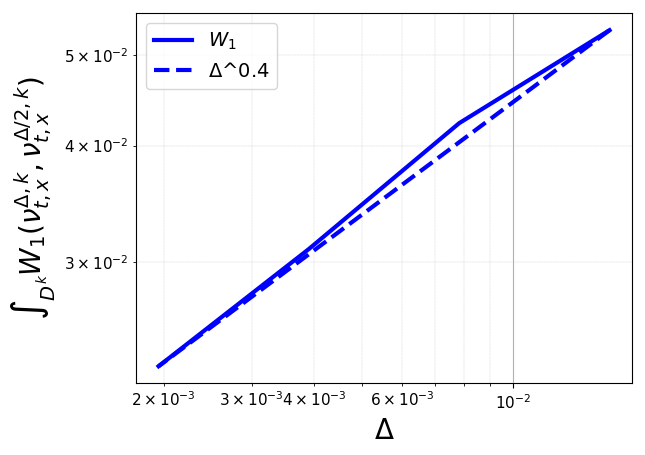}
\end{subfigure}
\begin{subfigure}{.3\textwidth}
\includegraphics[width=\textwidth]{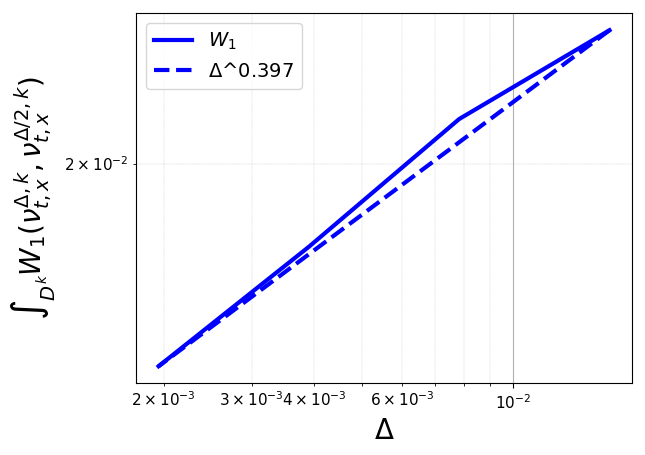}
\end{subfigure}

\begin{subfigure}{.3\textwidth}
\includegraphics[width=\textwidth]{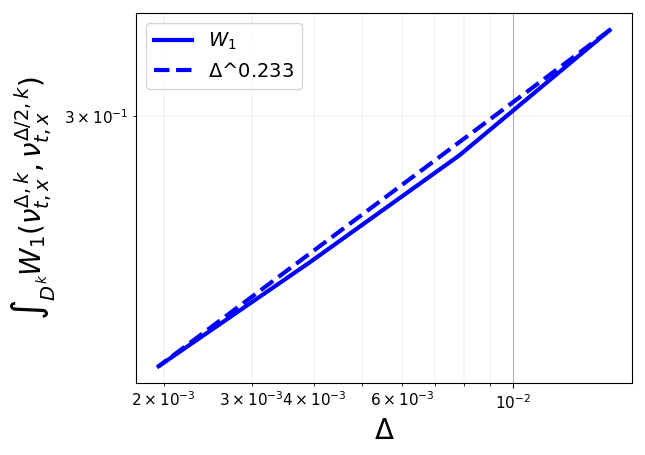}
\caption{$H=0.15$}
\end{subfigure}
\begin{subfigure}{.3\textwidth}
\includegraphics[width=\textwidth]{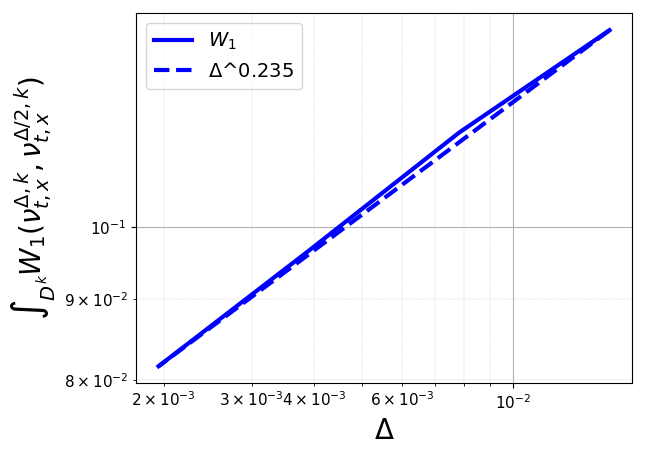}
\caption{$H=0.5$}
\end{subfigure}
\begin{subfigure}{.3\textwidth}
\includegraphics[width=\textwidth]{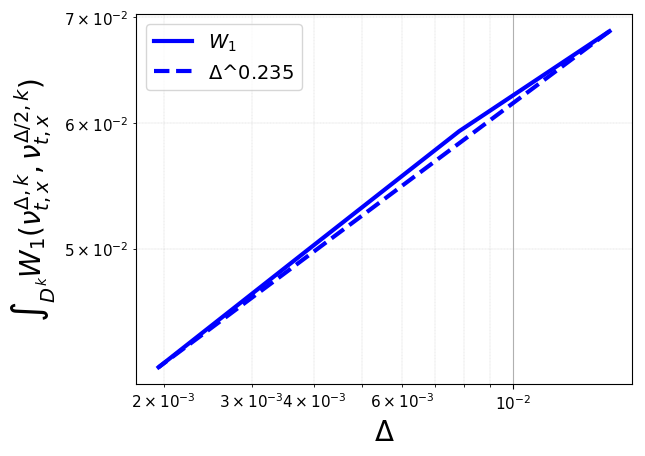}
\caption{$H=0.75$}
\end{subfigure}
\caption{Wasserstein distances $\int_{D^k}W_1(\nu^{\Delta,k}_{t,x}, \nu^{\Delta/2,k}_{t,x}) dx$ for $k=1$ (Top Row) and $k=2$ (Bottom Row)   for Fractional Brownian motion initial data with three different Hurst indices at time $t=1$.}
\label{fig:fbm_wass}
\end{figure}

\if{0}
\begin{remark}
It is straightforward to calculate that the structure function \eqref{eq:sfin} corresponding to initial fractional Brownian motion data scales as $r^H$, with $r$ being the correlation length and $H$ the initial Hurst index in \eqref{eq:fbmin}. On the other hand, the computations from algorithm \ref{alg:MC}, resulted in structure functions decaying at exponents, at time $T=1$, which were slightly more regular, i.e. $0.8$ (for $H=0.75$) and $0.6$ (for $H=0.5$), for more regular initial Hurst indices, while being significantly more regular, i.e. decay at exponent of $0.5$ for the very rough initial data with $H=0.15$. In fact, we have observed a long time decay exponent of at least $0.5$ in correlation length in all our simulations with fractional Brownian motion, indicating a possible gain in regularity in this sense. This gain is consistent with the one observed for the compressible Euler equations in \cite{FLMW1}. We will investigate the possible mechanisms and implications of this subtle gain in regularity in a forthcoming paper.  
\end{remark}
\fi

\subsection{Stability of the computed statistical solution.} The afore-presented numerical experiments clearly validate the convergence theory developed in this article. Therefore, algorithm \ref{alg:MC} provides a practical way to compute statistical solutions. We use this algorithm to study the sensitivity/stability of the computed statistical solution in the context of the discontinuous flat vortex sheet initial data ($\rho = 0$ in \eqref{eq:fvs1}) to the following perturbations,
\begin{itemize}
    \item Type of the initial perturbation, with respect to the underlying probability distribution.
    \item Type of underlying numerical method.
    \item Amplitude of the initial perturbation.
\end{itemize}

We start by changing the type of the underlying initial probability measure in the flat vortex sheet initial data. Instead of choosing the i.i.d. random variables $\alpha_k \sim \mathcal{U}[-1,1]$ according to a uniform distribution as before in \eqref{eq:fvs1}, we choose them from a normal distribution $\alpha_k \sim \mathcal{N}(0,1/3)$ with mean $0$ and variance $1/3$. The mean and variance are chosen so the distribution's mean and variance are consistent with those of $\mathcal{U}[-1,1]$. 

\begin{figure}[htbp]
\begin{subfigure}{.45\textwidth}
\includegraphics[width=\textwidth]{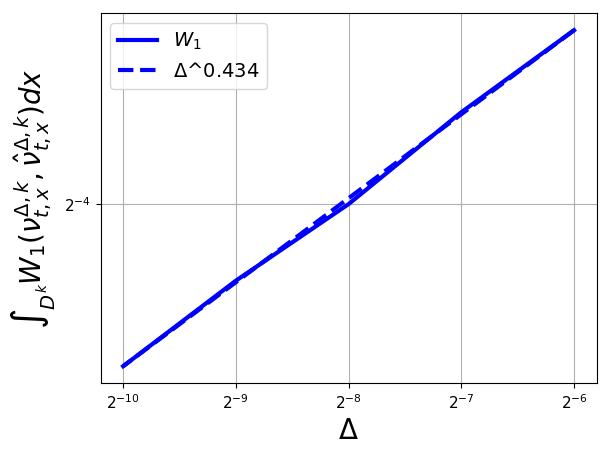}
\end{subfigure}
\begin{subfigure}{.45\textwidth}
\includegraphics[width=\textwidth]{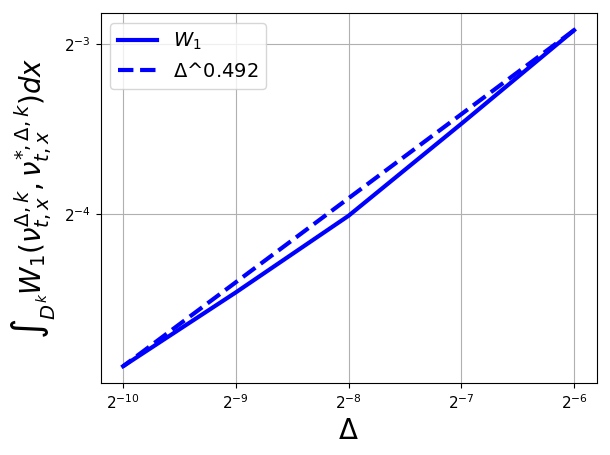}
\end{subfigure}

\begin{subfigure}{.45\textwidth}
\includegraphics[width=\textwidth]{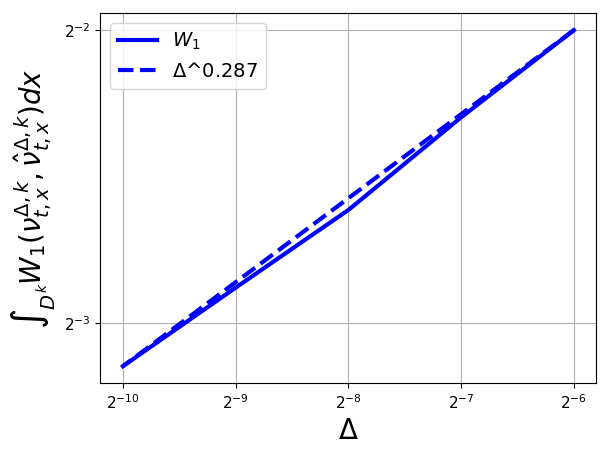}
\caption{Type}
\end{subfigure}
\begin{subfigure}{.45\textwidth}
\includegraphics[width=\textwidth]{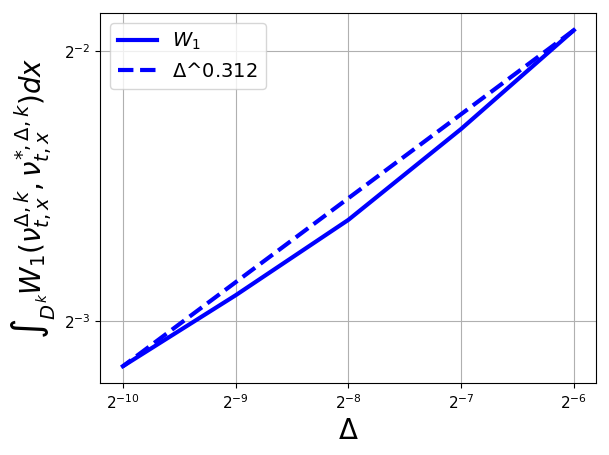}
\caption{Scheme}
\end{subfigure}
\caption{Stability of the flat vortex sheet \eqref{eq:fvs1} to perturbations. For $k=1$ (Top Row) and $k=2$ (Bottom Row) and at time $t=0.4$, we plot Left Column:  Distance $\int_{D^k}W_1(\nu^{\Delta,k}_{t,x}, \hat{\nu}^{\Delta,k}_{t,x}) dx$ with $\hat{\nu}$ the correlation marginal corresponding to a Normally distributed initial measure in \eqref{eq:fvs1}. Right Column: Distance $\int_{D^k}W_1(\nu^{\Delta,k}_{t,x}, \nu^{\ast,\Delta,k}_{t,x}) dx$ with $\nu^{\ast}$ the correlation marginal computed with the finite difference projection method. 
}
\label{fig:fvs_pert}
\end{figure}

In order to compare the corresponding approximate statistical solutions, we compute the Wasserstein distance $\int_{D^k}W_1(\nu^{\Delta,k}_{t,x}, \hat{\nu}^{\Delta/2,k}_{t,x}) dx$, at different resolutions $\Delta$. Here, $\hat{\mu}^{\Delta}$ is the statistical solution computed with normally distributed initial data and $\hat{\nu}^{\Delta}$ is the corresponding correlation measure. We set the initial perturbation amplitude $\delta = 0.05$ in \eqref{eq:fvs1}, $t=0.4$ and $k=1,2$ and plot the computed Wasserstein distances in figure \ref{fig:fvs_pert} (Left column). As seen from this figure, the two approximate solutions clearly converge to each other in this metric as the resolution is increased. This indicates that the computed statistical solutions are stable with respect to the variation of underlying initial probability measures. 

Next, we consider if the computed statistical solution depends on the underlying numerical method. To this end, we compute the statistical solution for the discontinuous flat vortex sheet initial data \eqref{eq:fvs1} with algorithm \ref{alg:MC} but replace the spectral viscosity method with a finite difference projection method \cite{Chorin,BCG,LeonardiPhD}. The convergence of this method to a statistical solution is considered in a forthcoming paper \cite{Pares-Pulido1}. We denote the computed statistical solutions with this method as $\mu^{\ast,\Delta}_t$ and the corresponding correlation measures as $\nu^{\ast,\Delta}$ and compute the Wasserstein distance $\int_{D^k}W_1(\nu^{\Delta,k}_{t,x}, \nu^{\ast,\Delta/2,k}_{t,x}) dx$ at time $t=0.4$ and $k=1,2$ and plot the results in figure \ref{fig:fvs_pert} (Right Column). From this figure, we readily conclude that the statistical solutions computed with the finite difference projection method converges to that computed with the spectral viscosity method on increasing resolution. This strongly suggests the stability of the computed statistical solutions to the underlying (convergent) numerical method. 

\begin{figure}[htbp]
\begin{subfigure}{.45\textwidth}
\includegraphics[width=\textwidth]{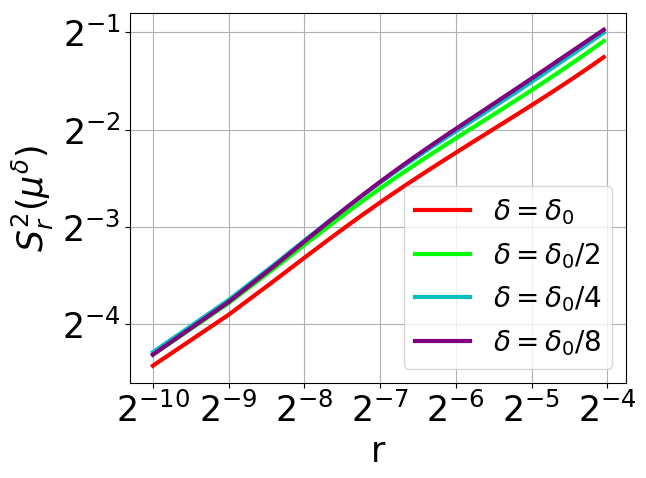}
\caption{Structure function}
\end{subfigure}
\begin{subfigure}{.45\textwidth}
\includegraphics[width=\textwidth]{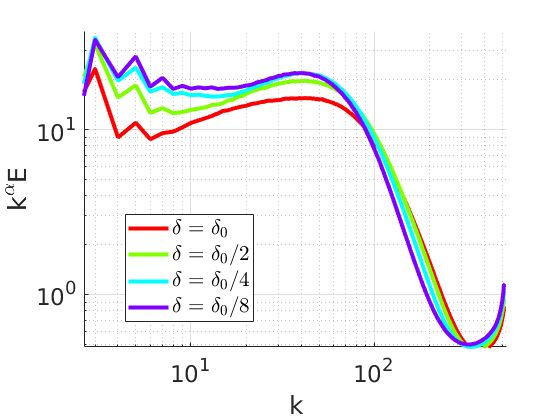}
\caption{Compensated Energy spectrum}
\end{subfigure}

\begin{subfigure}{.45\textwidth}
\includegraphics[width=\textwidth]{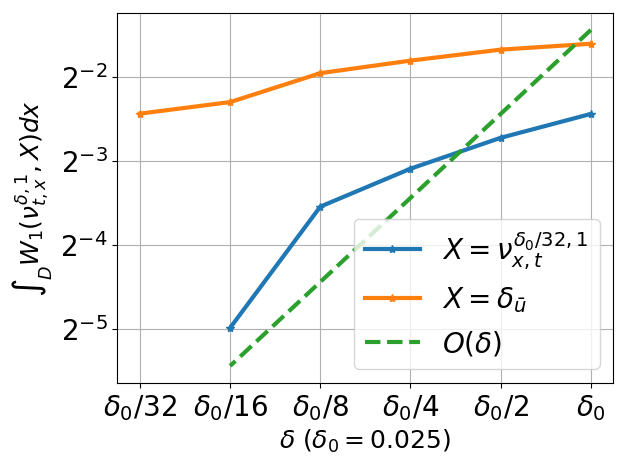}
\caption{$1$-pt Wasserstein distance}
\end{subfigure}
\begin{subfigure}{.45\textwidth}
\includegraphics[width=\textwidth]{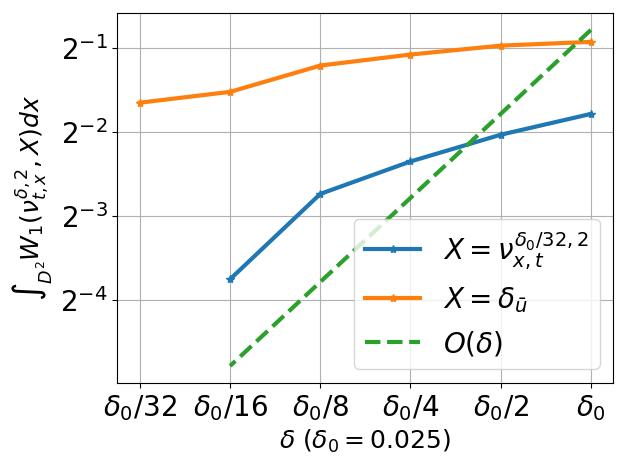}
\caption{$2$-pt Wasserstein distance}
\end{subfigure}
\caption{Stability of the flat vortex sheet \eqref{eq:fvs1} to amplitude of perturbations $\delta$. Top Row: Instantaneous structure function \eqref{eq:sfin} (Left) and Compensated energy spectrum \eqref{eq:cessf} with $\gamma = 2$ (Right) at time $T=0.4$ for different values of $\delta$. Bottom Row: Wasserstein distances with respect to the stationary solution (orange) and reference statistical solution (blue) for the $1$-point (Left) and $2$-point (Right) correlation marginals.
}
\label{fig:stabpert}
\end{figure}

Finally, we compute approximate statistical solutions for different amplitudes of the initial (random) perturbation in \eqref{eq:fvs1} by taking different values of the perturbation amplitude ranging from $\delta = 0.05$ to $\delta = 0.05/32$, corresponding to smaller and smaller perturbations of the underlying flat vortex sheet. We computed the statistical solution at the highest resolution of $N=1024$ with $m=N$ Monte Carlo samples in algorithm \ref{alg:MC}. 

First, we examine if there is convergence of the computed statistical solutions as $\delta \rightarrow 0$. To this end, we compute the instantaneous structure function \eqref{eq:sfin} at time $T=0.4$ for different values of $\delta$ and plot the results in figure \ref{fig:stabpert} (Top Row, Left). Clearly, the computed structure functions are very close and decay with decreasing correlation length with approximately the same exponent. Thus, we can appeal to theorem \ref{thm:timecompact} and claim that the approximate statistical solutions converge. This is further reinforced when we compute the instantaneous compensated energy spectrum \eqref{eq:cessf} at time $T=0.4$ and with $\gamma = 2$. We observe from figure \ref{fig:stabpert} (Top Row, Right) that the computed energy spectra with different values of $\delta$ are very close and decay at the approximately the same rate. Hence, from theorem \ref{prop:inertialrange}, the underlying approximate statistical solutions will converge. 

After establishing convergence of statistical solutions when the amplitude of perturbations in the initial datum \eqref{eq:fvs1} is decreased, it is natural to ask what these solutions converge to. One possibility is the initial discontinuous flat vortex sheet itself. After all, setting $\rho,\delta \rightarrow 0$ in \eqref{eq:fvs1}, it is easy to see that the discontinuous flat vortex sheet is a stationary solution of the incompressible Euler equations. To check if this stationary solution is indeed the zero perturbation limit of the computed statistical solution, we compute the Wasserstein distances $\int_{D}W_1(\nu^{\delta,1}_{t,x}, \delta_{\bar{u}(x)}) \, dx$ and $\int_{D^2}W_1(\nu^{\delta,2}_{t,x,y}, \delta_{\bar{u}(x)\otimes\bar{u}(y)}) \, dx \, dy$ for different values of the perturbation parameter $\delta$ with $t=0.4$ and $k=1,2$. Here, $\bar{u}$ is the stationary solution corresponding to the flat vortex sheet initial data, i.e. $\rho,\delta \rightarrow 0$. We display these distances in figure \ref{fig:stabpert} (Bottom Row). We observe from this figure that there does not seem to be any perceptible evidence of convergence to the stationary solution. This is consistent with the findings in \cite{LM2015}, where the authors had observed a non-atomic measure-valued solution as the limit of the perturbations to the flat vortex sheet. 

In order to further identify the limit of the computed statistical solutions as $\delta \rightarrow 0$, we compute a \emph{reference solution} $\mu_t^{\mathrm{ref}}$, by setting $\delta = \frac{0.05}{32}$ in \eqref{eq:fvs1} and $N=m=1024$ as the spectral resolution and Monte Carlo samples. We compute the Wasserstein distances $\int_{D^k}W_1(\nu^{\delta,k}_{t,x}, \nu^{\mathrm{ref},k}_{t,x}) \, dx$ for different values of the perturbation parameter $\delta$ with $t=0.4$ and $k=1,2$. Here, ${\bf \nu}^{\mathrm{ref}}$ is the correlation measure, corresponding to the reference statistical solution. The distances, plotted in figures \ref{fig:stabpert} (Bottom Row), do decrease as $\delta$ is decreased, albeit at a slow rate. However, this decrease is clearly visible when compared to the lack of convergence to the stationary solution in the same figures. Thus, we have established that the computed statistical solutions are stable with respect to the amplitude of perturbations and moreover, converge to a probability measure that is different from the one concentrated on the stationary solution.

\section{Discussion}
We consider the incompressible Euler equations \eqref{eq:Eulerfull} in this paper. The existence of classical (or weak) solutions is an outstanding open question in three space dimensions. Although weak solutions are known to exist in two space dimensions, even for very rough initial data, they may not be unique. Similarly, numerical experiments reveal that standard numerical methods may not converge, or converge very slowly, to weak solutions on increasing resolution.

Given these inadequacies of traditional notions of solutions, it is imperative to find solution concepts for \eqref{eq:Eulerfull} that are well-posed and amenable to efficient numerical approximation. In this context, we consider the solution framework of statistical solutions in this paper. Statistical solutions are time-parametrized probability measures on $L^2(D;\mathbb{R}^d)$. Given the characterization of probability measures on $L^p$ spaces in \cite{FLM17}, these measures are equivalent to so-called \emph{correlation measures}, i.e. Young measures on tensor-products of the underlying domain and phase space that represent multi-point spatial correlations. Furthermore, we require statistical solutions to satisfy an infinite number of PDEs (see definition \ref{def:statsol}) for the moments of the underlying correlation measure. Hence, a statistical solution can be interpreted as a measure-valued solution (in the sense of \cite{DipernaMajda}), augmented with information about the evolution of  all possible multi-point spatial correlations. 

Our aim in this paper was to study the well-posedness and efficient numerical approximation of statistical solutions. To this end, first, we had to characterize convergence on a weak topology on the space $L^1_t(\Prob(L^2(D;\mathbb{R}^d)))$, under an assumption of \emph{time-regularity} on the underlying measures. Convergence in this topology amounted to convergence of a very large class of observables (or statistical quantities of interest). We then proposed a notion of dissipative statistical solutions and also proved well-posedness results for them in a generic sense, namely when the initial measure is concentrated on functions sufficiently near initial data for which smooth solutions exist. This led to short-time well-posedness if the initial probability measure is concentrated on smooth functions. In two space dimensions, we proved global well-posedness for statistical solutions when the initial data is concentrated on smooth functions. Moreover, we also proved a suitable variant of weak-strong uniqueness. 

Our main contribution in this paper is the proposal of an algorithm \ref{alg:MC} to approximate statistical solutions of the Euler equations. This Monte Carlo type algorithm is a variant of the algorithms proposed recently in \cite{FKMT17,LM2015,FLMW1} and is based on an underlying spectral hyper-viscosity spatial discretization. Under verifiable hypotheses, we prove that the approximations converge in our proposed topology to a statistical solution. These hypotheses either rely on a suitable scaling (or uniform decay) for the structure function, or equivalently, on finding an inertial range (of wave numbers) on which the energy spectrum decays (uniformly in resolution). These hypotheses are very common in the extensive literature on turbulence (see \cite{Frisch} and references therein). \emph{A key novelty in this article is the rigorous proof of the fact that easily verifiable conditions on the structure functions or energy spectrum imply a rather strong form of convergence for (multi-point) statistical quantities of interest}. For instance, we observe a surprising fact that a bound on the compensated energy spectrum \eqref{eq:ces} implies that k-point statistics of interest, even for large $k$,  converge. The convergence results also provide a conditional global existence result for statistical solutions in both two and three space dimensions. 

We present results of several numerical experiments for the two-dimensional Euler equations. From the numerical experiments, we observe that
\begin{itemize}
    \item Our convergence theory is validated by all the numerical experiments. The assumptions on the structure functions and energy spectra appear to be very clearly fulfilled in practice. Moreover, the computed solutions converge to a statistical solution in suitable Wasserstein metrics on multi-point correlation marginals. In particular, all admissible observables of interest such as mean, variance, higher moments, structure functions, spectra, multi-point correlation functions, converge on increasing resolution and sample augmentation. 
    \item In clear contrast to the deterministic case where computed solutions may converge very slowly even if one can prove convergence of the underlying numerical method (see \cite{LM2019} and figure \ref{fig:svs_conv}), statistical quantities of interest seem to be better behaved and converge faster.
    \item For our numerical examples, we observe convergence of approximations even when the initial data was quite rough such as when the initial vorticity may not have definite sign (as in the flat vortex sheet) or may not even be a Radon measure (as in the fractional Brownian motion with any Hurst index $H \in (0,1)$). For such initial data, the samples are not in the Delort class and the convergence (and existence) theory for two-dimensional Euler equations is no longer valid. On the other hand, we find neat convergence to a statistical solution. 
    \item The computed statistical solutions where observed to be stable with respect to amplitude and type of perturbations of initial data. Moreover, we observed that two very different numerical methods converge to the same statistical solution for a fixed initial data. This is in contrast to the deterministic case, where the computed solutions can differ significantly \cite{LeonardiPhD}.

    \end{itemize}
Based on the above discussion, we conclude that statistical solutions are a promising solution framework for the incompressible Euler equations. In particular, there is some scope for proving well-posedness results within this class, possibly with further admissibility criteria. Moreover, numerical approximation of statistical solutions is feasible with ensemble averaging algorithms. Statistical solutions can be a suitable framework for uncertainty quantification and Bayesian inversion for the Euler equations and to encode and explain numerous computational and experimental results for turbulent fluid flows. 

There are several limitations of the current paper, which provide directions for future work. At the theoretical level, we seek to either relax the criteria on scaling of structure functions or prove it. This will pave the way for global existence results. Similarly, the weak-strong uniqueness results of this paper could be improved.

In terms of numerical approximation, the main issue with the Monte Carlo type algorithm \ref{alg:MC} is the slow convergence (in terms of number of samples). This necessitates a very high computational cost, particularly in three space dimensions. We plan to consider efficient variants such as multi-level Monte Carlo \cite{FLYM,MSS1,LMS1}, Quasi-Monte Carlo \cite{LyePhD} and deep learning algorithms \cite{LMR1}, for computing statistical solutions of the incompressible Euler equations in three space dimensions, in forthcoming papers. 

\appendix

\section{Proof of completeness of $L^1_t(\P)$, Prop. \ref{prop:completeness}} \label{app:completeness}

\begin{proof}[Proof of proposition \ref{prop:completeness}]

We wish to prove completeness of the metric space $(L^1_t(\P),d_T)$. Let $\mu^n_t$ be a Cauchy sequence, i.e.
\[
d_T(\mu^n_t,\mu^m_t) = \int_0^T W_1(\mu^n_t,\mu^m_t) \, dt \to 0, \quad \text{as }m,n \to \infty.
\]
It will suffice to prove that there exists a subsequence $\mu^{n_j}_t$ with a limit $\mu_t$, since then
\[
\limsup_{n\to \infty} d_T(\mu^{n}_t, \mu_t) \le \limsup_{n,j\to \infty} d_T(\mu^{n}_t, \mu^{n_j}_t) = 0,
\]
by the Cauchy property. 

We now choose a suitable subsequence (which we shall immediately reindex), by requiring that  $d_T(\mu^{n+1}_t,\mu^n_t) \le 2^{-n}$, for all $n\in \mathbb{N}$. It follows that $\sum_{n\in \mathbb{N}} d_T(\mu^{n+1}_t,\mu^n_t) \le 1$, and thus
\[
\sum_{n=1}^\infty W_1(\mu^{n+1}_t, \mu^{n}_t)
\]
is a convergent series in $L^1([0,T))$. In particular, this implies that
\[
\sum_{n=1}^\infty W_1(\mu^{n+1}_t, \mu^{n}_t) < \infty, \quad \text{for a.e. } t\in [0,T).
\]
For a.e. $t\in [0,T)$, we thus have that $\mu^n_t$ is Cauchy in $\P(L^p_x)$. Let $\mu_t$ denote the limit, which is defined a.e. on $[0,T)$. To see that $t \mapsto \mu_t$ is weak-$\ast$ measurable, we note that for any $F \in C_b(L^p_x)$, we have that $t \mapsto \langle \mu^n_t, F \rangle$ is measurable, and converges almost everywhere to $\langle \mu_t, F\rangle$. It follows that also $t\mapsto \langle \mu_t, F\rangle$ s measurable. By definition, this means that $t\mapsto \mu_t$ is weak-$\ast$ measurable.

To show that $\mu^n_t \to \mu_t$ in $L^1_t(\P)$, we note that for almost every $t\in [0,T)$, we have 
\begin{align*}
W_1(\mu^n_t, \mu_t) 
&= \lim_{m\to \infty} W_1(\mu^n_t, \mu^m_t) 
\\
&\le \lim_{m\to \infty} \sum_{k=n}^m W_1(\mu^{k+1}_t, \mu^{k}_t) \\
&= \sum_{k=n}^\infty W_1(\mu^{k+1}_t, \mu^{k}_t).
\end{align*}
Integrating over $[0,T)$, it follows that 
\[
d_T(\mu^n_t,\mu_t) \le \sum_{k=n}^\infty d_T(\mu^{k+1}_t, \mu^{k}_t) \le 2^{-n} \to 0, \quad \text{as }n\to \infty.
\]

\end{proof}

\section{Proof of time-compactness, Thm. \ref{thm:timecompact}} \label{app:timecompact}

\begin{lemma} \label{lem:equicont}
If $\mu_t^\Delta \in L^1_t(\P)$ is a uniformly time-regular family for which there exists a modulus of continuity $\omega(r)$, such that 
\[
S_r^2(\mu_t^\Delta, T)
\le \omega(r), \quad \forall \Delta > 0,
\]
and if there exists $M>0$, such that
\[
\int_{L^2_x} \Vert u \Vert_{L^2_x} \, d\mu_t^\Delta(u) \le M,  \quad \text{for a.e. }t\in[0,T), \quad \forall \Delta > 0, 
\]
then $t \mapsto \mu_t^\Delta$ is $L^1$-equicontinuous, in the sense that for any $\epsilon > 0$, there exists $\delta > 0$, such that if $h< \delta$, then 
\[
\int_{0}^{T-h}
W_1(\mu^\Delta_{t+h}, \mu^\Delta_t) \, dt
< \epsilon, \quad \forall \Delta > 0, 
\]
\end{lemma}

\begin{proof}
Since $\mu_t^\Delta$ is uniformly time-regular, by definition, there exists $L>0$ and $C>0$, and transfer plans $\pi^\Delta_{s,t}$ between $\mu^\Delta_s$ and $\mu^\Delta_t$, such that
\[
\int_{L^2_x\times L^2_x} 
\Vert u - v \Vert_{H^{-L}_x}
\, d\pi^\Delta_{s,t}(u,v) 
\le C |t-s|.
\] 
We note that by definition of $W_1$ as the infimum over all transport plans, we have
\begin{align*}
W_1(\mu^\Delta_{t+h},\mu^\Delta_t)
&\le 
\int_{L^2_x \times L^2_x}
\Vert u - v \Vert_{L^2_x} 
\, 
d\pi^\Delta_{t+h,t}(u,v).
\end{align*}
Integrating in time, we thus find
\begin{align*}
\int_0^{T-h}
W_1(\mu^\Delta_{t+h},\mu^\Delta_t)
\, dt
&\le 
\int_0^{T-h}
\int_{L^2_x \times L^2_x}
\Vert u - v \Vert_{L^2_x} 
\, 
d\pi^\Delta_{t+h,t}(u,v)
\, 
dt.
\end{align*}
Our goal is to show that the last term converges to $0$ as $h\to 0$, uniformly for all $\Delta>0$. Fix $\epsilon > 0$, arbitrary. We want to find $\delta > 0$, such that $h<\delta $ implies 
\begin{align} \label{eq:equicontest}
\int_0^{T-h}
\int_{L^2_x \times L^2_x}
\Vert u - v \Vert_{L^2_x} 
\, 
d\pi^\Delta_{t+h,t}(u,v)
\, 
dt
< \epsilon, \quad \forall \Delta > 0.
\end{align}
Given $u \in L^2_x$, let $u_\eta := \rho_\eta \ast u$ denote mollification with a standard mollifier $\rho_\eta(x) = \eta^{-d} \rho(\eta^{-1} x) \ge 0$, supported on a ball $B_\eta$ of radius $\eta$ around the origin. Then for all    $u, v \in L^2_x$,
\[
\Vert u - v \Vert_{L^2_x}
\le 
\Vert u - u_\eta \Vert_{L^2_x}
+ 
\Vert u_\eta - v_\eta \Vert_{L^2_x}
+
\Vert v_\eta - v \Vert_{L^2_x},
\]
together with the fact that $\mathrm{proj_1}\#\pi^\Delta_{t+h,t}(u,v) = \mu^\Delta_{t+h}(u)$ and $\mathrm{proj_2}\# \pi^{\Delta}_{t+h,t}(u,v) = \mu^\Delta_t(v)$, implies that
\begin{align*}
\int_0^{T-h}
\int_{L^2_x \times L^2_x}
&\Vert u - v \Vert_{L^2_x} 
\, 
d\pi^\Delta_{t+h,t}(u,v)
\, 
dt
\\
&\le 
\int_0^{T-h}
 \int_{L^2_x}
\Vert u - u_\eta \Vert_{L^2_x} 
\, d\mu_{t+h} \, dt
\\
&\quad
+
\int_0^{T-h}
\int_{L^2_x \times L^2_x}
\Vert u_\eta - v_\eta \Vert_{L^2_x}
\, d\pi^\Delta_{t+h,t}
\, dt
\\
&\quad
+
\int_0^{T-h}
 \int_{L^2_x}
\Vert v - v_\eta \Vert_{L^2_x} 
\, d\mu_{t} \, dt
\\
&\le 2 \int_0^T  \int_{L^2_x} \Vert u - u_\eta \Vert_{L^2_x} \, d\mu_t(u) \, dt
\\
&\quad
+
\int_0^{T-h}
\int_{L^2_x \times L^2_x}
\Vert u_\eta - v_\eta \Vert_{L^2_x}
\, d\pi^\Delta_{t+h,t}
\, dt
\end{align*}
Using the decay of the structure function, we have for all $\Delta > 0$:
\begin{align*}
\int_0^T \int_{L^2_x} \Vert u - u_\eta \Vert_{L^2_x} \, d\mu_t(u) \, dt
&\le 
C'\int_0^T S_\eta^2(\mu^\Delta_t) \, dt
= C' \sqrt{T} S_r^2(\mu^\Delta_t,T)
\le C' \sqrt{T} \omega(\eta),
\end{align*}
where the constant $C'=\Vert \rho \Vert_{L^\infty}$.
Next, we note that since $H^1_x \embedsc L^2_x \embeds H^{-L}_x$, it follows from a result originally due to J.L. Lions (cp. \cite[p. 59]{lions1961}, \cite[p. 58, Lemma 5.1]{lions1969}, or \cite[lemma 8, p.84]{Simon1986}) that for any given $\eta>0$, there exists a constant $Q(\eta)>0$ such that
\begin{align} \label{eq:Lionsest}
\Vert u \Vert_{L^2_x}
\le 
\eta^2 \Vert u \Vert_{H^1_x}
+ Q(\eta) \Vert u \Vert_{H^{-L}_x}, 
\quad \forall u \in H^{1}_x.
\end{align}
Thus, using also that $\Vert u_\eta - v_\eta \Vert_{H^1_x} \le C\eta^{-1} \Vert u - v \Vert_{L^2_x}$, we can estimate 
\begin{align*}
\Vert u_\eta - v_\eta \Vert_{L^2_x}
&\le 
\eta^2 \Vert u_\eta - v_\eta \Vert_{H^1_x}
+ Q(\eta) \Vert u_\eta - v_\eta \Vert_{H^{-L}_x}
\\
&\le 
C\eta \Vert u - v \Vert_{L^2_x}
+ Q(\eta) \Vert u - v \Vert_{H^{-L}_x}
\\
&\le 
C\eta \left(\Vert u \Vert_{L^2_x} + \Vert v \Vert_{L^2_x} \right)
+ Q(\eta) \Vert u - v \Vert_{H^{-L}_x}.
\end{align*}
Using this estimate, and the uniform time-regularity, we find
\begin{align*}
\int_0^{T-h}
\int_{L^2_x \times L^2_x}
\Vert u_\eta - v_\eta \Vert_{L^2_x}
\, d\pi^\Delta_{t+h,t}
\, dt
&\le 
2C\eta 
\int_0^{T}
\left(
\int_{L^2_x}
\Vert u \Vert_{L^2_x}
\, d\mu^\Delta_t(u)
\right)
\, dt
+ 
C Q(\eta) h.
\end{align*}
uniformly in $\Delta>0$. By assumption, the first term on the right-hand side is bounded by $2C\eta TM$.

Combining these estimates, we find
\[
\int_0^{T-h}
\int_{L^2_x \times L^2_x}
\Vert u - v \Vert_{L^2_x}
\, d\pi^\Delta_{t+h,t}
\, dt
\le 
C' \omega(\eta) + 2C\eta TM + CQ(\eta) h.
\]
Given $\epsilon > 0$, we can first choose $\eta>0$ sufficiently small so that $C'\omega(\eta) + 2\eta TM < \epsilon/2$. This fixes a value of $Q = Q(\eta)$ according to Lions' estimate \eqref{eq:Lionsest}. Next, we set $\delta := \epsilon/(2CQ)$. It then follows that for $h< \delta$, we have
\[
\int_0^{T-h}
\int_{L^2_x \times L^2_x}
\Vert u - v \Vert_{L^2_x}
\, d\pi^\Delta_{t+h,t}
\, dt
< \epsilon, \quad \forall \Delta > 0.
\]
Thus, given $\epsilon >0$, we have found $\delta>0$, such that for all $h<\delta$, we have \eqref{eq:equicontest}. This is what we set out to prove, and concludes our proof of the $L^1$-equicontinuity of $\mu^\Delta_t$.
\end{proof}

Next, we note that for any $0<a<T$, we can associate to any $\mu_t \in L^1([0,T);\P)$ a $M_a\mu_t \in L^1([0,T-a];\P)$, by defining 
\begin{align} \label{eq:timeav}
M_a \mu_t = \frac{1}{a}\int_0^a \mu_{t+h} \, dh.
\end{align}
The expression on the right-hand side is a straight-forward adaption of the Bochner integral for $L^1_t(\P)$ \cite{LanthalerThesis}. It is then easy to show that
\[
W_1(M_a \mu_t,\mu_t) 
\le 
\frac 1a \int_0^a W_1(\mu_{t+h},\mu_t) \, dh.
\]
Fix $T_1<T$. Integrating in time over $[0,T_1]$, it follows that
\begin{align} \label{eq:avgdist}
d_{T_1}(M_a \mu_t,\mu_t) 
\le 
\frac 1a \int_0^a d_{T_1}(\mu_{t+h},\mu_t) \, dh,
\end{align}
for all sufficiently small $a>0$. In particular, the following lemma is now immediate from this estimate and lemma \ref{lem:equicont}:

\begin{lemma} \label{lem:uniformapprox}
If $\mu^\Delta_t$ is a equicontinuous family, as in the conclusion of lemma \ref{lem:equicont}, then for any $\epsilon > 0$, we can find $\delta > 0$, such that for any $a<\delta$, we have
\[
d_{T_1}(M_a \mu^\Delta_t,\mu^\Delta_t) < \epsilon, \quad \forall \Delta > 0.
\]
Thus, the family $\{M_a \mu^\Delta_t \}_{\Delta>0}$ is a uniform approximation of $\{\mu^\Delta_t\}_{\Delta>0}$, in this case. 
\end{lemma}

To show that $\{\mu^\Delta_t\}_{\Delta>0}$ is relatively compact in $L^1([0,T_1];\P)$, it will thus suffice to prove that $\{M_a \mu^\Delta_t \}_{\Delta>0}$ is relatively compact, for any given (fixed) $a>0$. This is a consequence of the following lemmas and the Arzel\`a-Ascoli theorem.

\begin{lemma} \label{lem:mollequicont}
Let $T_1<T$, and fix $0<a<T-T_1$. Under the assumptions of theorem \ref{thm:timecompact}, the family $M_a\mu^\Delta_t$ is equicontinuous with respect to the $W_1$-norm, i.e. for any $\epsilon > 0$, there exists $\delta > 0$, such that
\[
W_1(M_a\mu^\Delta_t,M_a\mu^\Delta_s) < \epsilon, \quad \text{if } |s-t| < \delta.
\]
\end{lemma}

\begin{proof}
Fix $a>0$. Let $\epsilon > 0$ be given. By assumption on the uniform decay of the structure functions $S^2_r(\mu_t^\Delta,T)$, it follows from lemma \ref{lem:equicont} that the $\mu_t^\Delta$ are $L^1$-equicontinuous. Thus, we can find $\delta > 0$, such that 
\[
\int_0^{T-h} W_1(\mu_{t+h}^\Delta,\mu_t^\Delta) \, dt < a\, \epsilon, \quad \forall \Delta > 0,
\]
whenever $h< \delta$. Given $h<\delta$, we note that for any 1-Lipschitz continuous $\Phi: L^2_x \to \mathbb{R}$, we have
\begin{align*}
\int_{L^2_x} \Phi(u) d(M_a\mu^\Delta_{t+h}-M_a\mu^\Delta_{t})
&=
\frac{1}{a}
\int_0^a 
\int_{L^2_x}
\Phi(u) 
\, d(\mu_{t+h+s} - \mu_{t+s})
\, dt
\\
&\le 
\frac{1}{a}
\int_0^a 
W_1(\mu_{t+h+s}, \mu_{t+s})
\, ds
\\
&\le 
\frac{1}{a}
\int_0^{T-h} 
W_1(\mu_{h+s}, \mu_{s})
\, ds
\\
&< \epsilon.
\end{align*}
Taking the supremum over all $1$-Lipschitz continuous $\Phi$, we conclude that 
\begin{align*}
W_1(M_a\mu^\Delta_{t+h},M_a\mu^\Delta_{t}) \le \epsilon, \quad \forall \Delta>0,
\end{align*}
whenever $h< \delta$. Thus, $t\mapsto M_a\mu^\Delta_t$ is equicontinuous as an element of $C([0,T-h];\P(L^2_x))$.
\end{proof}

\begin{lemma} \label{lem:mollcompact}
Let $T_1<T$, and fix $0<a<T-T_1$. Under the assumptions of theorem \ref{thm:timecompact}, and for any $t\in [0,T_1]$, we have that
\[
\{
M_a\mu_t^\Delta \, | \, \Delta > 0
\}
\subset P(L^2_x)
\]
is relatively compact.
\end{lemma}

\begin{proof}
We note that for any $\tau\in [0,T_1]$, we have
\begin{align*}
S_r^2(M_a\mu_{\tau}^\Delta)^2
&=
\int_{L^2_x} \int_{D} \fint_{B_r(0)} |u(x+h)-u(x)|^2 \, dh \, dx \, d(M_a\mu_{\tau}^\Delta)
\\
&=
\frac{1}{a} \int_0^a \int_{L^2_x} \int_{D} \fint_{B_r(0)} |u(x+h)-u(x)|^2 \, dh \, dx \, d\mu_{\tau+s}^\Delta \, ds
\\
&\le
\frac{1}{a} \int_0^T \int_{L^2_x} \int_{D} \fint_{B_r(0)} |u(x+h)-u(x)|^2 \, dh \, dx \, d\mu_{t}^\Delta \, dt
\\
&= \frac{1}{a} S_r^2(\mu^\Delta_t, T)^2.
\end{align*}
By assumption, there exists a modulus of continuity $\omega(r)$, such that we have $S_r^2(\mu^\Delta_t, T) \le \omega(r)$, uniformly for all $\Delta$. It follows that
\[
S_r^2(M_a\mu_{\tau}^\Delta) \le \frac{1}{\sqrt{a}} \omega(r).
\]
Furthermore, it is clear from the definition of $M_a\mu^\Delta_{\tau}$ (cp. eq. \eqref{eq:timeav}), that if $\mu^\Delta_{t}(B_M) = 1$ for almost all $t\in [0,T)$, then also $M_a \mu^\Delta_{\tau}(B_M) = 1$. By theorem \ref{thm:compactnessL2}, it now follows that the family $\{M_a\mu^\Delta_{\tau} \, | \, \Delta > 0\}$ is pre-compact in $\P(L^2_x)$.
\end{proof}

We are now in a position to prove theorem \ref{thm:timecompact}.

\begin{proof}[Proof of theorem \ref{thm:timecompact}]

By lemma \ref{lem:equicont}, the family $\mu^\Delta_t$ is equicontinuous in $L^1_x(\P)$ (in the sense made precise in lemma \ref{lem:equicont}). Then by lemma \ref{lem:uniformapprox}, the mollifications $\{M_a \mu^\Delta_t\, |\, \Delta > 0\}_{a>0}$ form a uniform approximation of $\{\mu^\Delta_t\, |\, \Delta > 0\}_{a>0}$ in the sense of lemma \ref{lem:cpctapprox}. By the uniform approximation lemma \ref{lem:cpctapprox}, to show that $\{\mu^\Delta_t\, |\, \Delta > 0\}$ is precompact it therefore suffices to show that for $a>0$ fixed, the set $\{M_a \mu^\Delta_t\, |\, \Delta > 0\}$ is precompact. As was shown in lemma \ref{lem:mollequicont} and \ref{lem:mollcompact}, the mollifications $t \mapsto M_a\mu^\Delta_t$ satisfy:
\begin{itemize}
\item The mappings $t \mapsto M_a\mu^\Delta_t$ are (pointwise) equicontinuous wrt. $\Delta>0$, with values in the metric space $\P(L^2_x)$,
\item For each $t\in [0,T_1]$, the set 
\[
\{
M_a\mu_t^\Delta \, | \, \Delta > 0
\}
\subset \P(L^2_x)
\] is precompact.
\end{itemize}
Since $\P(L^2_x)$ is a complete metric space under the $W_1$-metric, it follows from the Arzel\`a-Ascoli characterization of compact subsets of $C([0,T_1];\P(L^2_x))$ that $\{M_a \mu_t^\Delta \, |\, \Delta > 0\}$ is precompact in $C([0,T_1];\P(L^2_x))$, and hence in particular that $\{M_a\mu_t^\Delta \, | \, \Delta > 0\} \subset L^1([0,T_1];\P)$ is precompact. By the uniform approximation lemma, we conclude that we also have that 
\[
\{\mu_t^\Delta \, | \, \Delta > 0\} \subset L^1([0,T_1];\P)
\]
is precompact for any $T_1<T$.

To conclude the proof, we note that the same argument also applies when time is reversed (or defining the regularization $M_a\mu_t^\Delta$ by averaging over the interval $[t-a,t]$, rather than over $[t,t+a]$). This implies that also 
\[
\{\mu_t^\Delta \, | \, \Delta > 0\} \subset L^1([T-T_1,T];\P)
\]
is precompact for any $T_1<T$. Combining these results (e.g. setting $T_1=T/2$), we finally conclude that 
\[
\{\mu_t^\Delta \, | \, \Delta > 0\} \subset L^1([0,T);\P)
\]
is precompact.
\end{proof}

\if{0}{
\section{Proof of weak/strong uniqueness, theorem \ref{thm:weakstrongunique}} \label{app:weakstrongunique}

\begin{proof}[Proof of theorem \ref{thm:weakstrongunique}]
Let $\bar{\mu}^\ast = (\bar{\mu}_1^\ast,\dots,\bar{\mu}_M^\ast) \in \Lambda(\alpha,\bar{\mu})$ define a transport plan which minimizes the transport cost, i.e.
\[
W_2(\bar{\mu},\bar{\rho})
=
\left[
\sum_{i=1}^M
\alpha_i
\int_{L^2_x}
\Vert {u} - \bar{{v}}_i \Vert^2 
\, d\bar{\mu}_i^\ast ({u})
\right]^{1/2}.
\] 
Since $\mu_t$ is a dissipative statistical solution, we have
\[
\int_{L^2_x} \Vert {u} \Vert^2 \, d\mu_t({u})
\le 
\int_{L^2_x} \Vert {u} \Vert^2 \, d\bar{\mu}({u}),
\]
and there exists a mapping $t\mapsto (\mu_{1,t},\dots,\mu_{M,t})$, such that
\[
 \int_0^T \int_{L^2_x} \int_{D} 
\left[
{u}\cdot \partial_t {\phi}
+ 
({u}\otimes {u}) : \nabla {\phi}
\right]
\, dx \, d \mu_{i,t}(u) \, dt 
=
-\int_{L^2_x} \int_{D} 
{u}\cdot {\phi}(x,0) \, dx \, d\overline{\mu}_i({u}),
\]
for all ${\phi}\in C^\infty_c$, $\div({\phi}) = 0$, and all $i=1,\dots, M$. Testing $\mu_{i,t}$ against ${v}_i(x,t) \theta(t)$ (with $\theta(t)$ to be specified later), it follows that
\begin{align*}
 \int_0^T \int_{L^2_x} \int_{D} 
&\left[
\left({u}\cdot \partial_t {v}_i\right)\theta 
+ 
\left({u}\cdot {v}_i\right) \theta'
\right]
\, dx  \, d \mu_{i,t}(u) \, dt
\\
&=
- \int_0^T \int_{L^2_x} \int_{D} 
\left(
({u}\otimes {u}) : \nabla {v}_i
\right) \theta
\, dx  \, d \mu_{i,t}(u) \, dt
\\
&\quad 
-\int_{L^2_x} \int_{D} 
{u}\cdot \bar{{v}}_i \, dx \, d\overline{\mu}_i({u}).
\end{align*}
We now fix $t\in [0,T)$ and choose $\theta \to 1_{(-\infty,t]}$. Then for a.e. $t$, it follows that
\begin{align*}
\int_0^t \int_{L^2_x} \int_{D} 
&
\left({u}\cdot \partial_t {v}_i\right)
\, dx \, d \mu_{i,s}(u) \, ds
-
\int_{L^2_x} \int_D \left({u}\cdot {v}_i\right) \, dx \, d\mu_{i,t}
\\
&=
- \int_0^t\int_{L^2_x} \int_{D} 
\left(
({u}\otimes {u}) : \nabla {v}_i
\right)
\, dx \, d \mu_{i,t}(u) \, ds
\\
&\quad 
-\int_{L^2_x} \int_{D} 
{u}\cdot \bar{{v}}_i \, dx \, d\overline{\mu}_i({u}).
\end{align*}
Re-arranging terms, and recalling that $\partial_t {v}_i = -{v}_i \cdot \nabla {v}_i - \nabla p$, we find
\begin{align*}
\int_{L^2_x} &\int_D \left({u}\cdot {v}_i\right) \, dx \, d\mu_{i,t}
\\
&=
\int_0^t\int_{L^2_x} \int_{D} 
\left({u}-{v}_i\right) \cdot \nabla {v}_i \cdot {u}
\, dx \, d \mu_{i,t}(u) \, ds
\\
&\quad 
+\int_{L^2_x} \int_{D} 
{u}\cdot \bar{{v}}_i \, dx \, d\overline{\mu}_i({u}).
\end{align*}
Noting also that $\int {w}\cdot \nabla {v}_i \cdot {v}_i \, dx = 0$ for any divergence-free ${w}\in L^2$, we can now write
\begin{gather} \label{eq:crossterm}
\begin{aligned}
\int_{L^2_x} &\int_D \left({u}\cdot {v}_i\right) \, dx \, d\mu_{i,t}
\\
&=
\int_0^t\int_{L^2_x} \int_{D} 
\left({u}-{v}_i\right) \cdot \nabla {v}_i \cdot \left({u}-{v}_i\right)
\, dx \, d \mu_{i,t}(u) \, ds
\\
&\quad 
+\int_{L^2_x} \int_{D} 
{u}\cdot \bar{{v}}_i \, dx \, d\overline{\mu}_i({u}).
\end{aligned}
\end{gather}

It follows from the fact that $\sum_{i=1}^M \alpha_i = 1$, each $\mu_{i,t}$ is a probability measure, and $\sum_{i=1}^M \alpha_i \mu_{i,t} = \mu_t$,  that
\begin{align*}
\sum_{i=1}^M \alpha_i \int_{L^2_x}\int_D \frac{1}{2}\vert {u}-{v}_i \vert^2 \, dx \, d\mu_{i,t}
&=
\int_{L^2_x}\int_D \frac{1}{2}\vert {u} \vert^2 \, dx \, d\mu_{t}
+
\sum_{i=1}^M \alpha_i \int_D \frac12 \vert {v}_i(x,t) \vert^2 \, dx
\\
&\quad - 
\sum_{i=1}^M \alpha_i \int_{L^2_x}\int_D \left( {u}\cdot {v}_i \right) \, dx \, d\mu_{i,t}
\end{align*}
Since $\mu_t$ is dissipative, the first term is estimate from above by $\int_{L^2_x} \int_D \frac 12 |{u}|^2 \, dx \, d\overline{\mu}$. Since ${v}_i$ is a strong solution, the second term is 
\[
\sum_{i=1}^M \alpha_i \int_D \frac12 \vert {v}_i(x,t) \vert^2 \, dx
= 
\sum_{i=1}^M \alpha_i \int_D \frac12 \vert \overline{{v}}_i(x) \vert^2 \, dx.
\]
Finally, using also equation \eqref{eq:crossterm}, and estimating 
\[
\int_{L^2_x} \int_{D} ({u}-{v}_i)\cdot \nabla {v}_i \cdot ({u}-{v}_i) \, dx \, d\mu_{i,t}({u})
\le 
\Vert \nabla {v}_i \Vert_{L^\infty}
\int_{L^2_x} \int_{D} |{u}-{v}_i|^2 \, dx \, d\mu_{i,t}({u}),
\]
we find
\begin{gather*}
\begin{aligned}
\sum_{i=1}^M \alpha_i &\int_{L^2_x}\int_D \frac{1}{2}\vert {u}-{v}_i(\cdot,t) \vert^2 \, dx \, d\mu_{i,t}
\\ 
&\le 
2K
\int_0^s 
\left(\sum_{i=1}^M \alpha_i \int_{L^2_x} \int_{D} \frac12|{u}-{v}_i(\cdot,s)|^2 \, dx \, d\mu_{i,s}\right) \, ds \\
&\quad 
+ \sum_{i=1}^M \alpha_i \int_{L^2_x} \int_{D} \frac12|{u}-\overline{{v}}_i|^2 \, dx \, d\overline{\mu}_{i},
\end{aligned}
\end{gather*}
for a.e. $t\in [0,T)$. From Gronwall's inequality, this implies that the 2-Wasserstein distance between $\mu_t$ and $\rho_t = \sum_{i=1}^M\alpha_i \delta_{{v}_i(t)}$ can be bounded by
\begin{align*}
W_2(\mu_t, \rho_t)^2
&\le \sum_{i=1}^M \alpha_i \int_{L^2_x} \int_D \frac 12 | {u}- {v}_i(\cdot,t) |^2 \, dx \, d\mu_{i,t}
\\
&\le e^{2Kt} \sum_{i=1}^M \alpha_i \int_{L^2_x} \int_D \frac12 |{u}-\overline{{v}}_i|^2 \, dx \, d\overline{\mu}_i \\
&= e^{2Kt} \, W_2(\bar{\mu},\bar{\rho})^2,
\end{align*}
where the last equality follows from our choice of $\alpha_i$ and $\overline{\mu}_i$. Taking square-roots, we obtain the claimed stability estimate.

\end{proof}
}\fi

\section{Proof of existence and uniqueness, Thm. \ref{thm:existuniq}} \label{app:existuniqgeneric}

We will follow the notation used in section \ref{sec:existuniq}. 
Before coming to the proof of theorem \ref{thm:existuniq}, we observe that
\[
\mathcal{G}_n 
= 
\bigcup_{\overline{v}\in \mathcal{C}} B_{r(\overline{v})/n}(\overline{v}), \quad r(\overline{v}) := e^{-C(\overline{v})T}.
\]
In particular, $\mathcal{G}_n$ is covered by the open balls $B_{r(\overline{v})/n}(\overline{v})$.
Since $L^2_x$ is a separable metric space, we can choose a countable subcovering, i.e. we can find a dense set of initial data $\{\overline{u}_i\}_{i\in \mathbb{N}} \subset \mathcal{C}$ in such a way that 
\begin{gather}\label{eq:Gndecomp}
\begin{aligned}
\mathcal{G}_n
&=
\bigcup_{i=1}^\infty B_{r(\overline{u}_i)/n}(\overline{u}_i)
\\
&= 
\left\{\overline{u} \, \Big| \, \exists i \in \mathbb{N} \text{ s.t. } \Vert \overline{u} - \overline{u}_i \Vert_{L^2_x} < \frac1n e^{-C(\overline{u}_i) T} \, \right\}.
\end{aligned}
\end{gather}
Furthermore, choosing such a countable subset $\{ \overline{u}_i^{(n)}\}_{i\in \mathbb{N}}$ for each $\mathcal{G}_n$, $n\in \mathbb{N}$, and considering finally the union $\bigcup_{n=1}^\infty \{\overline{u}_i^{(n)}\}_{i\in \mathbb{N}}$, we can in fact find a  single countable subset of $\mathcal{C}$ (not depending on $n$), such that \eqref{eq:Gndecomp} holds for all $\mathcal{G}_n$, $n\in \mathbb{N}$.

After this preliminary observation, we now come to the proof of theorem \ref{thm:existuniq}.

\begin{proof}[Proof of theorem \ref{thm:existuniq}]

Let $\bar{\mu} \in \P(L^2_x)$ be given, such that $\bar{\mu}(\mathcal{G}) = 1$. The idea of the proof is to first choose a countable set $\{\overline{u}_i\}_{i\in\mathbb{N}}$ satisfying \eqref{eq:Gndecomp} for all $n\in \mathbb{N}$, and to observe that $\bar{\mu}$ can be well approximated by atomic measures of the form $\bar{\rho} = \sum_{i=1}^\infty \alpha_i \delta_{\overline{u}_i}$. We then use the regularity of the solution $t \mapsto u_i(t)$ of the Euler equations with initial data $\overline{u}_i$ to show that a dissipative solution $\mu_t$ exists and that it is unique.

\textbf{Construction of suitable approximants:}
For any $n \in \mathbb{N}$, let us first construct a suitable approximant $\bar{\rho}^{(n)} \approx \bar{\mu}$, of the form 
\begin{align}
\bar{\rho}^{(n)} = \sum_{i=1}^\infty \alpha_i^{(n)} \delta_{\overline{u}_i}.
\end{align}
We define the coefficients $\alpha_i^{(n)}$, as well as a decomposition $\mathcal{G} = \bigcup_{i=1}^n S_i$, recursively as follows: First, denote 
\begin{align}
r_i^{(n)} := \frac 1n e^{-C(\overline{u}_i)T},
\end{align}
such that $\mathcal{G}_n = \bigcup_{i} B_{r_i^{(n)}}(\overline{u}_i)$. Then we set
\begin{gather}
\left\{
\begin{aligned}
S^{(n)}_1 &:= B_{r_1^{(n)}}(\overline{u}_1) \cap \mathcal{G}, \\
\Sigma^{(n)}_1 &:= S^{(n)}_1, \\
\alpha^{(n)}_1 &:= \bar{\mu}(S^{(n)}_1),
\end{aligned}
\right.
\end{gather}
and inductively
\begin{gather} 
\left\{
\begin{aligned}
S^{(n)}_{i+1} &:= \left[ B_{r_{i+1}^{(n)}}(\overline{u}_{i+1}) \setminus \Sigma^{(n)}_{i} \right] \cap \mathcal{G}, \\
\Sigma^{(n)}_{i+1} &:= \Sigma_i^{(n)} \cup S^{(n)}_{i+1}, \\
\alpha^{(n)}_{i+1} &:= \bar{\mu}(S^{(n)}_{i+1}).
\end{aligned}
\right.
\end{gather}
Note that $S_i^{(n)}$ can be thought of as the support of $\bar{\mu}$ close to $\overline{u}_i$, and $\Sigma_i^{(n)}$ keeps track of the set that has already been assigned to $\overline{u}_j$ for some $j< i$. By this construction, we have
\begin{gather} \label{eq:disjoint}
\begin{aligned}
S^{(n)}_i \cap S^{(n)}_j &= \emptyset, \quad \text{for }i\ne j, \\
\bigcup_{i=1}^\infty S^{(n)}_i &= \mathcal{G},
\end{aligned}
\end{gather}
and furthermore
\[
\sum_{i=1}^\infty \alpha^{(n)}_i = \sum_{i=1}^\infty \bar{\mu}(S^{(n)}_i) = \bar{\mu}\left(\mathcal{G}\right)= 1,
\]
for all $n\in \mathbb{N}$. Thus, setting $\bar{\mu}^{(n)}_i := \bar{\mu}|_{S^{(n)}_i}$, it follows that 
\[
\bar{\mu} = \sum_{i=1}^\infty \alpha_i^{(n)} \bar{\mu}^{(n)}_i.
\]
Since $\bar{\mu}^{(n)}_i$ is supported in a small neighborhood of $\overline{u}_i$, we expect $\bar{\mu}^{(n)}_i \approx \delta_{\overline{u}_i}$, and it is now natural to define 
\[
\bar{\rho}^{(n)} := \sum_{i=1}^\infty \alpha_i^{(n)} \delta_{\overline{u}_i}.
\]
A dissipative statistical solution with initial data $\bar{\rho}^{(n)}$ is given by
\[
\rho_t^{(n)} := \sum_{i=1}^\infty \alpha_i^{(n)} \delta_{u_i(t)}.
\]
In the following, we will show that $(\rho_t^{(n)})_{n=1,2,\dots}$ is Cauchy for any $t \in [0,T)$, implying the existence of a limit $\rho_t^{(n)} \weaklyto \mu_t$. Since the definition of a dissipative statistical solutions is linear in the probability measure, it is not hard to see that such a limit $\mu_t$ must itself be a dissipative statistical solution. Furthermore, we will show that the limit $\mu_t$ has $\bar{\mu}$ as initial data. Finally, we will show that this $\mu_t$ is in fact unique in the class of dissipative statistical solutions.

\textbf{The sequence $\rho_t^{(n)}$ is Cauchy:} Let $m,n\in \mathbb{N}$ be arbitrary. Note that 
\[
\mathcal{G} = \bigcup_{i=1}^\infty S_i^{(m)} = \bigcup_{i=1}^\infty S_i^{(n)},
\]
implies that upon defining $S^{(m,n)}_{i,j} := S^{(m)}_i \cap S^{(n)}_j$, we have
\[
\mathcal{G} = \bigcup_{i,j=1}^\infty S_{i,j}^{(m,n)},
\quad
S_{i}^{(m)} = \bigcup_{j=1}^\infty S_{i,j}^{(m,n)},
\quad
S_{j}^{(n)} = \bigcup_{i=1}^\infty S_{i,j}^{(m,n)}.
\]
Furthermore, the $S_{i,j}^{(m,n)}$ are pairwise disjoint by \eqref{eq:disjoint}. Thus, setting 
\[
\alpha_{i,j}^{(m,n)} := \bar{\mu}(S_{i,j}^{(m,n)}),
\]
we can define a transport plan $\pi_t \in \P(L^2_x \times L^2_x)$ from $\rho^{(m)}_t \in \P(L^2_x)$ to $\rho^{(n)}_t \in \P(L^2_x)$, by 
\[
\pi_t := \sum_{i,j} \alpha_{i,j}^{(m,n)} \delta_{u_i(t)} \otimes \delta_{u_j(t)}.
\]
Note that if $S_{i,j}^{(m,n)} \ne \emptyset$, then since $S_{i}^{(m)} \subset B_{r_i^{(m)}}(\overline{u}_i)$ and  $S_{j}^{(n)} \subset B_{r_j^{(n)}}(\overline{u}_j)$, we must in particular have
\[
B_{r_i^{(m)}}(\overline{u}_i) \cap B_{r_j^{(n)}}(\overline{u}_j) \ne \emptyset.
\]
This implies that
\[
\Vert \overline{u}_i - \overline{u}_j \Vert_{L^2_x} 
\le r_i^{(m)} + r_j^{(n)}
\le 2\max(r_i^{(m)},r_j^{(n)}).
\]
Assuming now wlog that the maximum is $\max(r_i^{(m)},r_j^{(n)})=r_i^{(m)}$, we find from the stability estimate for Lipschitz continuous solutions of the Euler equations that
\[
\Vert u_i(t) - u_j(t) \Vert_{L^2_x} 
\le 2 r_i^{(m)} e^{C(u_i)T} = \frac{2}{m} \le 2 \max\left(\frac 1m, \frac 1n\right).
\]
It follows that
\begin{align*}
W_2(\rho^{(m)}_t,\rho^{(n)}_t)^2
&\le 
\int_{L^2_x \times L^2_x} \Vert u - v \Vert_{L^2_x}^2 \, d\pi_t(u,v)
\\
&= 
\sum_{i,j} \alpha_{i,j}^{(m,n)} \Vert u_i - u_j \Vert_{L^2_x}^2 
\\
&\le 
4\max\left(\frac 1m, \frac 1n\right)^2 \sum_{i,j} \alpha_{i,j}^{(m,n)}
\\
&= 4\max\left(\frac 1m, \frac 1n\right)^2.
\end{align*}
In particular, this shows that $\rho^{(n)}_t$ is Cauchy $\P(L^2_x)$, and has a limit $\rho^{(n)}_t \to \mu_t$. 

\textbf{$\mu_t$ is a dissipative statistical solutions with initial data $\bar{\mu}$:} 
To see this, we first note that the 2-Wasserstein distance  between $\bar{\rho}^{(n)} = \rho_0^{(n)}$ and $\bar{\mu}$ can be bounded by
\begin{align*}
W_2(\bar{\mu},\bar{\rho}^{(n)})^2
&\le
\sum_{i=1}^\infty \alpha_i^{(n)} \int_{L^2_x} \Vert u - \overline{u}_i \Vert_{L^2_x}^2 \, d\bar{\mu}_i 
\\
&\le 
\sum_{i=1}^\infty \alpha_i^{(n)} \int_{L^2_x} \left[r_i^{(n)}\right]^2 \, d\bar{\mu}_i
\le 
\sum_{i=1}^\infty \alpha_i^{(n)} \frac{1}{n^2} = \frac{1}{n^2}.
\end{align*}
Thus, $\rho_0^{(n)}$ converges to $\bar{\mu}$. Since $\rho^{(n)}_t \to \mu_t$ uniformly for $t\in [0,T)$, and each $\rho^{(n)}_t$ is a dissipative statistical solution, this implies that $\mu_t$ is a statistical solution with initial data $\bar{\mu}$.

\textbf{$\mu_t$ is unique:} Let $\widetilde{\mu}_t$ be any dissipative statistical solution. We have to show that $\widetilde{\mu}_t = \mu_t$. To this end, we fix $\epsilon >0$ for the moment. Since $\rho^{(n)}_t \weaklyto \mu_t$, we have
\[
W_2(\mu_t,\widetilde{\mu}_t) \le \epsilon +  W_2(\widetilde{\mu}_t,\rho^{(n)}_t),
\]
for all sufficiently large $n$. It will thus suffice to show that $W_2(\widetilde{\mu}_t,\rho^{(n)}_t)$ can be made smaller than any $\epsilon$ for $n$ large. We now fix $n$, such that $1/n < \epsilon$. We will from now on denote $\alpha_i = \alpha_i^{(n)}$ for $i=1,2,\dots$. Define a finite convex combination of Dirac measures by
\[
\widetilde{\rho}_t = \sum_{i=1}^N \alpha_i \delta_{u_i(t)} + \alpha_0 \delta_{u_0(t)},
\]
where we set $\alpha_0= \sum_{i>N} \alpha_i$, and $u_0(t) \equiv 0$. Here, $N$ is chosen sufficiently large to ensure that 
\[
\alpha_0 = \sum_{i>N} \alpha_i <  \epsilon^2/M^2,
\]
where $M \ge 1$ is chosen such that $\bar{\mu}(B_M^c) = 0$. Such $M$ exists by assumption on $\bar{\mu}$. Note in particular, that for this choice of $\widetilde{\rho}_t$, we have
\[
W_2(\widetilde{\rho}_t, \rho_t^{(n)}) \le \epsilon,
\]
for all $t\in[0,T)$.

Define now $\bar{\mu}_i \in \Lambda(\alpha,\bar{\mu})$ by 
\begin{gather*}
\bar{\mu}_i = \bar{\mu}^{(n)}_i, \text{ for i=1,\ldots, N,}
\quad
\text{and}
\quad
\bar{\mu}_0 = 
\begin{cases}
\frac{1}{\alpha_0} \sum_{i>N} \alpha_i \bar{\mu}_i, & \alpha_0 >0, \\
0, & \alpha_0 = 0.\\
\end{cases}
\end{gather*}
Corresponding to this decomposition 
\[
\overline{\mu} = \sum_{i=0}^N \overline{\mu}_i,
\]
and by assumption on the diffusivity of $\widetilde{\mu}_t$ (cp. definition \ref{def:dissipative} on page~\pageref{def:dissipative}), there exists a decomposition $\hat{\mu}_{i,t}$ of $\widetilde{\mu}_t$, such that for each $i=0,\dots, N$:
\[
\int_0^T \int_{L^2_x} \int_D [u\cdot \partial_t \phi + (u\otimes u):\nabla \phi] \, dx\, d\hat{\mu}_{i,t}(u) \, dt
=
-\int_{L^2_x} \int_D u\cdot \phi(x,0) \, dx \, d\overline{\mu}_i(u),
\]
and
\[
\int_{L^2_x} \Vert u \Vert^2_{L^2_x} \, d\hat{\mu}_{i,t}(u) 
\le
\int_{L^2_x} \Vert u \Vert^2_{L^2_x} \, d\overline{\mu}_i(u).
\]
Repeating the argument used in the proof of weak-strong uniqueness for measure-valued solutions \cite{Brenier2011}, it follows from the fact that $u_i(t)$ is a strong solution, that 
\[
\int_{L^2_x} \Vert u - u_i(t) \Vert_{L^2_x}^2 \, d\hat{\mu}_{i,t}(u)
\le 
e^{C(\overline{u}_i)T} \int_{L^2_x} \Vert u-\overline{u}_i \Vert_{L^2_x}^2 \, d\bar{\mu}_i(u),
\]
for almost all $t\in [0,T)$, and all $i=0,\dots,N$. By our choice of $\bar{\mu}_i$, which is supported in $B_{r_i^{(n)}}(\overline{u}_i)$ and $r_i^{(n)} = e^{-C(\overline{u}_i)T}/n$, this implies 
\[
W_2(\hat{\mu}_{i,t},\delta_{u_i(t)})^2
\le 
\int_{L^2_x} \Vert u - u_i(t) \Vert_{L^2_x}^2 \, d\hat{\mu}_{i,t}(u)
\le \frac{1}{n^2} \le \epsilon^2, 
\]
for $i=1,\dots,N$. For $i=0$, we have instead
\[
W_2(\hat{\mu}_{i,t},\delta_{u_i(t)})^2
\le 
\int_{L^2_x} \Vert u - \underbrace{u_i(t)}_{=0} \Vert_{L^2_x}^2 \, d\hat{\mu}_{i,t}(u)
\le M^2,
\]
by assumption on the $L^2$-boundedness of $\overline{\mu}$.

We conclude that for almost all $t\in [0,T)$:
\begin{align*}
W_2(\widetilde{\mu}_t, \widetilde{\rho}_t)^2
&\le 
\alpha_0 W_2(\hat{\mu}_{0,t},\delta_{u_0(t)})^2 +  \sum_{i=1}^N \alpha_i W_2(\hat{\mu}_{i,t},\delta_{u_i(t)})^2
\\
&\le \frac{\epsilon^2}{M^2} M^2 + \sum_{i=1}^N \alpha_i \epsilon^2
\le 2\epsilon^2,
\end{align*}
and hence
\begin{align*}
W_2(\widetilde{\mu}_t,\mu_t)
&\le W_2(\widetilde{\mu}_t,\widetilde{\rho}_t) + W_2(\widetilde{\rho}_t,\rho_t^{(n)}) + W_2(\rho_t^{(n)},\mu_t)
\\
&\le \sqrt{2}\epsilon + \epsilon + \epsilon \le 4\epsilon.
\end{align*}
Since $\epsilon >0$ was arbitrary, it follows that $W_2(\widetilde{\mu}_t,\mu_t) = 0$ for almost all $t\in [0,T)$, i.e. $\mu_t$ is the unique dissipative statistical solution with initial data $\bar{\mu}$.
\end{proof}

\section*{Acknowledgments.}
The research of SL and SM is partially supported by the European Research Council Consolidator grant ERC-CoG 770880 COMANFLO.

\bibliographystyle{abbrv}

\bibliography{SpectralViscosity}

\begin{thebibliography}{10}

\bibitem{AM2018}
R.~Abgrall and S.~Mishra.
\newblock Uncertainty quantification for hyperbolic systems of conservation
  laws.
\newblock {\em Handbook of Numerical Analysis}, 18:507--544, 2018.

\bibitem{BardosTadmor}
C.~Bardos and E.~Tadmor.
\newblock Stability and spectral convergence of {F}ourier method for nonlinear
  problems: on the shortcomings of the 2/3 de-aliasing method.
\newblock {\em Numer. Math.}, 129(4):749--782, 2015.

\bibitem{BCG}
J.~Bell, P.~Collela, and H.~M. Glaz.
\newblock A second-order projection method for the incompressible
  {N}avier-{S}tokes equations.
\newblock {\em J. Comput. Phys.}, 85:257--283, 1989.

\bibitem{Bogachev}
V.~I. Bogachev.
\newblock {\em Measure theory. {V}ol. {I}, {II}}.
\newblock Springer-Verlag, Berlin, 2007.

\bibitem{Bonneel2011}
N.~Bonneel, M.~van~de Panne, S.~Paris, and W.~Heidrich.
\newblock Displacement interpolation using {L}agrangian mass transport.
\newblock {\em ACM Trans. Graph.}, 30(6):158:1--158:12, Dec. 2011.

\bibitem{Brenier2011}
Y.~Brenier, C.~De~Lellis, and L.~Sz{\'e}kelyhidi.
\newblock Weak-strong uniqueness for measure-valued solutions.
\newblock {\em Communications in Mathematical Physics}, 305(2):351, May 2011.

\bibitem{Chorin}
A.~Chorin.
\newblock Numerical solution of the {N}avier-{S}tokes equations.
\newblock {\em Math. Comput.}, 22:745--762, 1968.

\bibitem{DS1}
C.~{De Lellis} and L.~Sz{\'e}kelyhidi, Jr.
\newblock The {E}uler equations as a differential inclusion.
\newblock {\em Ann. of Math. (2)}, 170(3):1417--1436, 2009.

\bibitem{Delort1991}
J.-M. Delort.
\newblock Existence de nappes de toubillon en dimension deux.
\newblock {\em J. Am. Math. Soc.}, 4(3):553--586, 1991.

\bibitem{DipernaMajda}
R.~J. DiPerna and A.~J. Majda.
\newblock Oscillations and concentrations in weak solutions of the
  incompressible fluid equations.
\newblock {\em Comm. Math. Phys.}, 108(4):667--689, 1987.

\bibitem{FKMT17}
U.~S. Fjordholm, R.~Käppeli, S.~Mishra, and E.~Tadmor.
\newblock {Construction of approximate entropy measure valued solutions for
  hyperbolic systems of conservation laws}.
\newblock {\em Found. Comput. Math.}, 17(3):763–827, 2017.

\bibitem{FLM17}
U.~S. Fjordholm, S.~Lanthaler, and S.~Mishra.
\newblock {Statistical solutions of hyperbolic conservation laws {I}:
  {F}oundations}.
\newblock {\em Arch. Ration. Mech. An.}, 226(2):809–849, 2017.

\bibitem{FLYM}
U.~S. Fjordholm, K.~O. Lye, and S.~Mishra.
\newblock Numerical approximation of statistical solutions of scalar
  conservation laws.
\newblock {\em SIAM J. Numer. Anal.}, 56(5):2989--3009, 2018.

\bibitem{FLMW1}
U.~S. Fjordholm, K.~O. Lye, S.~Mishra, and F.~Weber.
\newblock Statistical solutions of hyperbolic systems of conservation law:
  {N}umerical approximation.
\newblock Preprint, available as arXiv:1906.02536v1.

\bibitem{FMT16}
U.~S. Fjordholm, S.~Mishra, and E.~Tadmor.
\newblock {On the computation of measure-valued solutions}.
\newblock {\em Acta Numer.}, 25:567–679, 2016.

\bibitem{FMW1}
U.~S. Fjordholm, S.~Mishra, and F.~Weber.
\newblock On the vanishing-viscosity limit of the incompressible
  {N}avier-{S}tokes equations.
\newblock in preparation, 2019.

\bibitem{Wass}
R.~Flamary and N.~Courty.
\newblock {POT} {P}ython {O}ptimal {T}ransport library, 2017.

\bibitem{FoiasRosaTemam2010}
C.~Foias, R.~M. Rosa, and R.~Temam.
\newblock A note on statistical solutions of the three-dimensional
  {N}avier–{S}tokes equations: The time-dependent case.
\newblock {\em Comptes Rendus Mathematique}, 348(3):235 -- 240, 2010.

\bibitem{Frisch}
U.~Frisch.
\newblock {\em Turbulence.}
\newblock Cambridge University Press, 1995.

\bibitem{Guzshu}
J.~Guzman, C.-W. Shu, and F.~A. Sequira.
\newblock {H}-{D}iv conforming and {DG} methods for the incompressible {E}uler
  equations.
\newblock {\em IMA J. Num. Anal.}, 37:1733--1771, 2017.

\bibitem{LanthalerThesis}
S.~Lanthaler.
\newblock {PhD} thesis.
\newblock in preparation, N.D.

\bibitem{LM2015}
S.~Lanthaler and S.~Mishra.
\newblock Computation of measure-valued solutions for the incompressible
  {E}uler equations.
\newblock {\em Mathematical Models and Methods in Applied Sciences},
  25(11):2043--2088, 2015.

\bibitem{LM2019}
S.~{Lanthaler} and S.~{Mishra}.
\newblock On the convergence of the spectral viscosity method for the
  incompressible {E}uler equations with rough initial data.
\newblock {\em arXiv e-prints}, page arXiv:1903.12361, Mar 2019.

\bibitem{LeonardiPhD}
F.~Leonardi.
\newblock {\em Numerical methods for ensemble based solutions to incompressible
  flow equations}.
\newblock PhD thesis, ETH Z\"urich, 2018.

\bibitem{LMS1}
F.~Leonardi, S.~Mishra, and C.~Schwab.
\newblock Numerical approximation of statistical solutions of planar,
  incompressible flows.
\newblock {\em Math. Model Meth. Appl. Sci.}, 26:2471--2523, 2016.

\bibitem{Levy1997}
D.~Levy and E.~Tadmor.
\newblock Non-oscillatory central schemes for the incompressible 2-{D} {E}uler
  equations.
\newblock {\em Mathematical Research Letters}, 4(3):321--340, 1997.

\bibitem{Levy1965}
P.~L{\'e}vy.
\newblock {\em Processus stochastiques et mouvement brownien}.
\newblock Les Grands Classiques Gauthier-Villars. [Gauthier-Villars Great
  Classics]. {\'E}ditions Jacques Gabay, Sceaux, 1992.
\newblock Followed by a note by {M}. {L}o{\`e}ve, Reprint of the second (1965)
  edition.

\bibitem{lions1961}
J.~L. Lions.
\newblock {\em Equations Differentielles Operationelles}.
\newblock Springer, 1961.

\bibitem{lions1969}
J.~L. Lions.
\newblock {\em Quelques m{\'e}thodes de r{\'e}solutions des probl{\`e}mes aux
  limites non lin{\'e}aires}.
\newblock Dunod, 1969.

\bibitem{Lions}
P.-L. Lions.
\newblock {\em Mathematical topics in fluid mechanics. {V}ol. 1}, volume~3 of
  {\em Oxford Lecture Series in Mathematics and its Applications}.
\newblock The Clarendon Press, Oxford University Press, New York, 1996.
\newblock Incompressible models, Oxford Science Publications.

\bibitem{LiuXin2001}
J.-G. Liu and Z.~Xin.
\newblock Convergence of the point vortex method for 2-{D} vortex sheet.
\newblock {\em Math. Comp.}, 70(234):595--606, 2001.

\bibitem{LiuXin1995}
J.-G. Liu and Z.~P. Xin.
\newblock Convergence of vortex methods for weak solutions to the {$2$}-{D}
  {E}uler equations with vortex sheet data.
\newblock {\em Comm. Pure Appl. Math.}, 48(6):611--628, 1995.

\bibitem{TadmorNussenzveig}
M.~C. Lopes-Filho, H.~J. Nussenzveig-Lopes, and E.~Tadmor.
\newblock Approximate solutions of the incompressible {E}uler equations with no
  concentrations.
\newblock {\em Annales de l'Institut Henri Poincare (C) Non Linear Analysis},
  17(3):371 -- 412, 2000.

\bibitem{LyePhD}
K.~O. Lye.
\newblock {PhD} thesis.
\newblock in preparation, 2019.

\bibitem{LMR1}
K.~O. Lye, S.~Mishra, and D.~Ray.
\newblock Deep learning observables in computational fluid dynamics.
\newblock Preprint, available as arXiv:1903.03040v1.

\bibitem{majda2001}
A.~J. Majda and A.~L. Bertozzi.
\newblock {\em Vorticity and Incompressible Flow}.
\newblock Cambridge Texts in Applied Mathematics. Cambridge University Press,
  2001.

\bibitem{MSS1}
S.~Mishra, C.~Schwab, and J.~{\v S}ukys.
\newblock Multi-level {M}onte {C}arlo finite volume methods for nonlinear
  systems of conservation laws in multi-dimensions.
\newblock {\em J. Comput. Phys.}, 231(8):3365--3388, 2012.

\bibitem{Pares-Pulido1}
C.~Pares-Pulido.
\newblock Numerical approximation of statistical solutions of the
  incompressible {E}uler equations with a finite difference-projection method.
\newblock in preparation, 2019.

\bibitem{Kras1}
R.Krasny.
\newblock A study of singularity formation in a vortex sheet with a point
  vortex approximation.
\newblock {\em J. Fluid Mech.}, 167:65--93, 1986.

\bibitem{Sch}
V.~Scheffer.
\newblock An inviscid flow with compact support in space-time.
\newblock {\em J. Geom. Anal}, 3:343--401, 1993.

\bibitem{scn}
A.~Schnirelman.
\newblock Weak solutions with decreasing energy of the incompressible {E}uler
  equations.
\newblock {\em Comm. Math. Phys}, 210:541--603, 2000.

\bibitem{Schochet1990}
S.~Schochet.
\newblock The rate of convergence of spectral-viscosity methods for periodic
  scalar conservation laws.
\newblock {\em SIAM J. Numer. Anal.}, 27(5):1142--1159, 1990.

\bibitem{SAF1}
Z.-S. She, E.~Aurell, and U.~Frisch.
\newblock The inviscid {B}urgers equation with initial data of {B}rownian type.
\newblock {\em Comm. Math. Phys.}, 148:623--641, 1992.

\bibitem{Simon1986}
J.~Simon.
\newblock Compact sets in the space {$L^p(O,T; B)$}.
\newblock {\em Annali di Matematica Pura ed Applicata}, 146(1):65--96, Dec
  1986.

\bibitem{Sinai}
Y.~G. Sinai.
\newblock Statistics of shocks in solutions of inviscid {B}urgers equation.
\newblock {\em Communications in Mathematical Physics}, 148(3):601--621, Sep
  1992.

\bibitem{Tadmor1989}
E.~Tadmor.
\newblock Convergence of spectral methods for nonlinear conservation laws.
\newblock {\em SIAM J. Numer. Anal.}, 26(1), 1989.

\bibitem{Tadmor2004}
E.~Tadmor.
\newblock {B}urgers' equation with vanishing hyper-viscosity.
\newblock {\em Commun. Math. Sci.}, 2(2):317--324, 2004.

\bibitem{VF1}
M.~I. Vishik and A.~V. Fursikov.
\newblock Homogeneous statistical solutions of the {N}avier-{S}tokes system.
\newblock {\em Uspehi Mat. Nauk.}, 32:179--180, 1977.

\end{thebibliography}

\end{document}

\end{document}